\documentclass[10pt]{amsart}
\usepackage[utf8]{inputenc}
\usepackage{amsmath}
\usepackage{amsfonts}
\usepackage{amssymb}
\usepackage{amsthm}
\usepackage{enumitem}
\usepackage{graphicx}
\usepackage{blindtext}
\usepackage{tgpagella}

\title[Stable/unstable Lagrangian laminations]{Towards a higher-dimensional construction of stable/unstable Lagrangian laminations}
\author{Sangjin Lee}
\date{}

\graphicspath{ {images/} }

\newtheorem{theorem}{Theorem}

\newtheorem{lemma}[theorem]{Lemma}

\theoremstyle{definition}
\newtheorem{definition}[theorem]{Definition}

\theoremstyle{definition}
\newtheorem{exmp}[theorem]{Example}%[section]

\theoremstyle{definition}
\newtheorem{remark}[theorem]{Remark}
\newtheorem*{remark*}{Remark}

\makeatletter
\@addtoreset{theorem}{section}
\makeatother

\usepackage{tikz}
\usepackage[framemethod=tikz]{mdframed}

\mdfapptodefinestyle{exampledefault}{topline=false,bottomline=false}
\mdfdefinestyle{commentstyle}{%
	linecolor=red,%
	innertopmargin=0.75em,%
	innerbottommargin=0.75em,%
	topline=false,%
	bottomline=false%
}

\setlist{noitemsep}
\begin{document}
\maketitle
\begin{abstract} 
	We generalize some properties of surface automorphisms of pseudo-Anosov type. 
	First, we generalize the Penner construction of a pseudo-Anosov homeomorphism and show that a symplectic automorphism which is constructed by our generalized Penner construction has an invariant Lagrangian branched submanifold and an invariant Lagrangian lamination, which are higher-dimensional generalizations of a train track and a geodesic lamination in the surface case.	
	Moreover, if a pair consisting of a symplectic automorphism $\psi$ and a Lagrangian branched surface $\mathcal{B}_{\psi}$ satisfies some assumptions, we prove that there is an invariant Lagrangian lamination $\mathcal{L}$ which is a higher-dimensional generalization of a geodesic lamination.
\end{abstract}

\section{Introduction}

By the Nielsen-Thurston classification of surface diffeomorphisms, an automorphism $\psi :S \stackrel{\sim}{\to} S$ of a compact oriented surface $S$ is of one of three types: periodic, reducible or pseudo-Anosov \cite{MR964685}, \cite{MR956596}.
A generic element of the mapping class group of $S$ is of pseudo-Anosov type.

Let us assume that $\psi$ is of pseudo-Anosov type.
For any closed curve $C \subset S$, it is known that there is a sequence $\{L_m\}_{m \in \mathbb{N}}$ of closed geodesics such that $L_m$ is isotopic to $\psi^m(C)$ for all $m \in \mathbb{N}$, and $\{L_m\}_{m \in \mathbb{N}}$, as a sequence of closed subsets, converges to a closed subset $\mathcal{L}$. 
Moreover, $\mathcal{L}$ is a geodesic lamination. 
The definitions of a lamination, a geodesic lamination and a Lagrangian lamination are the following: 
\begin{definition}
	$\mbox{}$
	\begin{enumerate}
		\item  
		A {\em $k$-dimensional lamination} on an $n$-dimensional manifold $M$ is a decomposition of a closed subset of $M$ into $k$-dimensional submanifolds called {\em leaves} so that $M$ is covered by charts of the form $I^k \times I^{n-k}$ where a leaf passing through a chart is a slice of the form $I^k \times \{pt\}$.
		\item 
		A $1$-dimensional lamination $\mathcal{L}$ on a Riemannian 2-manifold $(S,g)$ is a {\em geodesic lamination} if every leaf of $\mathcal{L}$ is geodesic.
		\item
		A $n$-dimensional lamination $\mathcal{L}$ on a symplectic manifold $(M^{2n},\omega)$ is a {\em Lagrangian lamination} if every leaf of $\mathcal{L}$ is a Lagrangian submanifold.  
	\end{enumerate}
\end{definition}
For more details, we refer the reader to \cite[Chapter 15]{MR2850125}.

In \cite{MR3289326}, Dimitrov, Haiden, Katzarkov, and Kontsevich defined the notion of a {\em pseudo-Anosov functor} of a category.
A pseudo-Anosov map $\psi$ on a compact oriented surface $S$ induces a functor, also called $\psi$, on the derived Fukaya category $D^{\pi}Fuk(S,\omega)$, where $\omega$ is an area form of $S$.
In \cite{MR3289326}, the authors showed that $\psi$ is a pseudo-Anosov functor.

In \cite[Section 4]{MR3289326}, the authors listed a number of open questions. 
One of them is to find a symplectic automorphism $\psi$ on a symplectic manifold $M$ of dimension greater than 2 which has invariant transversal stable/unstable Lagrangian measured foliations. 
A slightly weaker version of the question is to define a symplectic automorphism $\psi$ with invariant stable/unstable Lagrangian laminations.

The goal of the present paper is to prove Theorems \ref{branched surface thm}--\ref{thm Lagrangian floer homology}, which answer the latter question.

\begin{theorem}
\label{branched surface thm}
Let $M$ be a symplectic manifold and let $\psi : M \stackrel{\sim}{\to} M$ be a symplectic automorphism of generalized Penner type. Then, there exists a Lagrangian branched submanifold $\mathcal{B}_{\psi}$ such that if $L$ is a Lagrangian submanifold which is carried (resp.\ weakly carried) by $\mathcal{B}_{\psi}$, then $\psi^m(L)$ is carried (resp.\ weakly carried) by $\mathcal{B}_{\psi}$ for all $m \in \mathbb{N}$. 
\end{theorem} 
In Sections 2 and 3, we will explain the terminology that appears in the statement of Theorem \ref{branched surface thm}, i.e., a symplectic automorphism of generalized Penner type, a Lagrangian branched submanifold, and the notion of ``carried by''.

We would like to remark that Theorem \ref{branched surface thm} is for $\psi$ of generalized Penner type.
However, there would be a generalized version of Theorem \ref{branched surface thm}, which we do not prove in the current paper.

\begin{theorem}
\label{lamination thm}
Let $M$ be a symplectic manifold and let $\psi: M \stackrel{\sim}{\to} M$ be a symplectic automorphism of generalized Penner type.
Then, there is a Lagrangian lamination $\mathcal{L}$ such that
if $L$ is a Lagrangian submanifold of $M$ which is carried by $\mathcal{B}_{\psi}$, then there is a Lagrangian submanifold $L_m$ for all $m \in \mathbb{N}$, which is Hamiltonian isotopic to $\psi^m(L)$ and converges to $\mathcal{L}$ as closed sets as $m \to \infty$.
\end{theorem}
We will also prove the following generalization of Theorem \ref{lamination thm}

\begin{theorem}
	\label{generalized theorem}
	Let $\psi:M \stackrel{\sim}{\to} M$ be a symplectic automorphism and let $\mathcal{B}_{\psi}$ be a Lagrangian branched submanifold such that $\psi(\mathcal{B}_{\psi})$ is carried by $\mathcal{B}_{\psi}$.
	Moreover, if the associated branched manifold $\mathcal{B}_{\psi}$ admits a decomposition into singular and regular disks, then there is a Lagrangian lamination $\mathcal{L}$ such that
	if $L$ is a Lagrangian submanifold of $M$ which is carried by $\mathcal{B}_{\psi}$, then there is a Lagrangian submanifold $L_m$ for all $m \in \mathbb{N}$, which is Hamiltonian isotopic to $\psi^m(L)$ and converges to $\mathcal{L}$ as closed sets as $m \to \infty$.
\end{theorem}
The associated branched manifold and singular/regular disks will be defined in Sections 3 and 4.

\begin{theorem}
	\label{thm Lagrangian floer homology}
	Let $M$ be a plumbing space of Penner type and let $\eta : M \stackrel{\sim}{\to} M$ be the involution associated to $M$. 
	Let assume that a transversal pair $L_1, L_2 \subset M$ of Lagrangian submanifolds satisfies the following:
	\begin{enumerate}
		\item $\eta(L_i) = L_i$ for $i = 0, 1$.
		\item Let $\tilde{L}_i = L_i \cap M_i$. Then, $\tilde{L}_i$ is a Lagrangian submanifold of $\tilde{M}$ such that $\tilde{L}_0$ and $\tilde{L}_1$ are not isotopic to each other.
		\item $L_0 \cap L_1 = \tilde{L}_0 \cap \tilde{L}_1$,
		\item $L_0$ and $L_1$ are not isotopic to each other. 
	\end{enumerate} 
	Then, 
	\begin{gather*}
	\operatorname{dim} HF^0(L_1,L_2) + \operatorname{dim}HF^1(L_1,L_2) = i(\tilde{L}_1, \tilde{L}_2),
	\end{gather*}
	where $HF^k(L_1,L_2)$ denotes $\mathbb{Z}/2$--graded Lagrangian Floer homology over the Novikov ring of characteristic 2 and $i(\tilde{L}_1,\tilde{L}_2)$ denotes the geometric intersection number of $\tilde{L}_1$ and $\tilde{L}_2$ in the fixed surface $\tilde{M}$.
\end{theorem}
In Section \ref{section pseudo-Anosov functors}, we will explain the terminology that appears in the statement of Theorem \ref{thm Lagrangian floer homology}, i.e., a plumbing space $M$ of Penner type, the involution $\eta$ associated to $M$, and the fixed surface $\tilde{M}$ of $M$. 

This paper consists of 5 sections.
In Section 2, we review plumbing spaces and generalized Dehn twists. 
We will prove Theorem \ref{branched surface thm} in Section 3 and Theorems \ref{lamination thm} and \ref{generalized theorem} in Section 4.
In Section \ref{section pseudo-Anosov functors}, we will prove Theorem \ref{thm Lagrangian floer homology}.

\section{Preliminaries}

In this section, we will review plumbings of cotangent bundles and generalized Dehn twists, partly to establish notation.

\subsection{Plumbing spaces}

Let $\alpha$ and $\beta$ be oriented spheres $S^n$. 
We describe how to plumb $T^*\alpha$ and $T^*\beta$ at $p \in \alpha$ and $q \in \beta$. 
Let $U \subset \alpha$ and $V \subset \beta$ be small disk neighborhoods of $p$ and $q$.
Then, we identify $T^*U$ and $T^*V$ so that the base $U$ (resp.\ $V$) of $T^*U$ (resp.\ $T^*V$) is identified with a fiber of $T^*V$ (resp.\ $T^*U$).  

To do this rigorously, we fix coordinate charts $\psi_1: U \to \mathbb{R}^n$ and $\psi_2:V \to \mathbb{R}^n$.
Then, we obtain a compositions of symplectomorphisms
\begin{align*}
T^*U \xrightarrow{({\psi}_1^*)^{-1}} T^*\mathbb{R}^n \simeq \mathbb{R}^{2n} \xrightarrow{f} \mathbb{R}^{2n} \simeq T^*\mathbb{R}^n \xrightarrow{\psi_2^*} T^*V,
\end{align*}
where $f(x_1, \cdots, x_n, y_1, \cdots, y_n) = (y_1, \cdots, y_n, -x_1, \cdots, -x_n)$. 

A plumbing space $P(\alpha, \beta)$ of $T^*\alpha$ and $T^*\beta$ is defined by %$\bigslant{T^*\alpha \cup T^*\beta}{\sim}$ 
$T^*\alpha \sqcup T^*\beta / \sim$, where $x \sim (\psi^*_2 \circ f \circ \psi^{*-1}_1)(x)$ for all $x \in T^*U$.
Since $\psi^*_2 \circ f \circ \psi^{*-1}_1$ is a symplectomorphism, $P(\alpha,\beta)$ has a natural symplectic structure induced by the standard symplectic structures of cotangent bundles.

Since the plumbing procedure is a local procedure, we can plumb a finite collection of cotangent bundles of the same dimension at finitely many points. 
For convenience, we plumb cotangent bundles of oriented manifolds.

Note that we can replace $f$ by 
$$g(x_1,\cdots,x_n,y_1,\cdots,y_n) = (-y_1, y_2, \cdots,y_n,x_1, -x_2, \cdots, -x_n).$$
If we plumb $T^*\alpha$ and $T^*\beta$ at one point using $g$, this plumbing space is symplectomorphic to the previous plumbing space $P(\alpha,\beta)$, which is plumbed using $f$. 
However, if we plumb at more than one point, then by replacing $f$ with $g$ at a plumbing point, the plumbing space will change.

\begin{definition}
	\label{def plumbing space}
	Let $\alpha_1, \cdots, \alpha_m$ be oriented manifolds of dimension $n$.
	\begin{enumerate}
		\item A {\em plumbing data} is a collection of pairs of non-negative integers $(a_{i,j}, b_{i,j})$ for all $1 \leq i \leq j \leq m$ and collections of distinct points 
		\begin{gather*}
		\{p^{i,j}_k \in \alpha_i \hspace{0.2em} | \hspace{0.2em} 1 \leq i \leq j \leq m, \hspace{0.2em} 1 \leq k \leq a_{i,j} + b_{i,j} \} \hspace{0.5em} \text{and} \\
		\{ q^{i,j}_k \in \alpha_j \hspace{0.2em} | \hspace{0.2em} 1 \leq i \leq j \leq m, \hspace{0.2em} 1 \leq k \leq a_{i,j} + b_{i,j}  \}.
		\end{gather*}
		\item A {\em plumbing space} $P(\alpha_1, \cdots, \alpha_m)$, with the given plumbing data, is given by 
		$$P(\alpha_1, \cdots, \alpha_m) = T^*\alpha_1 \sqcup \cdots \sqcup T^*\alpha_m / \sim,$$
		where the equivalence relation $\sim$ is defined as follows:
		First, choose small disk neighborhoods $U^{i,j}_k \subset \alpha_i$ of $p^{i,j}_k$ and $V^{i,j}_k \subset \alpha_j$ of $q^{i,j}_k$ and orientation-preserving coordinate charts $\psi^{i,j}_k : U^{i,j}_k \stackrel{\sim}{\to} \mathbb{R}^n$ and $\phi^{i,j}_k:V^{i,j}_k \stackrel{\sim}{\to} \mathbb{R}^n$. 
		Then for all $ x \in T^*U^{i,j}_k$, 
		\begin{gather*}
		x \sim (\phi^{i,j*}_k \circ f \circ (\psi^{i,j*}_k)^{-1})(x) \hspace{0.5em}  \text{if} \hspace{0.5em} 1 \leq k \leq a_{i,j}, \\
		x \sim (\phi^{i,j*}_k \circ g \circ (\psi^{i,j*}_k)^{-1})(x) \hspace{0.5em}  \text{if} \hspace{0.5em} a_{i,j} +1 \leq k \leq a_{i,j} + b_{i,j}.
		\end{gather*}
		\item A {\em plumbing point} is an identified point $p^{i,j}_k \sim q^{i,j}_k \in P(\alpha_1, \cdots, \alpha_m).$ 
	\end{enumerate}
\end{definition}
Figure \ref{figure examples of plumbing space} is examples of plumbing spaces.

If $\alpha_i$ is of dimension $n \geq 2$, then specific choices of plumbing points do not change the symplectic topology of $P(\alpha_1, \cdots, \alpha_m)$.

\begin{figure}[h]
	\centering
	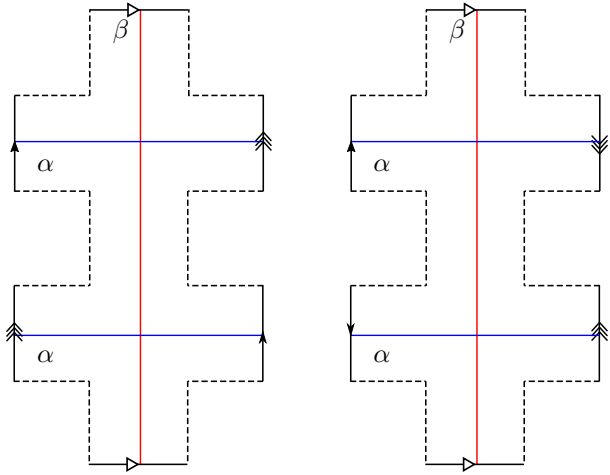
	\caption{$P(\alpha \simeq S^1, \beta \simeq S^1)$ with plumbing data $(2,0)$ (left) and $(1,1)$ (right).}
	\label{figure examples of plumbing space}
\end{figure}

\subsection{Generalized Dehn twist}

Let
\begin{gather*}
T^*S^n = \{(u;v) \in \mathbb{R}^{n+1} \times \mathbb{R}^{n+1} \mid \|u\| = 1, \left\langle u,v\right\rangle = 0 \}, \\
S^n = \{(u;0) \in T^*S^n \},
%T^*_{\epsilon}S^n = \{(u,v) \in T^*S^n \mid |v| < \epsilon \}.
\end{gather*}
where $(u;v) \in \mathbb{R}^{n+1} \times \mathbb{R}^{n+1}$ and $<u,v>$ is the standard inner product of $u$ and $v$ in $\mathbb{R}^{n+1}$.
Moreover, let $0_k$ be the origin in $\mathbb{R}^k$.

We fix a Hamiltonian function $\mu(u;v) = \|v\|$ on $T^*S^n \setminus S^n$.
Then, $\mu$ induces a circle action on $T^*S^n \setminus S^n$ given by 
$$ \sigma(e^{it})(u;v) = \big(\cos (t) u + \sin (t) \frac{v}{\|v\|}; \cos (t) v - \sin (t) \|v\| u  \big).$$
Let $r : [0,\infty) \to \mathbb{R}$ be a smooth decreasing function such that $r(0) = \pi$ and $r(t) = 0$ for all $t \geq \epsilon$ for a small positive number $\epsilon$.
If $\omega_0$ is the standard symplectic form of $T^*S^n$,
we define a symplectic automorphism $\tau : (T^*S^n, \omega_0) \stackrel{\sim}{\to} (T^*S^n, \omega_0) $ as follows
\begin{align}
\label{eqn generalized dehn twist definition}
\tau(u;v) = \left\{\begin{matrix}
\sigma(e^{i r(\mu{(u;v)})})(u;v) \hspace{5pt} &\text{if} \hspace{5pt} v \neq 0_{n+1},\\ 
(-u;0_{n+1})  \hspace{5pt} &\text{if} \hspace{5pt} v = 0_{n+1}.
\end{matrix}\right.
\end{align}  

Let $(M^{2n}, \omega)$ be a symplectic manifold and let $L \simeq S^n$ be a Lagrangian sphere in $M$. 
By the Lagrangian neighborhood theorem \cite{MR0286137}, there is a neighborhood $N(L) \supset L$ and a symplectomorphism $\phi : T^*S^n \stackrel{\sim}{\rightarrow} N(L)$.
We define a generalized Dehn twist $\tau_L$ along $L$ as follows:
\begin{align}
\label{eqn definition of generalized Dehn along L}
\tau_L (x) = \left\{\begin{matrix}
(\phi \circ \tau \circ \phi^{-1}) (x) \hspace{5pt} &\text{if} \hspace{5pt} x \in N(L), \\
x \hspace{5pt} &\text{if} \hspace{5pt} x \notin N(L).
\end{matrix}\right.
\end{align} 
Note that the support of $\tau_L$ is contained in $N(L)$.
From now on, a generalized Dehn twist will just be called a Dehn twist. 

\begin{remark}
	\label{rmk specific dehn twists}
	In this paper, we will use two specific Dehn twists $\tau, \tilde{\tau} : T^*S^n \stackrel{\sim}{\to} T^*S^n$ which are defined by Equation \eqref{eqn generalized dehn twist definition} and two functions $r, \tilde{r} : [0,\infty) \to \mathbb{R}$.
	The function $r$ (resp.\ $\tilde{r}$) defining $\tau$ (resp.\ $\tilde{\tau}$) satisfies the above conditions in addition to $r(t) = \pi$ for all $t \leq \tfrac{\epsilon}{2}$ (resp.\ $\tilde{r}'(0) < 0)$.
	Two Dehn twists $\tau$ and $\tilde{\tau}$ are equivalent in the sense that $\tau \circ \tilde{\tau}^{-1}$ is a Hamiltonian isotopy.
\end{remark}

Dehn twists have been studied extensively by Seidel. 
For example, Seidel \cite{MR1743463} proved the following theorem.
\begin{theorem}
	\label{lagrangian surgery theorem}
	Let $\alpha$ be a Lagrangian sphere and $\beta$ be a Lagrangian submanifold of a symplectic manifold $M$. 
	If $\alpha$ and $\beta$ intersect transversally at only one point, $\beta \# \alpha $ is Lagrangian isotopic to $\tau_{\alpha}(\beta)$ where $\beta \# \alpha$ is a Lagrangian surgery of $\alpha$ and $\beta$.
\end{theorem}

We prove Theorem \ref{lagrangian surgery theorem} in the special case that $\beta$ is also a sphere and $M = P(\alpha, \beta)$, as an illustration of the ``spinning'' procedure.
To define ``spinning'', we use the following notation.
Let $y \in S^{n-1} \subset \mathbb{R}^{n}$. 
Then, 
\begin{gather*}
\psi_y : T^*S^1 \simeq S^1 \times \mathbb{R}  \to T^*S^n, \\
(\theta, t) \mapsto (\cos \theta (0_n,1) + \sin \theta (y,0); t\cos \theta (y,0) - t \sin \theta (0_n,1) )
\end{gather*}
is a symplectic embedding. 
Let $W_y$ be the embedded symplectic surface $\psi_y(T^*S^1)$. 
\begin{definition}
	\label{def spinning}
	Given a curve $C$ in $T^*S^1$, its {\em spun image} $S(C)$ is $\cup_{y \in S^{n-1}} \psi_y(C)$. 
\end{definition}

%For a symplectic manifold $M^{2n}$, let $\alpha$ be a Lagrangian sphere. 
%Let $N(\alpha)$ be a neighborhood of $\alpha$ with an identification $\phi : T^*S^n \stackrel{\sim}{\to} N(\alpha)$. 
%For any $p \in \alpha$, without loss of generality, we can assume that $\phi(0_n,1;0_{n+1}) = p$.
%We will also consider $\phi(S(C))$ in this paper.

\begin{proof} [Proof of Theorem \ref{lagrangian surgery theorem}]
	We use $T^*\alpha$ and $T^*\beta$ to indicate neighborhoods of $\alpha$ and $\beta$ inside $M = P(\alpha, \beta)$. 
	Let $p$ be the intersection point of $\alpha$ and $\beta$.
	Then, $T^*_p\alpha = \beta \cap T^*\alpha$. 
	The closure of $T^*_p\alpha$ is denoted by $D_p^-$; we use $D$ to indicate that this is a disk and the subscript $p$ means that $p$ is the center of $D_p^-$. 
	The meaning of the negative sign in $D_p^-$ will be explained in the next section. 
	Since $\tau_{\alpha}$ is supported on 
	$T^*\alpha$, 
	$$\tau_{\alpha}(\beta) = \tau_{\alpha}(\beta \cap T^*\alpha) \cup \tau_{\alpha}(\beta \setminus T^*\alpha) = \tau_{\alpha}(D_p^-) \cup (\beta \setminus T^*\alpha).$$

	There exists 
	$\phi : T^*S^n \stackrel{\sim}{\to} T^*\alpha$ 
	such that $\tau_{\alpha} = \phi \circ \tau \circ \phi^{-1}$. 
	Without loss of generality, $\phi(0_n,1;0_{n+1}) = p$ and 
	$$D_p^- = \phi(\{(0_n,1;ty,0) \hspace{0.2em} | \hspace{0.2em} t \in \mathbb{R}, \hspace{0.2em} y \in S^{n-1} \subset \mathbb{R}^n \}).$$
	Then,
	\begin{align*}
	(\phi \circ \tau_{\alpha} \circ \phi^{-1}) (D_p^-)& = (\phi \circ \tau)(\{(0_n,1;ty,0) \hspace{0.2em} | \hspace{0.2em} t \in \mathbb{R}, \hspace{0.2em} y \in S^{n-1} \subset \mathbb{R}^n \}) \\
	 &= \cup_{y \in S^{n-1}}\phi( \{\tau(0_n,1;ty,0) \hspace{0.2em} | \hspace{0.2em} t \in \mathbb{R} \}).
	\end{align*} 
	Thus, $\tau_{\alpha}(D_p^+)$ is given by spinning with respect to $p$ and $\phi$. 	
	Similarly, we can construct a Lagrangian isotopy connecting $\tau_{\alpha}(\beta)$ and $\beta \# \alpha$ by spinning. 
	This completes the proof.  
\end{proof}

\section{Lagrangian branched submanifolds}

In Section 3.1, we will define Lagrangian branched submanifolds.
In Section 3.2, we will introduce a construction of a fibered neighborhood of a Lagrangian branched submanifolds. 
In Section 3.3, we will defined the notion of ``carried by'' by using a fibered neighborhood.
In Section 3.4, we will introduce the generalized Penner construction.
Finally, we will give a proof of Theorem \ref{branched surface thm} in Section 3.5.

\subsection{Lagrangian branched submanifolds.}

Thurston \cite{MR1435975} used train tracks, which are 1-dimensional branched submanifolds of surfaces, and defined the notion of ``carried by a train track".
In this subsection, we generalize train tracks.  

The generalization of a train track is an $n$-dimensional branched submanifold of a $2n$-dimensional manifold. 
We define the $n$-dimensional branched submanifolds with local models, as Floyd and Oertel defined a branched surface in a 3-dimensional manifold in \cite{MR721458}, \cite{MR746535}.
For our definition, we need a smooth function $s: \mathbb{R} \to \mathbb{R}$ such that $s(t) = 0$ if $t \leq 0$ and $s(t)>0$ if $t>0$. 

\begin{definition}
	\label{def of branched submfd}
	Let $M^{2n}$ be a smooth manifold.
	\begin{enumerate}
		\item A subset $\mathcal{B} \subset M$ is an {\em $n$-dimensional branched submanifold} if for every $p \in \mathcal{B}$, there exists a chart $\phi_p:U_p \stackrel{\sim}{\to} \mathbb{R}^{2n}$ about $p$ such that $\phi_p(p) = 0$ and $\phi_p(\mathcal{B} \cap U_p)$ is a union of submanifolds $L_0, L_1, \cdots, L_k$ for some $k \in \{0, \cdots, n\}$, where 
		\begin{align*}
		L_i := \{(x_1, \cdots, x_n, s(x_1), s(x_2), \cdots, s(x_i), 0, \cdots, 0) \in \mathbb{R}^{2n} \hspace{0.2em} | \hspace{0.2em} x_j \in \mathbb{R}\}. 
		\end{align*}
		\item A {\em sector} of $\mathcal{B}$ is a connected component of the set of all points in $\mathcal{B}$ that are locally modeled by $L_0$, i.e., $k=0$.
		\item A {\em branch locus $Locus(\mathcal{B})$} of $\mathcal{B}$ is the complement of all the sectors. 
		\item Let $(M^{2n}, \omega)$ be a symplectic manifold.
		A subset $\mathcal{B} \subset M$ is a {\em Lagrangian branched submanifold} if for every $p \in \mathcal{B}$, there exists a Darboux chart $\phi_p:(U_p,\omega|_{U_p}) \stackrel{\sim}{\to} (\mathbb{R}^{2n},\omega_0)$ about $p$, satisfying the conditions of an $n$-dimensional branched submanifold. 
	\end{enumerate}
\end{definition}

\begin{remark}
	\label{rmk some facts for lagrangian branched submanifold}
	$\mbox{}$
	\begin{enumerate}
		\item At every point $p$ of a branched submanifold $\mathcal{B}$, the tangent plane $T_p\mathcal{B}$ is well-defined.
		Moreover, if $\mathcal{B}$ is Lagrangian, then $T_p\mathcal{B}$ is a Lagrangian subspace of $T_pM$.  
		\item A point on the branch locus is (a smooth version of) an arboreal singularity in the sense of Nadler \cite{MR3626601}. 
	\end{enumerate}
\end{remark}

\begin{exmp}
	\mbox{}
	\label{ex Lagrangian branched submanifold}
	\begin{enumerate}
		\item Every train track of a surface equipped with an area form is a Lagrangian branched submanifold.
		\item Let $(M,\omega)$ be a symplectic manifold and let $L_1$ and $L_2$ be two Lagrangian submanifold of $M$ such that 
		$$L_1 \pitchfork L_2, \hspace{0.2em} L_1 \cap L_2 = \{p\}.$$
		The Lagrangian surgery of $L_1$ and $L_2$ at $p$ will be denoted by $L_2 \#_p L_1$.
		Then, $L_2 \#_p L_1 \cup L_1$ and $L_2 \#_p L_1 \cup L_2$ are examples of Lagrangian branched submanifold. 
	\end{enumerate}
\end{exmp}

In Section 3.3, we will define the notion of {\em ``carried by''} which appears in Theorems \ref{branched surface thm} - \ref{generalized theorem}.
In order to define the notion of carried by, we will construct a fibered neighborhood first in Section 3.2.
\vskip0.2in

\subsection{Construction of fibered neighborhoods.}
Let $\mathcal{B}$ be a Lagrangian branched submanifold. 
A fibered neighborhood $N(\mathcal{B})$ of $\mathcal{B}$ is, roughly speaking, a codimension zero compact submanifold with boundary and corners of $M$, which is foliated by Lagrangian closed disks which are called {\em fibers}. 

\begin{definition}
	\label{def fibered ngbd}
	A {\em fibered neighborhood of $\mathcal{B}$} is a union $\cup_{p \in \mathcal{B}} F_p$, where $\{F_p \hspace{0.2em} | \hspace{0.2em} p \in \mathcal{B} \}$ is a family of Lagrangian disks satisfying 
	\begin{enumerate}
		\item for any $p \in \mathcal{B}$, $F_p \pitchfork \mathcal{B}$,
		\item for any $p, q \in \mathcal{B}$, either $F_p = F_q$ or $F_p \cap F_q = \varnothing$,
		\item there exists a closed neighborhood $U \subset \mathcal{B}$ of $Locus(\mathcal{B})$, such that $\{F_p \hspace{0.2em} | \hspace{0.2em} p \in U \}$ is a smooth family over each local sheet $L_i \cap U$, 
		\item for each sector $S$ of $\mathcal{B}$, $\{F_p \hspace{0.2em} | \hspace{0.2em} p \in S \setminus U \}$ is a smooth family,
		\item if $p \in S \cap \partial U$ where $S$ is a sector of $\mathcal{B}$, then, for any sequence $\{q_n \in S \setminus U \}_{n \in \mathbb{N}}$, 
		\begin{gather*}
		\lim_{n \to \infty} F_{q_n} \text{  is a Lagrangian disk such that  } \lim_{n \to \infty} F_{q_n} \subset \mathring{F}_p = F_p \setminus \partial F_p.
		\end{gather*}
 	\end{enumerate} 
\end{definition}

We will now give a specific construction of a fibered neighborhood $N(\mathcal{B})$. 

\begin{remark}
	\label{rmk natural embedding}
	By the Lagrangian neighborhood theorem \cite{MR0286137}, for any Lagrangian submanifold $L$ of $M$, there exists a small neighborhood $\mathcal{N}(L)$ of the zero section of $T^*L$ such that a symplectic embedding $i_L : \mathcal{N}(L) \hookrightarrow M$ is defined on $\mathcal{N}(L)$. 
	Without loss of generality, we assume that $\mathcal{N}(L)$ is a closed neighborhood. 
	Than, $\mathcal{N}(L)$ is foliated by closed Lagrangian disks $\mathcal{N}(L) \cap T^*_pL$. 
\end{remark}
\vskip0.2in

\noindent{\em Fibration over $L(\ell)$.}
First, we will construct fibers near the branch locus.
For each connected component $\ell$ of $Locus(\mathcal{B})$, we choose a small closed Lagrangian neighborhood $L(\ell)$ of $\ell$. 
Then, by Remark \ref{rmk natural embedding}, there exists a symplectic embedding 
$$ i_{L(\ell)} : \mathcal{N}(L(\ell)) \hookrightarrow M.$$
Let $U(L(\ell)) = i_{L(\ell)}(\mathcal{N(L(\ell))})$.

By choosing a sufficiently small $L(\ell)$, without loss of generality, the following hold:
\begin{gather*}
i_{L(\ell)}(\mathcal{N}(L(\ell)) \cap T^*_xL(\ell)) \cap \mathcal{B} \neq \varnothing \text{  for all } x \in L(\ell), \\
i_{L(\ell)}(\mathcal{N}(L(\ell)) \cap T^*_xL(\ell)) \pitchfork \mathcal{B} \text{  for all } x \in L(\ell), \\
U(\ell) \cap U(\ell') = \varnothing \text{  if  }  \ell \neq \ell'.
\end{gather*}

If $p \in \mathcal{B}$ is close to the branch locus, in other words, there is a connected component $\ell$ of $Locus(\mathcal{B})$ such that $p \in \mathcal{B} \cap U(\ell)$, then there exists $x \in L(\ell)$ such that $p \in i_{L(\ell)}(\mathcal{N}(L(\ell)) \cap T^*_xL(\ell))$.
Let $F_p := i_{L(\ell)}(\mathcal{N}(L(\ell)) \cap T^*_xL(\ell))$.
%For each branch locus $\ell$, we choose a closed neighborhood $U(\ell)$ such that $\ell$ is contained in the interior of $U(\ell)$, and $U(\ell) \cap U(\ell') = \varnothing$ if $\ell \neq \ell'$.
%We note that $(U(\ell), \omega_{U(\ell)})$ is a symplectic manifold with boundary, equipped with a Riemannian metric $g_{U(\ell)}$. 
%By Definition \ref{def of branched submfd}, $\mathcal{B} \cap U(\ell)$ is a union of Lagrangian submanifolds of $U(\ell)$, which are containing $\ell$. 
%We randomly choose one of those Lagrangian submanifolds and call it $L(\ell)$. 
%
%By Remark \ref{rmk natural embedding}, there exits an embedding $i_{L(\ell)} : \mathcal{N}(L(\ell)) \hookrightarrow M$.
%We set $F_p$ as follows:
%\begin{gather*}
%F_p := i_{L(\ell)}\big(\mathcal{N}\big(L(\ell)\big) \cap T^*_p L(\ell)\big).
%\end{gather*} 
Then, $F_p$ is a closed Lagrangian disk containing $p$.

If $p \in \ell$, then, 
\begin{gather}
\label{eqn local properties}
F_p \pitchfork \mathcal{B} \text{ and } \partial F_p \cap \mathcal{B} = \varnothing.
\end{gather}
Moreover, by choosing a sufficiently small $L(\ell)$, for every $p \in \mathcal{B} \cap U(\ell)$, Equation \eqref{eqn local properties} holds. 

\begin{figure}[h]
	\centering
	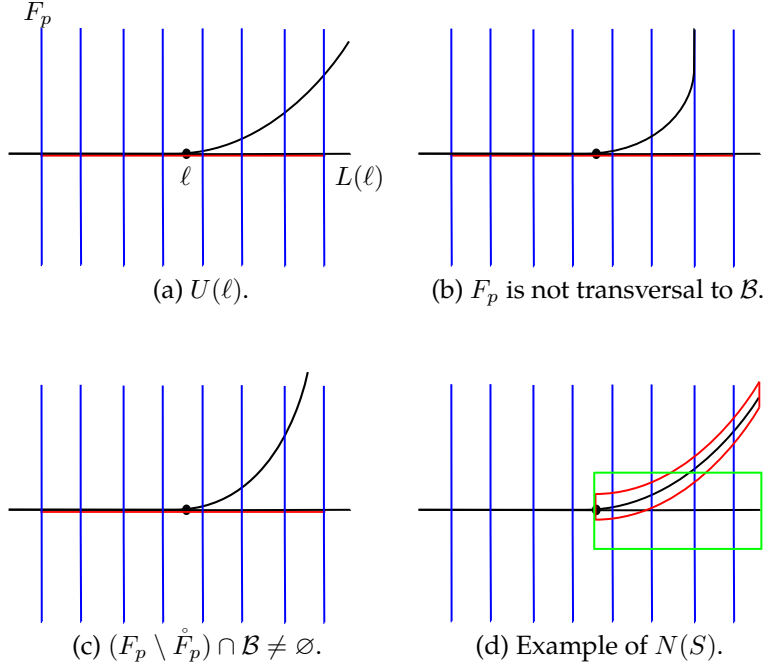
	\caption{Black curves are part of a Lagrangian branched submanifold and the black marked points denote a connected component $\ell$ of $Locus(\mathcal{B})$. 
	in (a), $L(\ell)$ is in red, and the fibers $F_p$, for $p \in \mathcal{B} \cap U(\ell)$, are in blue;
	(b) and (c) are not allowed by Equation \eqref{eqn local properties}; 
	and in (d), the red and green boxes are examples of $N(S)$.}
	\label{figure construction of fibered neighborhood}
\end{figure}

After possibly renaming $U(\ell)$, from now we assume that
$$U(\ell) = \cup_{p \in L(\ell)}F_p.$$
If $p \in \mathcal{B}\cap U(\ell)$, then there is a unique $q \in L(\ell)$ such that $p \in F_q$. 
We define $F_p := F_q$.
Thus, for $p \in \mathcal{B}$ which is close to $Locus(\mathcal{B})$, i.e., $p \in U(\ell)$ for some connected component $\ell$ of $Locus(\mathcal{B})$, we can define a fiber $F_p$ at $p$.
\vskip0.2in

\noindent{\em Fibration over $S \setminus \cup_{\ell} U(\ell)$.}
If $p \in \mathcal{B} \setminus \cup_\ell U(\ell)$, then there is a sector $S$ of $\mathcal{B}$ containing $p$. 
Since $S$ is Lagrangian, there is an embedding $i_S : \mathcal{N}(S) \hookrightarrow M$. 
We can assume that $\mathcal{N}(S)$ is small enough, so that 
\begin{gather*}
F_q \cap i_S\big(\mathcal{N}(S)\big) \subset \mathring{F_q} = F_q \setminus \partial F_q \text{  for any  } q \in \mathcal{B}\cap U(\ell), \\
\big( i_S(\mathcal{N}(S)) \setminus \cup U(\ell) \big) \cap \big( i_{S'}(\mathcal{N}(S')) \setminus \cup U(\ell) \big) = \varnothing.
\end{gather*}
Figure \ref{figure construction of fibered neighborhood} (d) represents examples of $\mathcal{N}(S)$.
We define $B_p$ for all $p \in S$ by setting
$$B_p:= i_S\big(\mathcal{N}(S) \cap T^*_pS\big).$$

For any sector $S$, let $S^\circ:= S - \cup_{\ell} \operatorname{Int} U(\ell)$.
Then, $S^\circ$ is a Lagrangian submanifold with boundary. 
The boundary of $S^\circ$ is a union of $S(\ell) := S \cap \partial\big(U(\ell)\big)$.
We fix a tubular neighborhood of $S(\ell)$, which is contained in $S^\circ$, and identify the tubular neighborhood with $S(\ell) \times [0,1)$. 
For convenience, we will pretend that $S(\ell) \times [0,1] \subset S$ and $S(\ell) \times \{0\} = S(\ell)$. 

If $p \in S^\circ$ does not lie in any $S(\ell) \times (0,1)$, then we set $F_p:=B_p$. 
\vskip0.2in

\noindent{\em Interpolation on $S(\ell)\times[0,1]$.}
If there is a connected component $\ell$ of $Locus(\mathcal{B})$ such that $p = (p_0, t_0) \in S(\ell) \times (0,1)$, we will construct $F_{p=(p_0,t_0)}$ from $F_{(p_0,0)}$ and $F_{(p_0,1)}$.
To do this, we need the following facts:

First, by the definition of $F_{(p_0,0)}$, $F_{(p_0,0)} \cap i_S\big(\mathcal{N}(S)\big)$ is a Lagrangian disk which contains $(p_0,0)$, and is transversal to $\mathcal{B}$ at $(p_0,0)$. 
Also, $B_{(p_0,0)}$ is also a Lagrangian disk which contains $(p_0,0)$, and is transversal to $\mathcal{B}$. 

By the Lagrangian neighborhood theorem \cite{MR0286137}, we can see $F_{(p_0,0)} \cap i_S\big(\mathcal{N}(S)\big)$ as a graph of a closed section in $T^*B_{(p_0,0)}$, i.e.,
$$F_{(p_0,0)} \cap i_S\big(\mathcal{N}(S)\big) = i_{B_{(p_0,0})}\big(\text{the graph of a closed section in  }T^*B_{(p_0,0)}\big).$$
Every closed section of $T^*B_{(p_0,0)}$ is an exact section because $B_{(p_0,0)}$ is a disk.
Thus, there is a function $f_{(p_0,0)} : B_{(p_0,0)} \to \mathbb{R}$ such that 
$$ F_{(p_0,0)} \cap i_S\big(\mathcal{N}(S)\big) = i_{B_{(p_0,0)}}\big(\text{the graph of } df_{(p_0,0)}\big).$$

Second, we will fix a Riemannian metric $g$ compatible with $\omega$ for convenience.
By restricting $g$ to $S$, $S$ is equipped with a Riemannian metric $g|_S$.
Thus, for any $t_0 \in [0,1]$, there is a parallel transport induced by $g|_S$, between $T_{(p_0,t_0)}S$ and $T_{(p_0,0)}S$ along $\gamma_{p_0}(t) = (p_0, t) \in S$. 
Also, $g$ induces a bijection between $T_{(p_0,0)}S$ (resp.\ $T_{(p_0,t_0)}S$) and $T_{(p_0,0)}^*S$ (resp.\ $T_{(p_0,t_0)}^*S$).
Thus, there is a bijective map between $B_{(p_0,t_0)}$ and $B_{(p_0,0)}$. 

From those two facts, we define a function $f_{(p_0,t)} : B_{(p_0,t)} \to \mathbb{R}$ as follows:
\begin{gather*}
f_{(p_0,t)} : B_{(p_0,t)} \stackrel{\sim}{\to} B_{(p_0,0)} \xrightarrow{(1-t)f_{(p_0,0)}} \mathbb{R}.
\end{gather*}
The first arrow comes from the parallel transport induced by $g$. 

There is a map,
\begin{gather*}
h : \cup_{(p_0,t) \in S(\ell) \times [0,1]} B_{(p_0,t)} \to M, \\
x \in B_{(p_0,t)} \mapsto i_{B_{(p_0,t)}} ( d f_{B_{(p_0,t)}}(x)).
\end{gather*}
It is easy to check that $h(p_0,t) = (p_0,t)$.
Moreover, $h$ is the associated (time 1) flow of the Hamiltonian vector field of 
$$f_{(p_0,t)} : \cup_{(p_0,t) \in S(\ell) \times [0,1]} B_{(p_0,t)} \to \mathbb{R}.$$
Finally, we construct $F_{(p_0,t_0)}$ by setting 
$$F_{(p_0,t_0)}:= h(B_{(p_0,t_0)}).$$

\begin{figure}[h]
	\centering
	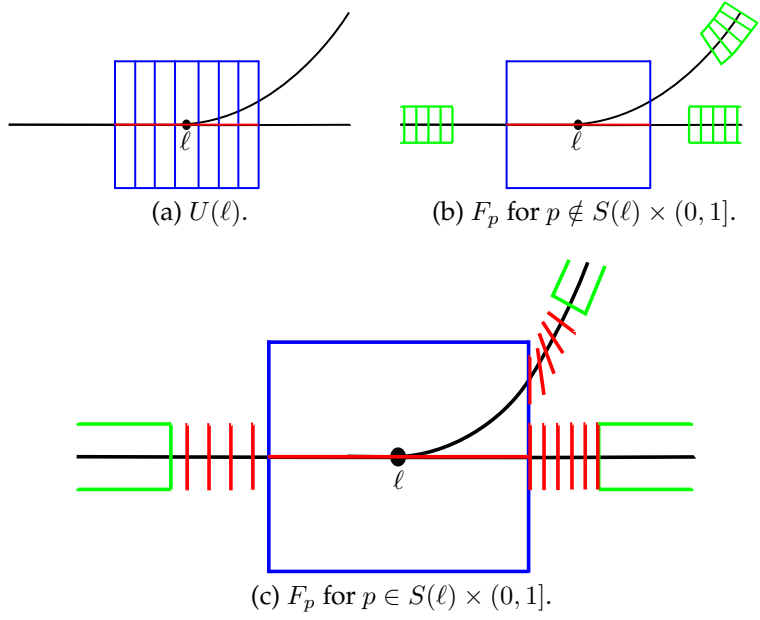
	\caption{Black curves are part of a Lagrangian branched submanifold and marked points denote $\ell$; 
		in (a), $U(\ell)$ is shaded blue, the vertical line segments are fibers;
		(b) fiber $F_p$ for $p \notin S(\ell) \times (0,1]$ is in green; 
		and in (c), fiber $F_p$ for $p \in S(\ell) \times (0,1]$ is in red}
	\label{figure fibered neighborhood}
\end{figure}

A fibered neighborhood $N(\mathcal{B})$ is given by the union of fibers, i.e., $ N(\mathcal{B}) = \cup_{p \in \mathcal{B}} F_p$.
Note that the construction of $N(\mathcal{B})$ is not unique because the construction depends on some choices, including the choices of $L(\ell)$ and a Riemannian metric $g$.
\vskip0.2in

\subsection{Associated branched manifolds and the notion of ``carried by''.} 
We constructed a fibered neighborhood $N(\mathcal{B})$. 
From now on, we will define a projection map defined on $N(\mathcal{B})$, in order to define the notion of ``carried by''.

First, we define {\em the associated branched manifold $\mathcal{B}^*$} of $\mathcal{B}$. 
\begin{definition}
	Let $\mathcal{B}$ be a Lagrangian branched submanifold of $M$ and let $N(\mathcal{B})$ be a fibered neighborhood of $\mathcal{B}$.
	Then, the {\em associated branched submanifold} $\mathcal{B}^*$ is defined by setting 
	$$ \mathcal{B}^* := N(\mathcal{B}) / \sim, \hspace{0.2em} x \sim y \text{  if  } \exists F_p \text{  such that  } x, y \in F_p.$$	 
\end{definition} 
Let  $\pi : N(\mathcal{B}) \to \mathcal{B}^*$ denote the quotient map. 

Before defining the notion of ``carried by'', we note that $\mathcal{B}^*$ is not contained in $M$.
Moreover, since $\mathcal{B}^*$ is a branched manifold, we can define the branch locus and sectors of $\mathcal{B}^*$ as follows:
\begin{definition}
	\label{def of branch locus and sector for abstract branched manifold}
	$\mbox{}$
	\begin{enumerate}
		\item A {\em sector} of $\mathcal{B}^*$ is a connected component of 
		$$ \{ p \in \mathcal{B}^* \hspace{0.2em} | \hspace{0.2em} p \text{  has a neighborhood which is homeomorphic to  } \mathbb{R}^n \}.$$
		\item A {\em branch locus} of $\mathcal{B}^*$ is the complement of all the sectors. 
	\end{enumerate}
\end{definition}

\begin{figure}[h]
	\centering
	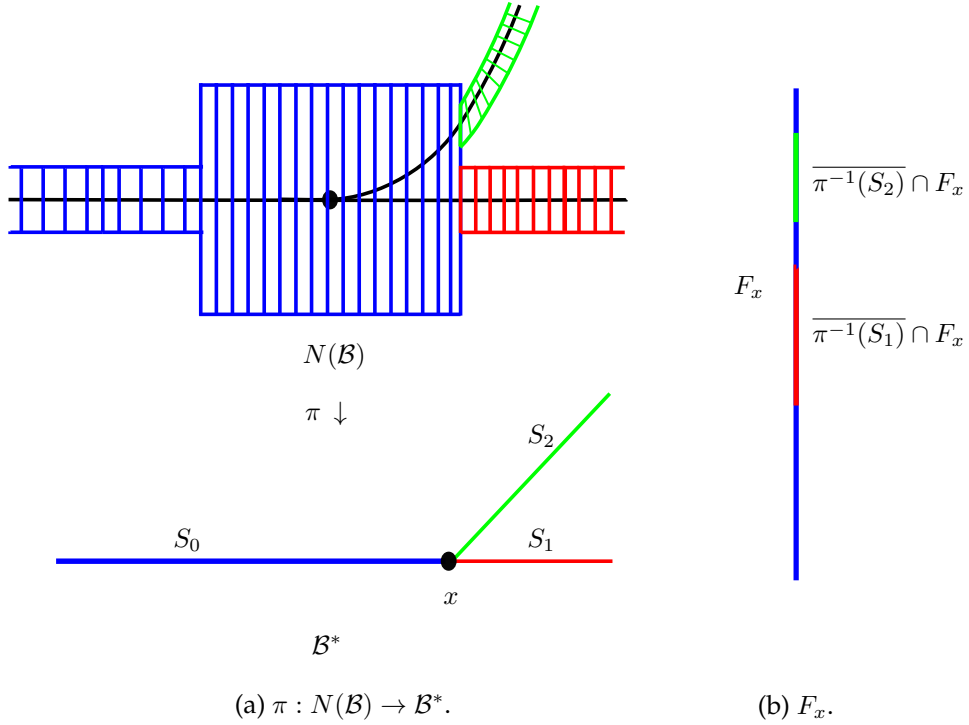		
	\caption{(a) represents $\pi:N(\mathcal{B}) \to \mathcal{B}^*$.
		In $N(\mathcal{B})$, the blue, red, and green represent $\pi^{-1}(S_0)$, $\pi^{-1}(S_1)$, and $\pi^{-1}(S_2)$, where $S_i$ is the corresponding sector of $\mathcal{B}^*$;
		(b) represents $F_x$ where $x$ is in the branch locus of $\mathcal{B}^*$ in (a).}
	\label{fig abstract branched manifold}
\end{figure}

\begin{remark}
	\label{rmk nonuniqueness of fibered neigbhrohood}
	\mbox{}
	\begin{enumerate}
		\item The construction of $N(\mathcal{B})$ depends on the choices of a Riemannian metric, a closed neighborhood of $Locus(\mathcal{B})$, and so on. 
		Thus, fibered neighborhoods $N(\mathcal{B})$ of $\mathcal{B}$ are not unique.
		However, $\mathcal{B}^*$ is unique as a branched manifold since $\mathcal{B}$ and $\mathcal{B}^*$ are equivalent as branched manifolds.
		 
		In the rest of this paper, when it comes to a Lagrangian branched submanifold $\mathcal{B}$, we will consider a triple $(\mathcal{B}, N(\mathcal{B}), \mathcal{B}^*)$ with an arbitrary choice of $N(\mathcal{B})$. 
		Moreover, for any triple $(\mathcal{B}, N(\mathcal{B}), \mathcal{B}^*)$, the projection map is denoted by $\pi$ for convenience. 
		\item A fibered neighborhood$N(\mathcal{B})$ is a union of fibers, i.e., $N(\mathcal{B}) = \cup_{p \in \mathcal{B}} F_p$.
		In the equation, $\mathcal{B}$ is an index set. 
		However, there is a possibility of having two distinct points $p, q \in \mathcal{B}$ such that $F_p = F_q$. 
		From now on, we will use $\mathcal{B}^*$ as an index set.
		In other words, we replace $F_p$ by $\pi^{-1}(\pi(p))$.
		By abuse of notation, $F_x$ denotes $\pi^{-1}(x)$ for all $x \in \mathcal{B}^*$. 
		\item Let $x$ be a branch point of $\mathcal{B}^*$. 
		Then, there are sectors $S_0, S_1, \cdots, S_l$ of $\mathcal{B}^*$ for some $l$ such that 
		\begin{gather*}
		x \in \bar{S}_i \text{  for every  } i = 0, 1, \cdots, l \\
		F_x \cap \overline{\pi^{-1}(S_0)} = F_x  \text{  and  } F_x \cap \overline{\pi^{-1}(S_i)} \subset \mathring{F}_x = F_x \setminus \partial F_x  \text{  for every  } i = 1, 2, \cdots, l.
		\end{gather*}
		Figure \ref{fig abstract branched manifold} represents this. 
	\end{enumerate}
\end{remark}
	
From now on, we define the notion of ``carried by''. 
If a Lagrangian submanifold $L$ (resp.\ a Lagrangian branched submanifold $\mathcal{L}$) is contained in $N(\mathcal{B})$, there is a restriction of $\pi$ to $L$ (resp.\ $\mathcal{L}$).
For convenience, we will simply use $\pi$ instead of $\pi|_L : L \to \mathcal{B}^*$.  

\begin{definition}
	\label{def singular/regular point}
	Let $L$ be a Lagrangian submanifold (resp.\ $\mathcal{L}$ be a Lagrangian branched submanifold) of $N(\mathcal{B})$. 
	\begin{enumerate}
		\item $x \in L$ (resp.\ $\mathcal{L}$) is a {\em regular point} of $\pi$ if $L \pitchfork F_{\pi(x)}$ (resp.\ $\mathcal{L} \pitchfork F_{\pi(x)}$) at $x$.
		\item $x \in L$ (resp.\ $\mathcal{L}$) is a {\em singular point} of $\pi$ if $x$ is not regular point of $\pi: L' \to \mathcal{B}^*$.
		Moreover, values of $\pi$ at singular points are called {\em singular values} of $\pi$.
		$y \in \mathcal{B}^*$ is a {\em singular value} of $\pi$ if there is a singular point $x$ of $\pi$ such that $\pi(x) = y$. 
		\item $L$ is {\em minimally singular with respect to $\mathcal{B}$} if $\pi: L \to \mathcal{B}^*$ has no singular value on the branch locus of $\mathcal{B}^*$ and $|F_x \cap L| = |F_y \cap L|$, for any non-singular value $x$ and $y$ which lie in the same sector of $\mathcal{B}^*$, where $|\cdot|$ means the cardinality of a set.
	\end{enumerate}
\end{definition}
	
\begin{definition}
	\label{def of carried by}
	$\mbox{}$
	\begin{enumerate}
		\item A Lagrangian submanifold $L$ (resp.\ a Lagrangian branched submanifold $\mathcal{L}$) is {\em strongly carried by} a Lagrangian branched submanifold $\mathcal{B}$ if $L$ (resp.\ $\mathcal{L}$) is Hamiltonian isotopic to a Lagrangian submanifold $L'$ (resp.\ a Lagrangian branched submanifold $\mathcal{L}'$) such that $L'$ (resp.\ $\mathcal{L}'$) $\subset N(\mathcal{B})$ and $\pi: L' \to \mathcal{B}^*$ has no singular value.
		\item A Lagrangian submanifold $L$ (resp.\ a Lagrangian branched submanifold $\mathcal{L}$) is {\em weakly carried by} a Lagrangian branched submanifold $\mathcal{B}$ if $L$ (resp.\ $\mathcal{L}$) is Hamiltonian isotopic to a Lagrangian submanifold $L'$ (resp.\ a Lagrangian branched submanifold $\mathcal{L}'$) such that $L'$ (resp.\ $\mathcal{L}'$) $\subset N(\mathcal{B})$, $L'$ is minimally singular, and $\pi: L' \to \mathcal{B}^*$ has countably many singular values.
		\item Two Lagrangian submanifolds $L$ and $L'$ that are weakly carried by $\mathcal{B}$ are {\em weakly fiber isotopic} if there exists an isotopy for $L$ and $L'$ through Lagrangians that are weakly carried by $\mathcal{B}$.
	\end{enumerate}
\end{definition}
In the rest of this paper, if $L$ is weakly carried by $\mathcal{B}$, then we will assume that $L \subset N(\mathcal{B})$ and $L$ is minimally singular with respect to $\mathcal{B}$. 
 
Note that the notion of ``carried by'' used by Thurston in \cite{MR956596} is our notion of ``strongly carried by''. 
Thurston showed that for a pseudo-Anosov surface automorphism $\psi : S \stackrel{\sim}{\to} S$, there is a 1-dimensional branched submanifold $\tau$ which is called a train track such that $\psi(\tau)$ is strongly carried by $\tau$.

Our higher-dimensional generalization is slightly weaker,
i.e., for some symplectic automorphism $\psi : (M,\omega) \stackrel{\sim}{\to} (M,\omega)$, we construct a Lagrangian branched submanifold $\mathcal{B}_{\psi}$ such that $\psi(\mathcal{B}_{\psi})$ is weakly carried by $\mathcal{B}_{\psi}$.
In other words, we allow non-transversality at countably many point $p \in \mathcal{B}_{\psi}$. 
However, we allow only one type of non-transversality. 
In the rest of the present subsection, we will describe the unique type of non-transversality.   

\begin{definition}
	\label{def of singularity}
	Let $L$ be weakly carried by $\mathcal{B}$. 
	A {\em singular component} $V$ of $\pi : L \to \mathcal{B}$ is a connected component of the set of all singular points of $\pi$.
\end{definition}

\begin{exmp}
	\label{exmp of simplest singularity}
	Let $M_*$ be a symplectic manifold $T^*\mathbb{R}^n \simeq \mathbb{R}^{2n}$ equipped with the canonical symplectic form.
	The zero section $\mathcal{B}_*:= \mathbb{R}^n \times 0 \subset \mathbb{R}^{2n}$ is a Lagrangian branched submanifold. 
	We assume that the fibered neighborhood $N(\mathcal{B}_*)$ is $M_*$, by setting $F_p := T^*_p \mathbb{R}^n$ for all $p \in \mathbb{R}^n = \mathcal{B}_*$.
	Then, a Lagrangian submanifold
	$$L_*:= \{(tx,x) \in \mathbb{R}^n \times \mathbb{R}^n \hspace{0.2em} | \hspace{0.2em} t \in \mathbb{R}, x \in S^{n-1} \subset \mathbb{R}^n\}$$
	is weakly carried by $\mathcal{B}_*$ and $\pi_*$ has only one singular component
	$$V_* := \{ (0,x) \hspace{0.2em} | \hspace{0.2em} x \in S^{n-1} \}.$$
\end{exmp}

\begin{definition}
	\label{def of real blow-up type}
	A singular component $V$ of $\pi : L \to \mathcal{B}$ is of {\em real blow-up type} if there exists an open neighborhood $U$ of $V$ and a symplectomorphism $\phi : U \stackrel{\sim}{\to} M_*$ such that $\phi(U \cap \mathcal{B}) = \mathcal{B}_*, \phi(V) = V_*$, and $\phi^{-1} \circ \pi_* \circ \phi = \pi$, where $M_*, \mathcal{B}_*$, $V_*$, and $\pi_*$ are defined in Example \ref{exmp of simplest singularity}.
\end{definition}

\begin{definition}
	\label{def of fully/weakly carried by}
	A Lagrangian submanifold $L$ (resp.\ a Lagrangian branched submanifold $\mathcal{L}$) is {\em carried by} a Lagrangian branched submanifold $\mathcal{B}$ if $L$ (resp.\ $\mathcal{L}$) is weakly carried by $\mathcal{B}$ and every singular component of $\pi$ (resp.\ $\pi$) is a singular component of real blow-up type.
\end{definition}

\subsection{The generalized Penner construction}  
In this subsection, we give a higher-dimensional generalization of Penner construction \cite{MR930079} of pseudo-Anosov surface automorphisms. 
The generalization replaces Dehn twists by generalized Dehn twists along Lagrangian spheres.

\underline{Generalized Penner construction} : Let $M$ be a symplectic manifold.
A symplectic automorphism $\psi : M \stackrel{\sim}{\to} M$ is of {\em generalized Penner type} if there are two collections $A = \{\alpha_1, \cdots, \alpha_m \}$ and $B = \{\beta_1, \cdots, \beta_l \}$ of Lagrangian spheres 
such that 
\begin{gather*}
\alpha_i \cap \alpha_j = \varnothing, \hspace{0.2em} \beta_i \cap \beta_j =\varnothing, \hspace{0.2em} \text{for all} \hspace{0.2em} i \neq j, \\
\alpha_i \pitchfork \beta_j \hspace{0.2em} \text{for all} \hspace{0.2em} i, j,
\end{gather*}
so that
$\psi$ is a product of positive powers of Dehn twists $\tau_i$ along $\alpha_i$ and negative powers of Dehn twists $\sigma_j$ along $\beta_j$, subject to the condition that every sphere appear in the product. 

A Lagrangian sphere $\alpha_i$ (resp.\ $\beta_j$) is called a {\em positive} (resp.\ {\em negative}) sphere since only positive powers of $\tau_i$ (resp.\ negative powers of $\sigma_j$) appear in $\psi$.

\begin{remark}
	$\mbox{}$
	\begin{enumerate}
		\item In Theorems \ref{branched surface thm} and \ref{lamination thm}, we can assume that the symplectic manifold $M$ is a plumbing space.
		Every $\tau_i$ (resp.\ $\sigma_j$) is supported on a neighborhood of $\alpha_i$ (resp.\ $\beta_j$), which is denoted by $T^*\alpha_i$ (resp.\ $T^*\beta_j$). 
		Thus, $\psi$ is supported on the union of $T^*\alpha_i$ and $T^*\beta_j$. 
		By the transversality condition $\alpha_i \pitchfork \beta_j$, we can identify the union with a plumbing space $P=P(\alpha_1, \cdots, \alpha_m, \beta_1, \cdots, \beta_l)$.
		Thus, it is suffices to prove Theorems \ref{branched surface thm} and \ref{lamination thm} on the plumbing space $P$, which we take to be connected.
		\item In \cite{MR930079}, the Penner construction required that $A$ and $B$ fill the surface $S$, i.e., the complement of $A \cup B$ is a union of disks and annuli, one of whose boundary components is a component of $\partial S$.
		In the current paper, we do not require the analogue of the filling condition since we only construct an invariant Lagrangian branched submanifold and an invariant Lagrangian lamination, not an invariant singular foliation on all of $M$.
	\end{enumerate}
\end{remark}

In the rest of this subsection, we define a set of Lagrangian branched submanifolds in a plumbing space $P(\alpha_1, \cdots, \alpha_m, \beta_1, \cdots, \beta_l)$. 
We start from the simplest plumbing space, having one positive and one negative sphere intersecting at only one point. 
\begin{exmp}
	\label{exmp of Lagrangian branched submanifold}
	Let $\alpha$ and $\beta$ be $n$-dimensional spheres and let $M$ be a plumbing $P(\alpha, \beta)$ which is plumbed at only one point $p$.
	Let $\beta \#_p \alpha$ be the Lagrangian surgery of $\alpha$ and $\beta$ at $p$ such that $\beta \#_p \alpha \simeq \tau_{\alpha}(\beta) \simeq \sigma_{\beta}^{-1}(\alpha)$.
	See Figure \ref{LBS example}, which represents the case $n=1$.
	The cross-shape is the plumbing space $P(\alpha,\beta)$, where $\alpha$ is the horizontal line and $\beta$ is the vertical line. 
	\begin{figure}[h]
		\centering
%		\fontsize{11pt}{11pt}\selectfont% or whatever fontsize you like
		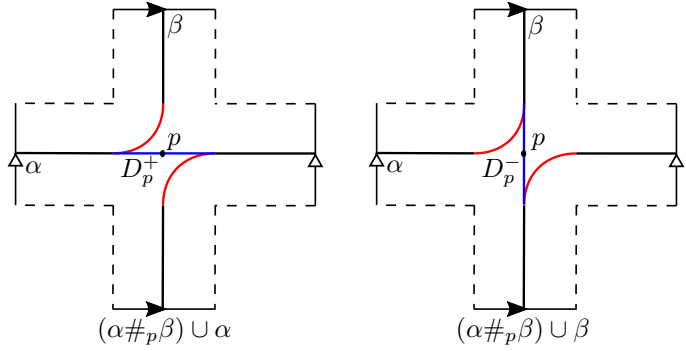
		\caption{The blue curves represent $D_p^+$ in the left hand picture and $D_p^-$ in the right hand picture, the red curves represent $N_p$ in both.}
		\label{LBS example}
	\end{figure}

	The {\em neck $N_p$ at $p$} connecting $\alpha$ and $\beta$ is the closure of $(\beta\#_p \alpha) - (\alpha \cup \beta)$.
	In Figure \ref{LBS example}, $N_p$ is drawn in red.
	The {\em positive disk $D_p^+$at $p$} is the closure of $\alpha - (\beta\#_p \alpha)$ and the {\em negative disk $D_p^-$ at $p$} is the closure of $\beta - (\beta \#_p \alpha)$. 
	The disks $D_p^{\pm}$ are drawn in blue in Figure \ref{LBS example}.
	Then, by attaching $D_p^+$ or $D_p^-$ to $\beta \#_p \alpha$, we obtain  Lagrangian branched submanifolds $(\beta \#_p \alpha) \cup \alpha$ and $(\beta \#_p \alpha) \cup \beta$.
\end{exmp}

On a general plumbing space $M = P(\alpha_1, \cdots, \alpha_m, \beta_1, \cdots, \beta_l)$ with positive spheres $\alpha_i$ and negative spheres $\beta_j$, we similarly construct Lagrangian branched submanifolds.
More precisely, given a plumbing point $p, N_p, D_p^+, D_p^-$ are the closures of $(\beta_j \#_p \alpha_i) - (\alpha_i \cup \beta_j), \alpha_i - (\beta_j \#_p \alpha_i), \beta_j - (\beta_j \#_p \alpha_i)$ respectively.
Let $D_p$ be either $D_p^+$ or $D_p^-$.
Then, we construct a Lagrangian branched submanifold $\mathcal{B}$ by setting 
\begin{align}
\label{eqn definition of lbs}
\mathcal{B} := \cup_{i}(\alpha_i - \cup_{p \in \alpha_i}D_p^+) \bigcup \cup_{j}(\beta_j - \cup_{p \in \beta_j}D_p^-) \bigcup \cup_p N_p \bigcup \cup_p D_p.
\end{align}
There are $2^N$ possible choices of $\mathcal{B}$, where $N$ is the number of plumbing points. 
Let $\mathbb{B}$ be the set of all $2^N$ Lagrangian branched submanifolds constructed above.

\subsection{Proof of Theorem \ref{branched surface thm}}

In this subsection, let $M = P(\alpha_1, \cdots, \alpha_m, \beta_1, \cdots, \beta_l)$, let $\tau_i$ (resp.\ $\sigma_j$) be a Dehn twist along $\alpha_i$ (resp.\ $\beta_j$), and let $\psi$ be of generalized Penner type. 

In the rest of the paper, we assume that every Dehn twist $\tau_i$ and $\sigma_j$ satisfies the following:
\begin{enumerate}
	\item $\tau_i$ (resp.\ $\sigma_j$) is supported on a small neighborhood $T^*\alpha_i$ (resp.\ $T^*\beta_j$) of $\alpha_i$ (resp.\ $\beta_j$). 
	\item $\tau_i$ (resp.\ $\sigma_j$) agrees with the antipodal map on $\alpha_i$ (resp.\ $\beta_j$).
\end{enumerate}

We define the following: 
\begin{gather}
\label{eqn definition of disks and primes}
\bar{D}_p^+ := \tau_i(D_p^+), \hspace{0.3em} \bar{D}_p^- := \sigma_j^{-1}(D_p^-) \hspace{1em} \text{   if }  p \in \alpha_i \cap \beta_j, \\
\nonumber
\alpha_i' := \alpha_i - \cup_{p \in \alpha_i} (D_p^+ \cup \bar{D}_p^+), \hspace{0.3em} \beta_j' : = \beta_j - \cup_{p \in \beta_j} (D_p^- \cup \bar{D}_p^-).  
\end{gather}
In words, $\bar{D}_p^+$ (resp.\ $\bar{D}_p^-$) is a neighborhood of an antipodal point of $p$ in $\alpha_i$ (resp.\ $\beta_j$).
We are assuming that $D_p^{\pm}$ and $\bar{D}_p^{\pm}$ are sufficiently small so that they are disjoint to each other.

Recall that $\mathbb{B}$ is the set of Lagrangian branched submanifolds defined in Section 3.2.

\begin{lemma}
	\label{lem1}
	For all $k$, there exists a function $F_{\tau_k}: \mathbb{B} \to \mathbb{B}$ such that $\tau_k(\mathcal{B})$ is carried by $F_{\tau_k}(\mathcal{B})$  for all $\mathcal{B} \in \mathbb{B}$.
	Similarly, there is a function $F_{\sigma_j^{-1}}:\mathbb{B} \to \mathbb{B}$ for all $j$ such that $\sigma_j^{-1}(\mathcal{B})$ is carried by $F_{\sigma_j^{-1}}(\mathcal{B})$. 
\end{lemma}

\begin{proof}
	In this proof, $\tau_k$ is given by Equation \eqref{eqn definition of generalized Dehn along L} and $\tilde{\tau}:T^*S^n \stackrel{\sim}{\to} T^*S^n$ defined in Section 2.2, i.e., $\tau_k = \phi \circ \tilde{\tau} \circ \phi^{-1}$ in a neighborhood of $\alpha_k$, where $\phi$ is an identification of $T^*S^n$ and a neighborhood of $\alpha_k$.  
	
	Given $\mathcal{B} \in \mathbb{B}, \mathcal{B}$ admits the following decomposition:
	\begin{align}
	\label{eqn decomposition}
	\mathcal{B}= \cup_i \alpha_i' \bigcup \cup_j \beta_j' \bigcup \cup_p N_p \bigcup \cup_p \bar{D}_p^+ \bigcup \cup_p \bar{D}_p^- \bigcup \cup_p D_p,
	\end{align}	
	where $D_p$ is either $D_p^+$ or $D_p^-$.
	This follows from Equations \eqref{eqn definition of lbs} and \eqref{eqn definition of disks and primes}.
	
	We prove the first statement for $\tau_k$; the proof for $\sigma_j^{-1}$ is analogous.
	Our strategy is to apply $\tau_k$ to $\alpha_i', \beta_j', N_p, \bar{D}_p^{\pm}$, and $D_p^{\pm}$.
	We claim the following: 
	\begin{itemize}
		\item[(i)] $\tau_k(\alpha'_i) = \alpha_i', \tau_k(\beta'_j) = \beta_j'$ and they are strongly carried by $\alpha_i', \beta_j'$.
		\item[(ii)] If $p \notin \alpha_k$, then
		$\tau_k(N_p) = N_p, \tau_k(D_{p}^{\pm}) = D_{p}^{\pm}, \tau_k(\bar{D}_{p}^{\pm}) = \bar{D}_{p}^{\pm}$ and they are strongly carried by $N_p,\hspace{2pt} D_{p}^{\pm}, \hspace{2pt} \bar{D}_{p}^{\pm}$.
		\item[(iii)] If $ p \in \alpha_k$, then
		$\tau_k(D_p^+) = \bar{D}_p^+, \tau_k(\bar{D}_p^+) = D_p^+$,  $\tau_k(\bar{D}_p^-) = \bar{D}_p^-$ and they are strongly carried by $\bar{D}_p^+, D_p^+, \bar{D}_p^-$.
		\item[(iv)] If $ p \in \alpha_k$, then $\tau_k(D_p^-)$ and $\tau_k(N_p)$ are obtained by spinning with respect to $p$.
		Moreover, $\tau_k(D_p^-)$ is strongly carried by $N_p \cup (\alpha_k - D_p^+ )$ and $\tau_k(N_p)$ is carried by $N_p \cup (\alpha_k - D_p^+ )$. 
	\end{itemize}
	By Equation \eqref{eqn decomposition} and $(i)$--$(iv), \tau_k(\mathcal{B})$ is carried by $\mathcal{B}'$ such that 
	\begin{align}
	\label{eqn define F_}
	\mathcal{B}' = \cup_i \alpha_i' \bigcup \cup_j \beta_j' \bigcup \cup_p N_p \bigcup \cup_p \bar{D}_p^+ \bigcup \cup_p \bar{D}_p^- \bigcup \cup_p \tilde{D}_p,
	\end{align}
	where $\tilde{D}_p$ is $D_p$ if $p \notin \alpha_k$ and $D_p^+$ if $p \in \alpha_k$.
	Then, $F_{\tau_k} : \mathbb{B} \to \mathbb{B}$ is defined by $F_{\tau_k}(\mathcal{B}) = \mathcal{B}'$. 	

	$(i)$ Since $\tau_k$ agrees with the antipodal map on $\alpha_k, \tau_k(\alpha_k') = \alpha_k'$ and $\tau_k(\alpha_k')$ is strongly carried by $\alpha_k'$.
	Moreover, since $\tau_k$ is supported on $T^*\alpha_k, \alpha_i'$ does not intersect the support of $\tau_k$ for all $i \neq k$. 
	Thus, $\tau_k(\alpha_i')$ agrees with $\alpha_i'$ and $\tau_k(\alpha_i')$ is strongly carried by itself.
	The same proof applies to $\tau_k(\beta_j')$.
	
	$(ii)$ and $(iii)$ are proved in the same way.
	
	$(iv)$ We compute $\tau_k(D_p^-)$ and $\tau_k(N_p)$ by spinning with respect to $p$ and $\phi$.
	We assume $\phi((1,0_{n};0_{n+1})) = p$ without loss of generality.
	Using the notation from Section 2, $D_p^-$ and $N_p$ are contained in $\cup_{y \in S^{n-1}} \phi(W_y)$. 
	Thus,
	\begin{align}
	\label{eqn spinning of disk}
	\tau_k(D_p^-) &= \cup_{y \in S^{n-1}} (\phi \circ \tilde{\tau} \circ \phi^{-1}) (D_p^- \cap \phi(W_y))\\
	\nonumber & = \cup_{y \in S^{n-1}} (\phi(\tilde{\tau}|_{W_y}(\phi^{-1}(D_p^-)\cap W_y))) \\
	\nonumber &= \cup_{y \in S^{n-1}} \tau_k(D_p^-) \cap \phi(W_y), \\
	\label{eqn spinning of neck}
	\tau_k(N_p) &= \cup_{y \in S^{n-1}} (\phi \circ \tilde{\tau} \circ \phi^{-1}) (N_p \cap \phi(W_y)) \\
	\nonumber &= \cup_{y \in S^{n-1}} \phi(\tilde{\tau}|_{W_y}(\phi^{-1}(N_p)\cap W_y))\\
	\nonumber &= \cup_{y \in S^{n-1}} \tau_k(N_p) \cap \phi(W_y).
	\end{align}
	The restriction $\tilde{\tau}|_{W_y}$ is a Dehn twist on $W_y \simeq T^*S^1$ along the zero section. 
	Thus, we obtain Figure \ref{figure for lemma 1}, which represents intersections $\phi(W_y) \cap D_p^-, \phi(W_y) \cap N_p, \phi(W_y) \cap \tau_k (D_p^-)$, and $\phi(W_y) \cap \tau_k(N_p)$.
	Equation \eqref{eqn spinning of neck} and Figure \ref{figure for lemma 1} imply that $\tau_k(N_p)$ is carried by $N_p \cup (\alpha_k - D_p^+)$ and $\tau_k(D_p^-)$ is strongly carried by $N_p \cup (\alpha_k - D_p^+)$.  
	\begin{figure}[h]
		\centering
		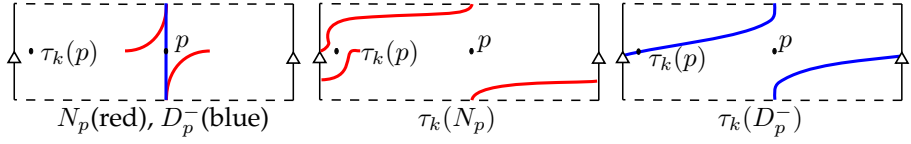		
		\caption{In the left picture, the blue curve represents $D_p^-$ and the red curve represents $N_p$; in the middle picture, the red curve represents $\tau_k(N_p)$; and in the right picture, the blue curve represents $\tau_k(D_p^-)$.}
		\label{figure for lemma 1}
	\end{figure}
 
	Then, $(i)$--$(iv)$ and Equation \eqref{eqn decomposition} prove that $\tau_k(\mathcal{B}) $ is carried by $F_{\tau_k}(\mathcal{B})$.
\end{proof}

\begin{lemma}
	\label{lem2}
	If $L$ is a Lagrangian submanifold which is carried by (resp.\ weakly carried by) $\mathcal{B} \in \mathbb{B}$, then $\tau_k(L)$ is carried (resp.\ weakly carried) by $F_{\tau_k}(\mathcal{B})$. 
	The case of $\sigma_j^{-1}$ is analogous.
\end{lemma}

\begin{proof}
	We can assume that $L$ is contained in an arbitrary small neighborhood of $\mathcal{B}$.
	Then, we apply a Dehn twist $\tau_k$ as we did in the proof of Lemma \ref{lem1}.
	The details are similar to the proof of Lemma \ref{lem1}.	
\end{proof}

\begin{proof}[Proof of Theorem \ref{branched surface thm}]
	Let $\psi: M \stackrel{\sim}{\to} M$ be a symplectic automorphism of generalized Penner type.
	Then, we can write $\psi = \delta_1 \circ \cdots \circ \delta_l$ where $\delta_k$ is a Dehn twist $\tau_i$ or $\sigma_j^{-1}$.
	We then define $F_{\psi} = F_{\delta_1} \circ \cdots \circ F_{\delta_l}: \mathbb{B} \to \mathbb{B}$.
	By Lemma \ref{lem1}, we have specific functions $F_{\tau_i}$ and $F_{\sigma_j^{-1}}$ acting on $\mathbb{B}$. 

	We claim that $F_{\psi}$ is a constant map, i.e., $\operatorname{Im}(F_{\psi})$ is a point $\mathcal{B}_{\psi}$, which we define as follows: in Equation \eqref{eqn definition of lbs}, we set $D_p = D_p^+$ for $p \in \alpha_i \cap \beta_j$ if the last $\tau_i$ in $\psi$ appears later than the last $\sigma_j^{-1}$, and $D_p = D_p^-$ otherwise.
	Note that every Dehn twist $\tau_i$ and $\sigma_j^{-1}$ appears in $\psi$, thus $\mathcal{B}_{\psi}$ is well-defined.
	By Equation \eqref{eqn define F_}, $F_{\psi}(\mathcal{B}) = \mathcal{B}_{\psi}$ for all $\mathcal{B} \in \mathbb{B}$.
\end{proof}

\begin{remark}
	\label{rmk position of singular value}
	$\mbox{}$
	\begin{enumerate}
		\item Note that a singular value of $\pi : \psi^m(L) \to \mathcal{B}^*$, which is defined in Section 3.1, can be moved by isotoping $\psi^m(L)$.
		\item We observe that every singular value of $\pi:\psi^m(\mathcal{B}_{\psi}) \to \mathcal{B}^*$ lies near $\pi(p), \pi\big(\tau_i(p)\big),$ or $\pi\big(\sigma_j^{-1}(p)\big)$ by isotoping, where $p$ is a plumbing point. 
		More precisely, let $S_{p, \mathcal{B}_{\psi}}^+$ (resp.\ $S_{p, \mathcal{B}_{\psi}}^-$) be the sector of $\mathcal{B}_{\psi}^*$ containing $\pi(p)$ if $D_p = D_p^+$ (resp.\ $D_p^-$), where $D_p$, $D_p^+$ and $D_p^-$ are defined in Section 3.4.
		Similarly, let $\bar{S}_{p, \mathcal{B}_{\psi}}^{\pm}$ be $\pi(\bar{D}_p^{\pm})$, where $\bar{D}_p^{\pm}$ is defined in Equation \eqref{eqn definition of disks and primes}. 
		Then, by isotoping $\psi^n(\mathcal{B}_{\psi})$, every singular value of $\pi:\psi^n(\mathcal{B}_{\psi}) \to \mathcal{B}^*$ lies in the interiors of $S_{p, \mathcal{B}_{\psi}}^{\pm}$ or $\bar{S}_{p, \mathcal{B}_{\psi}}^{\pm}$ for some plumbing point $p$. 
		
		For convenience, let the {\em centers} of $S_p^\pm$, $\bar{S}_p^+$, $\bar{S}_p^-$ be $p$, $\tau(p)$, $\sigma^{-1}(p)$ respectively. 
		Then, the singular values in $S_p^\pm$, $\bar{S}_p^\pm$ lie near the centers of them. 
		Moreover, $S_{p, \mathcal{B}_{\psi}}^{\pm}$ and $\bar{S}_{p, \mathcal{B}_{\psi}}^{\pm}$ will be simply called $S_p^{\pm}$ and $\bar{S}_p^{\pm}$. 
	\end{enumerate}
\end{remark}

\section{Construction of Lagrangian laminations}
In this section, we will prove Theorems \ref{lamination thm} and \ref{generalized theorem}..

\subsection{Singular and regular disks}
In order to prove Theorems \ref{lamination thm} and \ref{generalized theorem}, we would like to construct a stable Lagrangian lamination $\mathcal{L}$ of a symplectic automorphism $\psi$ from a Lagrangian branched submanifold $\mathcal{B}_{\psi}$. 
One of the difficulties is that singular components occur naturally.
In order to control the singularities, we introduce singular and regular disks.

In general, we assume that $\mathcal{B}_{\psi}^*$, the associated branched manifold, can be decomposed into the union of a finite number of disks $S_i \simeq \mathbb{D}^n$, which are called {\em singular disks}, and $R_j \simeq \mathbb{D}^n$, which are called {\em regular disks}, i.e.,
\begin{gather}
\label{eqn decomp into singular/regular disks}
\mathcal{B}_{\psi}^* = \bigcup_i S_i \cup \bigcup_j R_j,
\end{gather} 
such that
\begin{enumerate}
	\item each singular disk $S_i$ is either a closed disk contained in the interior of a sector of $\mathcal{B}_{\psi}^*$ or a closure of a sector,
	\item $S_i \cap S_j = \varnothing$ for any $i \neq j$,
	\item every singular value of $\pi: \psi^m(\mathcal{B}_{\psi}) \to \mathcal{B}_{\psi}$ after weakly fibered isotopy lies in $\cup_i \mathring{S}_i$ for all $m \in \mathbb{N}$, where $\mathring{S}_i$ is the interior of $S_i$,
	\item each regular disk $R_j$ is obtained by cutting up a closure of a sector minus $\cup_i \mathring{S_i}$,
	\item $S_i$ and $R_j$ (resp.\ $R_i$ and $R_j$ for $i \neq j$) meet only along their boundaries.  
\end{enumerate}

\begin{remark}
	\label{rmk singular value condition}
From now on, for any compact Lagrangian submanifold $L$ which is carried by $\mathcal{B}_{\psi}$, we will assume that every singular value of $\pi : L \to \mathcal{B}_{\psi}$ lies in the interior of a singular disk by Remark \ref{rmk position of singular value}.
\end{remark}

If $\mathcal{B}^*$ admits Equation \eqref{eqn decomp into singular/regular disks}, then one obtains a decomposition of $N(\mathcal{B})$ as follows:
\begin{gather*}
N(\mathcal{B}) = \bigcup_i \pi^{-1}(S_i) \cup \bigcup_j \pi^{-1}(R_j).
\end{gather*}

In Section 4.2, we will define braids $b(L,S_i)$ for a Lagrangian $L$, which is carried by $\mathcal{B}_{\psi}$, and a singular disk $S_i$. 
By Theorem \ref{branched surface thm}, there exist sequences of braids ${b(\psi^m(L),S_i)}_{m \i \mathbb{N}}$, and we will construct limits of those braid sequences as $m \to \infty$.
We then extend the limit lamination to a Lagrangian lamination of $\pi^{-1}(S_i)$ in Section 4.3, and a Lagrangian lamination of $\pi^{-1}(R_j)$ in Section 4.4.

\begin{remark}
	\label{rmk identify with cotangent bundle of disk}
	\mbox{}
	\begin{enumerate}
		\item In Section 4.3 (resp.\ Section 4.4), we will construct a Lagrangian lamination on  $\overline{\pi^{-1}(\mathring{S}_i)} \subset \pi^{-1}(S_i)$ (resp.\ $\overline{\pi^{-1}(\mathring{R}_j)} \subset \pi^{-1}(R_j)$),  the closure of $\pi^{-1}(\mathring{S}_i)$.
		This is because $\pi^{-1}(S_i)$ (resp.\ $\pi^{-1}(R_j)$) is not a (closed) submanifold of $M$ if $S_i$ (resp.\ $R_j$) intersects the branch locus of $\mathcal{B}^*$. 
		
		Figure \ref{fig abstract branched manifold} is an example. 
		If $S_1$ in Figure \ref{fig abstract branched manifold} is a singular disk, then $\pi^{-1}(S_1)$ is the union of the red box in Figure \ref{fig abstract branched manifold} (a) and $F_x$.

	\item 
	We note that $(\overline{\pi^{-1}(\mathring{S}_i)}, \omega)$ (resp.\ $(\overline{\pi^{-1}(\mathring{R}_j)}, \omega)$) and $(DT^*\mathcal{D}, \omega_0)$ are symplectomorphic to each other, where $\mathcal{D}$ is a closed disk, $DT^*\mathcal{D}$ is a disk cotangent bundle of $\mathcal{D}$, and $\omega_0$ is the standard symplectic form of the cotangent bundle. 
	
	In order to construct a symplectomorphism, we will consider the following:
	Let $\mathcal{D}$ be a largest Lagrangian disk such that 
	$$\mathcal{D} \subset \pi^{-1}(S_i) \cap \mathcal{B} \hspace{0.2em} (\text{resp. } \pi^{-1}(R_j) \cap \mathcal{B}) \text{  and  } \pi(\mathcal{D}) = S_i \hspace{0.2em} (\text{resp. } R_j).$$
	
	By Remark \ref{rmk natural embedding}, there exists a symplectic embedding $i_{\mathcal{D}}:\mathcal{N}(\mathcal{D}) \hookrightarrow M$. 
	It is easy to construct a vector field on $i_{\mathcal{D}}(\mathcal{N}(\mathcal{D}))$, whose (time 1) flow moves $i_{\mathcal{D}}(\mathcal{N}(\mathcal{D}) \cap T_p^*\mathcal{D})$ to $F_{\pi(p)}$ for any $p \in \operatorname{Int}(\mathcal{D})$. 
	Moreover, the vector field is a symplectic vector field, i.e., the flow is a symplectomorphism, and $$\cup_{p \in \operatorname{Int}(\mathcal{D})}i_{\mathcal{D}}(\mathcal{N}(\mathcal{D}) \cap T^*_p \mathcal{D}) \simeq \cup_{p \in \operatorname{Int}(\mathcal{D})} F_{\pi(p)} = \pi^{-1}(\mathring{S}_i) (\text{resp.  } \pi^{-1}(\mathring{R}_j)).$$
	
	By taking the closures, $i_{\mathcal{D}}(\mathcal{N}(\mathcal{D})) \simeq \overline{\pi^{-1}(\mathring{S}_i)}$ (resp.\ $\overline{\pi^{-1}(\mathring{R}_j)}$). 
	Moreover, $\mathcal{N}(\mathcal{D})$ is symplectomorphic to $DT^*\mathcal{D}$.
	Thus, $DT^*\mathcal{D}$ and $\overline{\pi^{-1}(\mathring{S}_i)}$  (resp.\ $\overline{\pi^{-1}(\mathring{R}_j)}$) are symplectomorphic.
	\end{enumerate}
\end{remark}

From now on, we assume that a symplectic automorphism $\psi$ is of generalized Penner type until the end of Section 4.3. 
\vskip0.2in

\noindent{\em Decomposition of $\mathcal{B}^*_{\psi}$ for $\psi$ of generalized Penner type.}
We will now explain how to decompose $\mathcal{B}^*$, the associated branched manifold of $\mathcal{B} \in \mathbb{B}$, into the union of specific singular and regular disks. 
Note that $\mathbb{B}$ is defined in Section 3.4.

By Remark \ref{rmk position of singular value}, after weakly fiber isotoping, every singular value of $\pi:\psi^m(\mathcal{B}) \to \mathcal{B}^*$ lies in the interior of $S_p$ or $\bar{S}_p^\pm$, where $S_p = S_P^+$ if $D_p= D_P^+$ and $S_p = S_p^-$ if $D_p = D_p^-$.
Let $S_p$ and $\bar{S}_p^\pm$ be the specific singular disks of $\mathcal{B}^*$.

We will divide the complement of singular disks from $\mathcal{B}^*$, i.e.,
\begin{gather}
\label{eqn regular part}
\mathcal{B}^* \setminus \big(\cup_p S_p \sqcup \cup_p \bar{S}_p^+ \sqcup \cup_p \bar{S}_p^- \big),
\end{gather}
into regular disks.
In order to do this, we use a symplectic submanifold $W^{2n-2} \subset M^{2n}$, which is defined as follows: 
For each $\alpha_i$ (resp.\ $\beta_j$), there is an equator $C_{\alpha_i}$ (resp.\ $C_{\beta_j}$) $\simeq S^{n-1}$ such that
\begin{enumerate}
	\item for any plumbing point $p \in \alpha_i$ (resp.\ $\beta_j$), $p$ lies on $C_{\alpha_i}$ (resp.\ $C_{\beta_j}$),
	\item if $p \in \alpha_i \cap \beta_j$, then $T^*C_{\alpha_i} \equiv T^*C_{\beta_j}$ near $p$.  
\end{enumerate}
Note that the equators on a Lagrangian sphere $\alpha_i$ (resp.\ $\beta_j$) are defined using an identification $\phi_{\alpha_i} : \alpha_i \stackrel{\sim}{\to} S^n$ (resp.\ $\phi_{\beta_j}: \beta_j \stackrel{\sim}{\to} S^n$).
Thus, by choosing proper identification $\phi_{\alpha_i}$ and $\phi_{\beta_j}$, we can assume the existence of $C_{\alpha_i}$ and $C_{\beta_j}$.
Then,
$$W:= \cup_{i} T^*C_{\alpha_i} \bigcup \cup_{j}T^*C_{\beta_j}$$
is a $(2n-2)$-dimensional symplectic submanifold of $M$.

We cut \eqref{eqn regular part} along $\pi(W)$.
These are the regular disks $R_k$.
Each $R_k$ is a manifold with corners, where the corners are at $R_k \cap \pi(W) \cap S_l$.

\subsection{Braids}
Consider the decomposition of $\mathcal{B}_{\psi}^*$ into specific singular and regular disks as in the previous subsection.
In this subsection, for a given compact Lagrangian submanifold $L$ which is carried by $\mathcal{B}_{\psi}$, we define a sequence of braids $b(\psi^m(L), S_i)$ corresponding to $\psi^m(L)$ over the boundary of each singular disk $S_i$ of $\mathcal{B}_{\psi}^*$.
Lemma \ref{lem3} gives an inductive description of the sequences $b(\psi^m(L), S_i)$. 
We will end this subsection by constructing limits of $b(\psi^m(L), S_i)$ as $m \to \infty$. 
 
For a singular disk $S, \pi^{-1}(\partial S) = \cup_{p \in \partial S} F_p$ is a $\mathbb{D}^n$-bundle over $\partial S \simeq S^{n-1}$. 
Note that we use $\mathbb{D}^n$ to indicate a closed disk, and we will use $\mathring{\mathbb{D}}^n$ to indicate an open disk.
Let $\varphi : \pi^{-1}(\partial S) \stackrel{\sim}{\to} S^{n-1} \times \mathbb{D}^n$ be a bundle map.
If $L$ is a Lagrangian submanifold which is carried by $\mathcal{B}_{\psi}$, then, for all $p \in \partial S$, $\varphi(L \cap F_p)$ is a finite collection of isolated points in $\mathbb{D}^n$; recall that $\pi : L \to \mathcal{B}^*$ has no singular value on $\partial S$.
Thus, $\varphi(L \cap \pi^{-1}(\partial S))$ can be identified with a map from $\partial S \simeq S^{n-1}$ to the configuration space $\operatorname{Conf}_{l}(\mathbb{D}^n)$ of $l$ points on $\mathbb{D}^n$ where $l = l(L,S)$, i.e., a braid.

We explained that $L \cap \pi^{-1}(\partial S)$ could be identified with a braid. 
Since $L$ is a Lagrangian submanifold of $M$, the braid corresponding to $L \cap \pi^{-1}(\partial S)$ satisfies a symplectic property.
The symplectic property is the following: 
For the bundle map $\varphi: \pi^{-1}(\partial S) \stackrel{\sim}{\to} S^{n-1} \times \mathbb{D}^n, (\varphi^{-1})^*(\omega)$ is a 2-form on $S^{n-1} \times \mathbb{D}^n$ such that $(\varphi^{-1})^*(\omega)$ is zero on $\varphi \big(L \cap \pi^{-1}(\partial S) \big)$.

From now on, we will define the braids on the boundary of a singular disk $S$.  
Let $f : S^{n-1} \to \operatorname{Conf}_{l}(\mathbb{D}^n)$ for some $l$. 
In other words, there are maps 
$$f_1, \cdots, f_{l} : S^{n-1} \to \mathbb{D}^n,$$ 
such that $f(p) = \{f_1(p), \cdots, f_{l}(p)\}$ as $f_i(p) \neq f_j(p)$ for all $i \neq j$.
We define 
\begin{align*}
B(f) := \{ (p,& f_i(p)) \in S^{n-1} \times \mathbb{D}^n \hspace{0.2em} | \hspace{0.2em} i \in \{1, \cdots, \ell\} \},\\
\tilde{Br}_{\partial S}:= \{ \varphi^{-1}\big(B(&f)\big) \hspace{0.2em} | \hspace{0.2em} f: S^{n-1} \to \operatorname{Conf}_{l}(\mathbb{D}^n) \text{ for some } l \text{  such that},\\
& (\varphi^{-1})^*(\omega) \text{  is a zero on  } B(f) \}.
\end{align*}
Note that $\tilde{Br}_{\partial S}$ is a set of closed subsets of $\pi^{-1}(\partial S)$ and independent of $\varphi$. 

We define an equivalence relation on $\tilde{Br}_{\partial S}$ as follows:
$b_0 \sim b_1$ for $b_i \in \tilde{Br}_{\partial S}$ if there exists a smooth 1-parameter family $b_t \in \tilde{Br}_{\partial S}$ connecting $b_0$ and $b_1$.
Let $Br_{\partial S} := \tilde{Br}_{\partial S}/\sim$. 

\begin{definition}
	Let $\mathcal{B} \in \mathbb{B}$ and let $S$ be a singular disk of $\mathcal{B}$.
	If $L$ is a Lagrangian submanifold which is carried by $\mathcal{B}$,
	then the {\em braid $b(L,S)$} of $L$ on $S$ is the braid isotopy class of $Br_{\partial S}$ which is given by
	$$ b(L,S) = \big[L \cap \pi^{-1}(\partial S)\big] \in Br_{\partial S}.$$
\end{definition} 
Recall that $\mathbb{B}$ is a set of Lagrangian branched submanifold defined in Section 3.4 and for any $\mathcal{B} \in \mathbb{B}$, we decompose $\mathcal{B}$ into the union of specific singular disks and regular disks, introduced in Section 4.1.

\begin{lemma}
	\label{lem3}
	Let $L$ be a Lagrangian submanifold of $M$ which is carried by $\mathcal{B}$.
	For a given singular disk $S$ of $F_{\tau_i}(\mathcal{B})$ (resp.\ $F_{\sigma_j^{-1}}(\mathcal{B})$), there exist maps $f_k$ from $\tilde{Br}_{S_{i_k}}$ to $\tilde{Br}_{S}$, where $S_{i_k}$ is a singular disk of $\mathcal{B}$, and there exist closed sets $\mathring{b}_{i_k} \in \tilde{Br}_{S_{i_k}}$, such that $b(\tau_i(L),S)$ (resp.\ $b(\sigma_j^{-1}(L),S)$) is $\big[\bigsqcup_k f_{k}(\mathring{b}_{i_k})\big] \in Br_{\partial S}$.
\end{lemma}
Recall the functions $F_{\tau_i}$ and $F_{\sigma_j^{-1}}$ in Lemma \ref{lem3} are defined in Lemma \ref{lem1}.

\begin{proof}[Proof of Lemma \ref{lem3}]
In Steps 1--3, we prove Lemma \ref{lem3} for a particular example; this is just for notational simplicity.
In Step 4, we briefly describe how to prove the general case.

The example we consider is the Lagrangian branched submanifold $\mathcal{B}_{\psi}$ in $M = P(\alpha, \beta_1, \beta_2)$, 
where $\alpha$ and $\beta_j$ are spheres such that $\alpha \cap \beta_1 = \{p\}$ and $\alpha \cap \beta_2 = \{q\}, \tau_0$ and $\sigma_j$ are Dehn twists along $\alpha$ and $\beta_j$, and $\psi = \tau_0 \circ \sigma_1^{-1} \circ \sigma_2^{-1}$.
Then, $\mathcal{B}_{\psi}$ is given by Theorem \ref{branched surface thm}. 

\vskip0.2in
\noindent{\em Step 1 (Notation).}
First, we will choose $\varphi:\pi^{-1}(\partial S) \stackrel{\sim}{\to} S^{n-1} \times \mathbb{D}^n$ for $S = S_p^\pm, S_q^\pm, \bar{S}_p^\pm$, and $\bar{S}_q^{\pm}$.
We will use $\varphi$ in the next steps. 

In order to construct $\varphi: \pi^{-1}(\partial S_p^+) \stackrel{\sim}{\to} S^{n-1} \times \mathbb{D}^n$, we observe that 
$$\pi^{-1}(S_p^+) \cap \mathcal{B} \subset D_p^+,$$
by Remark \ref{rmk identify with cotangent bundle of disk}. 
Moreover, we can assume that $\pi^{-1}(S_p^+) \subset i_{D_p^+}\big(\mathcal{N}(D_p^+)\big)$.
Note that $i_{D_p^+}$ and $\mathcal{N}(D_p^+)$ are defined in Remark \ref{rmk natural embedding}.
Thus, by choosing coordinate charts for $D_p^+$, one obtains $\varphi: \pi^{-1}(S_p^+) \stackrel{\sim}{\to} \mathbb{D}^n \times \mathbb{D}^n$.
By abuse of notation, the restriction $\varphi|_{\pi^{-1}(\partial S_p^+)} : \pi^{-1}(\partial S_p^+) \stackrel{\sim}{\to} S^{n-1} \times \mathbb{D}^n$ is simply called $\varphi$ again.  
Similarly, it is enough to choose coordinate charts for $D_p^-, D_q^\pm, \bar{D}_p^{\pm}, \bar{D}_q^{\pm}$, in order to fix $\varphi : \pi^{-1}(\partial S) \stackrel{\sim}{\to} S^{n-1} \times \mathbb{D}^n$ for $S = S_p^-, S_q^\pm, \bar{S}_p^\pm, \bar{S}_q^\pm$. 

In order to choose specific coordinate charts for $D_p^\pm, D_q^\pm, \bar{D}_p^\pm$, and $\bar{D}_q^\pm$, we use the $(2n-2)$-dimensional submanifold $W \subset M$ defined in Section 4.1. 
For convenience, we consider the lowest nontrivial dimension, i.e., $n=2$. 
For higher $n$, we can fix coordinate charts similarly.

Let $(x_1,x_2)$ be a coordinate chart on $D_p^+ \subset \alpha$ such that the $x_1$-axis agrees with $W \cap D_p^+$.
There are two choices for the positive $x_1$-direction corresponding to the two orientations of $W \cap D_p^+$, or equivalently orientations of $C_{\alpha}$.
We can choose either of them.
Then, let $(y_1,y_2)$ be an oriented chart on $D_p^-$ such that the $y_1$-axis agrees with $W \cap \beta_1$ and $\omega(\partial_{x_1}, \partial_{ y_1})>0$. 
The positive $y_1$-direction determines an orientation of $C_{\beta_1}$. 
On $\bar{D}_p^+$, there exists an oriented chart $(x_1,x_2)$ such that the positive $x_1$-direction agrees with the orientation of $C_{\alpha}$.  
For the other singular disks, we obtain oriented coordinate charts from the orientations of $C_{\alpha}, C_{\beta_i}, \alpha$ and $\beta_i$ in the same way.

Let $b_1 = b(L,S_p^+), b_2 = b(L,\bar{S}_p^+), b_3 = b(L,\bar{S}_p^-), b_4 = b(L,S_q^+ ), b_5 = b(L,\bar{S}_q^+)$, and $b_6 = b(L,\bar{S}_q^-)$, and let $\mathring{b}_i$ be a representative of $b_i$.  

The boundaries of $S_p^+$ is a component of the branch locus of $\mathcal{B}_{\psi}^*$.
By Remark \ref{rmk nonuniqueness of fibered neigbhrohood} (3), one can decompose $\mathring{b}_1$.
More precisely, in this case, Remark \ref{rmk nonuniqueness of fibered neigbhrohood} says that for any $x \in \partial S_p^+$, there are three sectors $S_0, S_1, S_2$ such that 
\begin{gather*}
x \in S_i \text{  for all  } i = 0, 1, 2, \\
F_x \cap \overline{\pi^{-1}(\mathring{S}_0)} = F_x \text{  and  } F_x \cap \overline{\pi^{-1}(\mathring{S}_i)} \subset \mathring{F}_x \text{  for  } i = 1, 2.
\end{gather*}
Moreover, it is easy to check that $S_p^+$ is either $S_1$ or $S_2$. 
Without loss of generality, let us label $S_1 = S_p^+$. 

If $L$ is carried by $\mathcal{B}$, we assume that $L \subset N(\mathcal{B})$.
Then, one obtains 
$$L \cap F_x \subset \big(F_x \cap \overline{\pi^{-1}(\mathring{S}_1)}\big) \cup \big(F_x \cap \overline{\pi^{-1}(\mathring{S}_2)}\big)$$ 
We decompose $\mathring{b}_1$ into $\mathring{b}_1 = \tilde{b}_1 \sqcup \bar{b}_1$, where $\tilde{b}_1 = \mathring{b}_1 \cap \overline{\pi^{-1}(\mathring{S}_1)}$ and $\bar{b}_1 = \mathring{b}_1 \cap \overline{\pi^{-1}(\mathring{S}_2)}$. 
The decomposition $\mathring{b}_4 = \bar{b}_4 \sqcup \tilde{b}_4$ is similar.
	
	We will explain the effects of $\sigma_2^{-1}$ on $\mathcal{B}_{\psi}$ in Step 2 and $\tau_0$ on $\mathcal{B}_{\psi}$ in Step 3.
	The effect of $\sigma_1^{-1}$ is similar to that of $\sigma_2^{-1}$.

	\vskip0.2in	
	\noindent{\em Step 2 (Effect of $\sigma_2^{-1}$ on $\mathcal{B}_{\psi}$).}
	In the rest of this paper, we make specific choices of $\tau_0$ and $\sigma_j$ which are given by Equation \eqref{eqn definition of generalized Dehn along L}, and $\tau : T^*S^2 \stackrel{\sim}{\to} T^*S^2$, which is defined in Remark \ref{rmk specific dehn twists}.  
	In other words, $\tau_0 = \phi_{\alpha} \circ \tau \circ \phi_{\alpha}^{-1}$ and $\sigma_j = \phi_{\beta_j} \circ \tau \circ \phi_{\beta_j}^{-1}$, where $\phi_{\alpha}$ (resp.\ $\phi_{\beta_j}$) is a symplectomorphism from $T^*S^2$ to a neighborhood of $\alpha$ (resp.\ $\beta_j$).
	The neighborhood of $\alpha$ (resp.\ $\beta_j$) will be denoted by $T^*\alpha$ (resp.\ $T^*\beta_j$). 
	
	\begin{remark}
		\label{rmk acts like antipodal}
		Recall that $\tau$ is a Dehn twist on $T^*S^n$ which agrees with the antipodal map 
		$$ T^*S^n \stackrel{\sim}{\to} T^*S^n, (u;v) \mapsto (-u;-v),$$
		on a neighborhood of the zero section $S^n$. 
	\end{remark}
	
	By Lemma \ref{lem2}, $\sigma_2^{-1}(L)$ is carried by $\mathcal{B}' = F_{\sigma_2^{-1}}(\mathcal{B}_{\psi})$.
	We label 
	\begin{gather*}
	b_1'=b(\sigma_2^{-1}(L),S_p^+), b_2'=b(\sigma_2^{-1}(L),\bar{S}_p^+), b_3'=b(\sigma_2^{-1}(L),\bar{S}_p^-),\\
	b_4'=b(\sigma_2^{-1}(L),S_q^-), b_5'=b(\sigma_2^{-1}(L),\bar{S}_q^+), b_6' = b(\sigma_2^{-1}(L),\bar{S}_q^-).
	\end{gather*}
	Note that the singular disk for $b_4$ is $S_q^+$ and the singular disk for $b_4'$ is $S_q^-$, i.e., two singular disks have the same center but different sign.
	However, for $i \neq 4$, the singular disks for $b_i$ and $b_i'$ have the same center and the same sign. 
	
	For convenience, the singular disk of $\mathcal{B}_{\psi}$ (resp.\ $F_{\sigma_2^{-1}}(\mathcal{B}_{\psi})$) will be called $S_i$ (resp.\ $S_i'$), so that $b_i$ (resp.\ $b'_i$) is a braid on $\pi^{-1}(\partial S_i)$ (resp.\ $\pi^{-1}(\partial S_i')$).
	Also, let $\varphi_i : \overline{\pi^{-1}(\mathring{S}_i)} \stackrel{\sim}{\to} \mathbb{D}^2 \times \mathring{\mathbb{D}}^2$ (resp.\ $\varphi_i' : \overline{\pi^{-1}(\mathring{S}'_i)} \stackrel{\sim}{\to} \mathbb{D}^2 \times \mathring{\mathbb{D}}^2$) be the identification which is fixed in Step 1.
	
	Since $\sigma_2^{-1}$ is supported on $T^*\beta_2$, a small neighborhood of $\beta_2$, $b_i$ and $b_i'$ are the same braid in $Br_{\partial S_i}$ for $i = 1, 2, 3$, and $5$. 
	We will explain how $b_6'$ is constructed.
	
	We can obtain $\sigma_2^{-1}(\mathcal{B}_{\psi})$ by spinning with respect to $q$ in $T^*\beta_2$, i.e., $\sigma_2^{-1}(\mathcal{B}_{\psi})$ is the union of curves in 2-dimensional submanifold $\phi_{\beta_2}(W_y)$ over $y \in S^1$.
	Recall that the spinning and $W_y$ are defined in Section 2.2. 
	\begin{figure}[h]
		\centering
		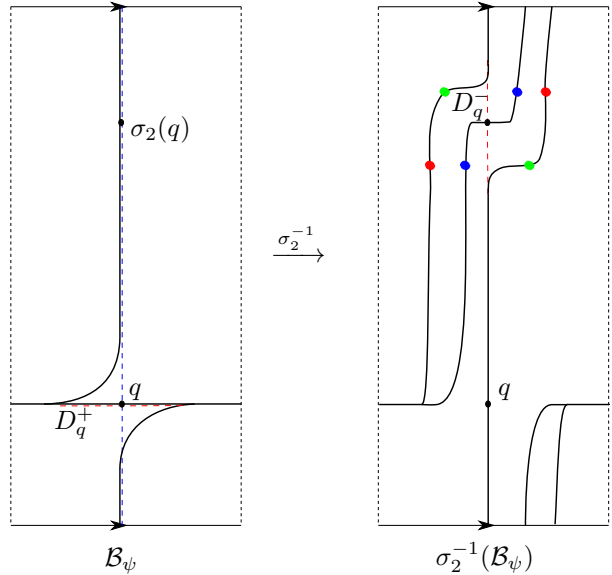
		\caption{The left picture represents $\mathcal{B}_{\psi} \cap \phi_{\beta_2}(W_y)$ and the right picture represents $\sigma_2^{-1}(\mathcal{B}_{\psi}) \cap \phi_{\beta_2}(W_y)$.}
		\label{RBG}
	\end{figure}

	Figure \ref{RBG} represents $\mathcal{B}_{\psi} \cap \phi_{\beta_2}(W_y)$ and $\sigma_2^{-1}(\mathcal{B}_{\psi}) \cap \phi_{\beta_2}(W_y)$ on $\phi_{\beta_2}(W_y)$.
	We obtain Figure \ref{RBG} because we choose specific $\sigma_2$.
	 
	By spinning blue, red, and green points in Figure \ref{RBG}, we obtain $\sigma_2^{-1}(\mathcal{B}_{\psi}) \cap \pi^{-1}(\partial S_6')$. 
	Let $B, R$, and $G$ be the circles obtained by spinning blue, red, and green points respectively.
	
	Since $N(\mathcal{B}_{\psi}) \supset \mathcal{B}_{\psi}$, 
	$\sigma_2^{-1}\big(N(\mathcal{B}_{\psi})\big) \cap \pi^{-1}(\partial S_6')$ is a neighborhood of $\sigma_2^{-1}(\mathcal{B}_{\psi}) \cap \pi^{-1}(\partial S_6')$.
	By assuming that $N(\mathcal{B}_{\psi})$ is a sufficiently small neighborhood of $\mathcal{B}_{\psi}$, $\sigma_2^{-1}\big(N(\mathcal{B}_{\psi})\big)  \cap \pi^{-1}(\partial S_6')$ consists of three connected components, which are neighborhoods of $B, R$, and $G$. 
	Each connected component will be called $N(B), N(R)$, and $N(G)$.  
	
	By definition, $b_6' = \big[\sigma_2^{-1}(L) \cap \pi^{-1}(\partial S_6') \big]$.
	Without loss of generality, we assume that $L \subset N(\mathcal{B}_{\psi})$. 
	Then, 
	$$\sigma_2^{-1}(L)  \cap \pi^{-1}(\partial S_6') \subset \sigma_2^{-1}\big(N(\mathcal{B}_{\psi})\big)  \cap \pi^{-1}(\partial S_6') = N(B) \sqcup N(R) \sqcup N(G).$$
	Thus, strands of $\sigma_2^{-1}(L) \cap \pi^{-1}(\partial S_6')$, or equivalently $b_6'$, are divided into three groups, which are contained in $N(B), N(R)$, and $N(G)$ respectively.
	We argue the group which is contained in $N(B)$ first.
	 
	The group of strands in $N(B)$ is given by $\sigma_2^{-1}(L) \cap N(B)$.
	Thus, we will consider $\sigma_2\big( \sigma_2^{-1}(L) \cap N(B) \big) = L \cap \sigma_2\big(N(B)\big)$.
	One of the main difficulties is that the action of $\sigma_2^{-1}$ on $\sigma_2\big(N(B)\big)$ is not simple. 
	To make it simpler, we will construct a Hamiltonian isotopy $\Phi_t$, so that there is a disk $D_B \subset S_q^+$ such that 
	$$(\Phi_1 \circ \sigma_2^{-1})\big(\pi^{-1}(\partial D_B) \big) \subset \pi^{-1}(\partial S_6').$$
	Then, $(\Phi_1 \circ \sigma_2^{-1})\big(\pi^{-1}(\partial D_B) \cap L\big)$ corresponds to the group of strands in $N(B)$.
	
	We construct $\Phi_t$ as follow:
	Let $H_t:\mathbb{R}^4 \to \mathbb{R}^4$ be a Hamiltonian isotopy given by
	\begin{align*}
	H_t = \begin{pmatrix}
	\cos t& 0 & -\sin t  &0 \\ 
	0& \cos t  &0  &- \sin t \\ 
	\sin t& 0 & \cos t & 0\\ 
	0& \sin t & 0 & \cos t
	\end{pmatrix},
	\end{align*} 
	and let $\delta : [0,\infty) \to \mathbb{R}$ be a smooth decreasing function such that $\delta(x) = \tfrac{\pi}{2}$ for all $x <1$ and $\delta(x)= 0$ for all $x > 2$. 
	We choose a neighborhood $U \subset \beta_2$ of $\sigma_2^{-1}(q)$ and a Darboux chart $\phi_q: T^*U \stackrel{\sim}{\to} \mathbb{R}^4$ such that $\phi_q(\sigma_2^{-1}(q))$ is the origin. 
	We remark that $T^*\beta_2$ denotes a neighborhood of $\beta_2$ in $M$, which is symplectomorphic to the cotangent bundle of $\beta_2$.
	Thus, for a subset $U$ of $\beta_2$, one can assume that $T^*U$ is a subset of $M$. 
	
	For convenience, let $\phi_q(x) = (x_1;x_2)$ where $x_i \in \mathbb{R}^2$.
	Then, there is a Hamiltonian isotopy 
	\begin{align}
	\label{eqn Hamiltonian isotopy}
	 \Phi_t(x) = \left\{\begin{matrix}
	(\phi_q^{-1} \circ H_{t \delta(c_1\|x_1\|+c_2\|x_2\|)} \circ \phi_q) (x) &\text{  if  } x \in T^*U, \\ 
	x \hspace{2em}  &\text{   if  } x \notin T^*U,
	\end{matrix}\right.
	\end{align}
	where $c_i$ is a positive constant and $\|\cdot \|$ is the standard norm on $\mathbb{R}^2$. 
	
	\begin{figure}[h]
		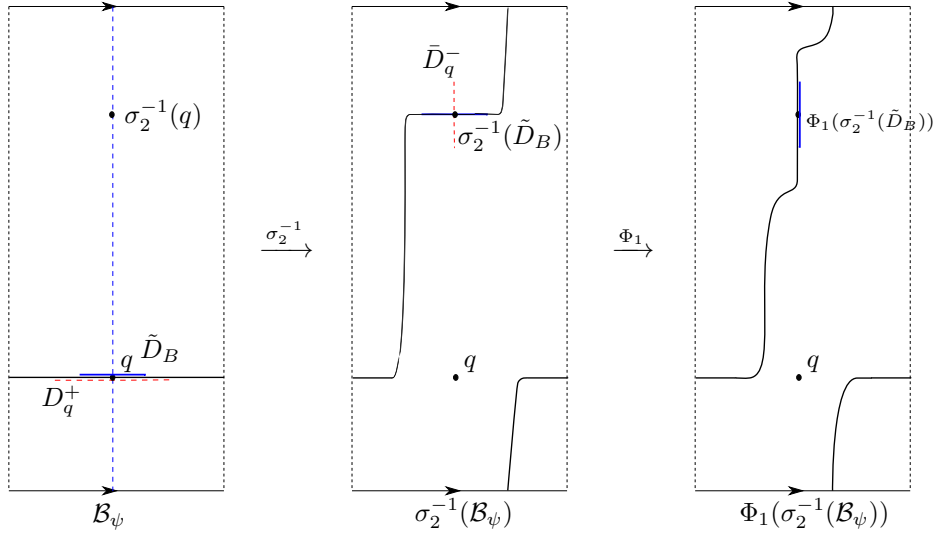
		\caption{The blue curves represent $\tilde{D}_B \cap \phi_{\beta_2}(W_y)$ in the left picture, $\sigma_2^{-1}(\tilde{D}_B) \cap \phi_{\beta_2}(W_y)$ in the middle picture, and $\Phi_1(\sigma_2^{-1}(\tilde{D}_B)) \cap \phi_{\beta_2}(W_y)$ in the right picture.}
		\label{afterhomotopy}
	\end{figure}
	To visualize, we use $D_q^+$ and $\bar{D}_q^-$ instead of $S_4$ and $S'_6$ in Figure \ref{afterhomotopy}.
	Figure \ref{afterhomotopy} represents $\phi_{\beta_2}(W_y) \cap D_q^+, \phi_{\beta_2}(W_y) \cap \sigma_2^{-1}(D_q^+)$ and $\phi_{\beta_2}(W_y) \cap \Phi_1(\sigma_2^-(D_q^+))$ in the left, middle, and right pictures respectively. 
	By choosing proper $c_i$, we obtain a small disk $D_B \subset S_q^+$ such that $(\Phi_1 \circ \sigma_2^{-1})\big(\pi^{-1}(\partial D_B) \big) \subset \pi^{-1}(\partial S_6')$. 
	More precisely, we obtain a disk $\tilde{D} \subset D_q^+$ which is in blue in the left of Figure \ref{afterhomotopy}. 
	Blue curves in the middle and right of Figure \ref{afterhomotopy} represent $(\pi \circ \sigma_2^{-1})\big(\tilde{D}_B\big)$ and $(\Phi_1 \circ \sigma_2^{-1})\big(\pi(\tilde{D}_B)\big)$.
	Then, $D_B$ is given by $D_B:= \pi(\tilde{D}_B)$. 

	On a small neighborhood of $D_B, \sigma_2^{-1}$ agrees with the antipodal map of $\phi_{\beta_2}(T^*\beta_2) \simeq T^*S^2$, as we mentioned in Remark \ref{rmk acts like antipodal}. 
	Then, we obtain a map
	\begin{gather*}
	f_1 : S^1 \times \mathring{\mathbb{D}}^2 \stackrel{\tilde{\varphi}_4^{-1}}{\simeq} \pi^{-1}(\partial D_B) \xrightarrow{\Phi_1 \circ \sigma_2^{-1}} \pi^{-1}(\partial S_6') \stackrel{\varphi_6'}{\simeq} S^1 \times \mathbb{D}^2,\\
	(\theta, x, y) \mapsto (\theta + \pi, -r_1x, -r_1y ).
	\end{gather*}
	The first identification $\tilde{\varphi}_4$ is the restriction of $\varphi_4: \pi^{-1}(S_4) \stackrel{\sim}{\to} \mathbb{D}^2 \times \mathbb{D}^2$. 
	\begin{remark}
		\label{rmk radius of solid torus}
		\mbox{}
		\begin{enumerate}
			\item Note that $\varphi_6'\big(\operatorname{Im}(f_1)\big) = (\Phi_1 \circ \sigma_2^{-1})\big(\pi^{-1}(\partial D_B)\big) \cap \pi^{-1}(\partial S_6')$. 
			Similarly, for the groups of strands in $N(R)$ and $N(G)$, one can obtain two functions $f_2$ and $f_3$ on $S^1 \times \mathbb{D}^2$ in the same way. 
			Then, the images $\operatorname{Im}(f_2)$ and $\operatorname{Im}(f_3)$ correspond to 
			$$(\Phi_1 \circ \sigma_2^{-1})\big(\pi^{-1}(\pi(N_q))\big) \cap \pi^{-1}(\partial S_6') \subset \pi^{-1}(\partial S_6') \stackrel{\varphi_6'}{\simeq} S^1 \times \mathbb{D}^2.$$ 
			Thus, $f_1$ explains the contribution of $\tilde{b}_4$, and $f_2$ and $f_3$ explain the contribution of $\bar{b}_4$ on the construction of $b_6'$.			
			\item The constant $r_1$ is determined by specific choices of an identification $\phi_{\beta_2}: T^*S^2 \stackrel{\sim}{\to} T^*\beta_2$, the fixed Dehn twist $\tau$ in Remark \ref{rmk specific dehn twists}, and so on. 
			However, $r_1$ has to be smaller than 1. 
			This is because $\operatorname{Im}(f_1), \operatorname{Im}(f_2)$, and $\operatorname{Im}(f_3)$ are mutually disjoint, since they corresponds to $N(B), N(R)$, and $N(G)$ respectively.
		\end{enumerate}

	\end{remark}
		
	The strands of $b'_6$ which are contained in $N(B)$ correspond to $$\varphi_6'^{-1}(f_1(\tilde{\varphi}_4(L \cap \pi^{-1}(\partial D_B)))).$$ 
	We will prove that $L \cap \pi^{-1}(\partial D_B)$ represents the same braid with $\tilde{b}_4$. 
	We can assume that there is no singular value of $\pi$ on $S_4 \setminus D_B$. 
	Then, $\varphi_4(\tilde{b}_4)$ and $\tilde{\varphi}_4(L \cap \pi^{-1}(\partial D_B))$ represent the same braid in $S^1 \times \mathbb{D}^2$ because of non-singularity on $S_4 \setminus D_B$. 
	Thus, in $S^1 \times \mathbb{D}^2$, $f_1(\varphi_4(\tilde{b}_4))$ and $f_1(\tilde{\varphi}_4(L \cap \pi^{-1}(\partial D_B)))$ represent the same braid.
	It proves that $\varphi_6'^{-1}(f_1(\varphi_4(\tilde{b}_4)))$ and the group of strands in $N(B)$ represent the same braid in $Br_{\partial S_6'}$. 

\begin{remark}
	\label{rmk abuse of notation}
	For convenience, we simply use $f_1(\tilde{b}_4)$, instead of $\varphi_6'^{-1}(f_1(\varphi_4(\tilde{b}_4)))$.
	In the rest of this paper, we will abuse notation in the same way.
\end{remark}

	For the groups of strands in $N(R)$ and $N(G)$, we obtain the following maps $f_2$ and $f_3$ in the same way,
	\begin{gather*}
	f_2: S^1 \times \mathbb{D}^2 \to S^1 \times \mathbb{D}^2, \\
	(\theta, x, y) \mapsto (\theta + \pi, r_0\cos \theta+ r_2 x, r_0\sin \theta  + r_2y),\\
	f_3: S^1 \times \mathbb{D}^2 \to S^1 \times \mathbb{D}^2, \\
	(\theta, x, y) \mapsto (\theta + \pi, -r_0\cos \theta+ r_2(x\cos 2\theta  - y\sin 2\theta ),\\
	\hspace{7em} -r_0 \sin \theta+ r_2(x\sin 2 \theta  + y \cos 2\theta)),
	\end{gather*}
	where $r_0$ and $r_2$ are positive constants which are smaller than 1.
	\begin{remark}
		\mbox{}
		\label{rmk radius of solid torus 2}
		\begin{enumerate}
			\item To obtain $f_1$, we used a Hamiltonian isotopy $\Phi_t$. 
			Similarly, to obtain $f_2$ and $f_3$, we need a Hamiltonian isotopy.
			We construct a Hamiltonian isotopy by extending a Lagrangian isotopy connecting $\sigma_2^{-1}(N_q) \cap \overline{\pi^{-1}(S_6')}$ and  
			$$\varphi_6'^{-1}(\{ (s \cos (\theta + \pi), s \sin (\theta+\pi), r_0 \cos \theta, r_0 \sin \theta) \hspace{0.2em} | \hspace{0.2em} s \in [-1,1], \hspace{0.5em} \theta \in S^1\}),$$
			in $\overline{\pi^{-1}(\mathring{S}_6')} \stackrel{\varphi_6'}{\simeq} \mathbb{D}^2 \times \mathbb{D}^2$.
			\item Note that $r_0$ and $r_2$ are positive constants which are determined by specific choices.
			However, $r_0$ and $r_2$ have to satisfy $r_1 + r_2 < r_0$ since $\operatorname{Im}(f_1), \operatorname{Im}(f_2)$ and $\operatorname{Im}(f_3)$ are mutually disjoint.
		\end{enumerate}
			\end{remark}

	In the same way that we proved that $f_1(\tilde{b}_4)$ and the group of strands in $N(B)$ represent the same braid in $Br_{\partial S_6'}$, we can prove that $f_2(\bar{b}_4)$ (resp.\ $f_3(\bar{b}_4)$) and the group of strand in $N(R)$ (resp.\ $N(G)$) represent the same braid in $Br_{\partial S_6'}$.  
	Then, $b_6'$ is represented by $f_1(\tilde{b}_4) \sqcup f_2(\bar{b}_4) \sqcup f_3(\bar{b}_4)$.
	Note that we are abusing notation for convenience as we mentioned in Remark \ref{rmk abuse of notation}.
	
	The situation for $b_4'$ is analogous. 
	We obtain three maps $g_1, g_2$ and $g_3$ in the same way.
	At the end, $b_4'$ is represented by $g_1(\bar{b}_4) \sqcup g_2(\bar{b}_4) \sqcup g_3(b_6)$. 
	This proves Lemma \ref{lem3} for the case of $\sigma_2^{-1}$.
	
	Note that maps $f_i$ and $g_j$ are given by specific maps acting on $S^1 \times \mathbb{D}^2$, but we would like to consider them as maps on $\tilde{Br}_{\partial S_k}$ for some $k$.
	Then, we summarize the effect of $\sigma_2^{-1}$ as a matrix 
	\begin{align*}
	\Sigma_{2,\mathcal{B}_{\psi}} = \begin{pmatrix}
	id & 0 & 0 & 0 & 0 & 0 \\ 
	0 & id  & 0  & 0 & 0 & 0 \\ 
	0 & 0 & id & 0 & 0 & 0\\ 
	0 & 0 & 0 & g_1 + g_2 & 0 & g_3 \\
	0 & 0 & 0 & 0 & id & 0 \\
	0 & 0 & 0 & f_1 + f_2 + f_3 & 0 & 0 
	\end{pmatrix}.
	\end{align*}  
	Thus, if $\mathring{b}_i$ is a representative of a braid $b_i$ for $L$, then $\mathring{b}_i'$ is a representative of $b_i'$ where 
	\begin{gather*}
	\begin{pmatrix}
	\mathring{b}_1' \\
	\mathring{b}_2' \\
	\mathring{b}_3' \\
	\mathring{b}_4' \\
	\mathring{b}_5' \\
	\mathring{b}_6' 
	\end{pmatrix}
	 = \Sigma_{2,\mathcal{B}_{\psi}} 
	 \begin{pmatrix}
	 \mathring{b}_1 \\
	 \mathring{b}_2 \\
	 \mathring{b}_3 \\
	 \mathring{b}_4 \\
	 \mathring{b}_5 \\
	 \mathring{b}_6 
	 \end{pmatrix} 
	 = \begin{pmatrix}
	 \mathring{b}_1 \\
	 \mathring{b}_2 \\
	 \mathring{b}_3 \\
	 g_1(\bar{b}_4) \sqcup g_2(\bar{b}_4) \sqcup g_3(\mathring{b}_6) \\
	 \mathring{b}_5 \\
	 f_1(\tilde{b}_4) \sqcup f_2(\bar{b}_4) \sqcup f_3(\bar{b}_4) 
	 \end{pmatrix}.
	\end{gather*}
	
	\begin{remark}
		\label{rmk non-linear algebra}
		We remark that in surface theory, we can do linear algebra on weights, but in a higher-dimensional case, we cannot do linear algebra with the matrix $\Sigma_{2,\mathcal{B}_{\psi}}$, because there is no module structure on $\tilde{Br}_{\partial S_i}$.
	\end{remark}	
\vskip0.2in

	\noindent
	{\em Step 3 (Effects of $\tau_0$ on $\mathcal{B}_{\psi}$).} We use the same notation, i.e.,  $b_1, \cdots, b_6$ denote the braids on singular disks $S-i$ of $\mathcal{B}_{\psi}^*$, and 
	$$b_1' = b(\tau_0(L),S_p^+), \cdots, b_6' = b(\tau_0(L),\bar{S}_q^-),$$ 
	so that the singular disk corresponding to $b'_i$ has the same center as the singular disk corresponding to $b_i$.
	We also use $\mathring{b}_i$ and $\mathring{b}_i'$, $S_i$ and $S_i'$, $\varphi_i$ and $\varphi_i'$ to indicate representatives of braids, singular disks in $\mathcal{B}_{\psi}$ and $F_{\tau_0}(\mathcal{B}_{\psi})$, identifications induced by fixed coordinate charts.
	
	The situation for $\tau_0$ is similar to that for $\sigma_2^{-1}$.
	For example, by observing how $\tau_0$ acts on $\overline{\pi^{-1}(\mathring{S}_1)}$, we obtain 
	$$h_1: S^1 \times \mathbb{D}^2 \to S^1 \times \mathbb{D}^2,$$ 
	explaining the contribution of $\tilde{b}_1$ on the construction of $b_3'$.
	Then, $h_1$ is given by a translation on $S^1$ and a scaling on $\mathbb{D}^2$, as $f_1$ is.
	Similarly, we obtain $h_2$ and $h_3$, which explain the contributions of $\bar{b}_1$ on the construction of $b_3'$. 
	The map $h_2$ (resp.\ $h_3$) is of the same types with $f_2$ (resp.\ $f_3$), i.e., 
	\begin{gather*}
	h_2(\theta, x, y) = \big(\theta \text{ or } \theta + \pi, \pm r_1 \cos \theta + r_2 x, \pm r_1 \sin \theta + r_2 y \big), \\
	h_3(\theta, x, y) = \big(\theta \text{ or } \theta + \pi, \pm r_1 \cos \theta + r_2 (x \cos 2\theta - y \sin 2 \theta), \\
	\hspace{7em} \pm r_1 \sin \theta + r_2 (x \sin 2\theta + y \cos 2 \theta) \big),
	\end{gather*}
	where $r_1$ and $r_2$ are constants.
	 
	If a map is of the same type to $f_1$, in other words, if the map is given by a translation on $S^1$ and a scaling on $\mathbb{D}^2$, let the map be of {\em scaling type}.
	This is because the formula defining the map is given by a scaling on fibers.
	The maps of scaling type explain how braids $b(L, S_p^\pm)$ or $b(L,\bar{S}_p^\pm)$ contribute on the construction of braids $b(\delta(L), S_{\delta(p)}^\pm)$ or $b(\delta(L),\bar{S}_{\delta(p)}^\pm)$ through $\delta\big(\pi^{-1}(S_p^\pm)\big)$, where $\delta$ is a Dehn twist.
	
	If a map is of the same type to $f_2$ (resp.\ $f_3$), let the map be of {\em the first (resp.\ second) singular type}.
	This is because they are related to a creation of new singular component.
	The maps of the first and second singular types explain how the braid $b(L,\delta(S_p))$ contributes on the construction of braid $b(\delta(L),\bar{S}_{\delta(p)}^\pm)$.
	
	To summarize, if $b_i$ contributes the construction of $b'_j$ and if the center of a singular disk corresponding to $b_i$ is either the same point or the antipodal point of the center of the singular disk corresponding to $b_j'$, maps of these three types explain the contribution of $b_i$ on the construction of $b_j'$. 
	Note that the center of a singular disk is defined in Remark \ref{rmk position of singular value}.
	
	The maps of these three types explain the effects of $\sigma_2^{-1}$ on $\mathcal{B}$.
	However, to explain the effects of $\tau_0$ on $\mathcal{B}_{\psi}$, we need maps of one more type.  	
			
	This is because $\alpha$ has two plumbing points, unlike $\beta_i$ has only one plumbing point.
	Thus, when we apply $\tau_0$, $b_i$ can contribute on $b_j'$ even if the centers of singular disks corresponding to $b_i$ and $b_j'$ are neither the same nor antipodals of each other. 
	For example, $L \cap \pi^{-1}\big(\pi(N_p)\big)$ is stretched by $\tau_0$.
	The stretched part $\tau_0\big(L \cap \pi^{-1}(\pi(N_p))\big)$ has intersection with $\pi^{-1}(S_4)$ and $\pi^{-1}(S_5)$. 
	Thus, $b_4'$ has some strands corresponding to $\tau_0(L \cap \pi^{-1}(\pi(N_p))) \cap \pi^{-1}(\partial S_4)$
	These strands are the contribution of $\bar{b}_1$ on the construction of $b_4'$.
	Similarly, $\bar{b}_1$ contributes on the construction of $b_5'$, and $\bar{b}_4$ contributes on the constructions of $b_1'$ and $b_2'$.
	
	To describe the contribution of $\bar{b}_1$ on $b_4'$, without loss of generality, we assume that there is no singular value for 
	$$\tau_0(L \cap \pi^{-1}(\pi(N_p))) \cap \overline{\pi^{-1}(\mathring{S}_4)} \stackrel{\pi}{\to} S_4,$$ 
	by Remark \ref{rmk position of singular value}. 
	Thus, $\tau_0(L \cap \pi^{-1}(\pi(N_p))) \cap \overline{\pi^{-1}(\mathring{S}_4)}$ is a union of disjoint Lagrangian disks on $\overline{\pi^{-1}(\mathring{S}_4)}$
	and $\bar{b}_1$ contributes on $b_4'$ by adding strands near $\tau_0(N_p) \cap \pi^{-1}(\partial S_4)$ which are not braided to each other.
	The number of the added strands is the same as the number of strands of $\bar{b}_1$. 
	In the same way, $\bar{b}_1$ contributes on the construction of $b_5'$.
	
	To describe the contribution of $\bar{b}_1$ on $b_4'$ as a map acting on $S^1 \times \mathbb{D}^2$, we define $\bar{b}_1^\circ \subset \pi^{-1}(\partial S_1)$ such that 
	$$\varphi_1(\bar{b}_1^\circ) := \{ (\theta, x_0, y_0) \hspace{0.2em} | \hspace{0.2em} \phi_1^{-1}(0,x_0,y_0) \in \bar{b}_1 \} \subset S^1 \times \mathbb{D}^2 \stackrel{\varphi_1}{\simeq} \pi^{-1}(\partial S_1),$$
	which represents a trivial braid having the same number of strands with $\bar{b}_1$.
	This is because we only need the number of the strands in $\bar{b}_1$, not the way $\bar{b}_1$ is braided.
	
	We construct a Hamiltonian isotopy $\Phi_t$ by extending a Lagrangian isotopy connecting $\tau_0(N_p) \cap \pi^{-1}(\partial S_4)$ and 
	$$\varphi_4'^{-1}(\{(s \cos \theta, s \sin \theta, c_1, c_2) \hspace{0.2em} | \hspace{0.2em} s \in [-1,1], \theta \in S^1, c_i \hspace{0.5em} \text{is constants} \}) \subset \pi^{-1}(S_4),$$ 
	as we did in Remark \ref{rmk radius of solid torus 2}.
	Then, one obtains
	\begin{gather*}
	h_t: S^1 \times \mathbb{D}^2 \stackrel{\varphi_1}{\simeq} \pi^{-1}(\partial S_1) \xrightarrow{\Phi_1 \circ \tau_0} \pi^{-1}(\partial S_4) \stackrel{\varphi_4'}{\simeq} S^1 \times \mathbb{D}^2,\\
	(\theta, x, y) \mapsto (\theta, r_0x + c_1, r_0y+c_2),
	\end{gather*}
	where $r_0$ is a positive constant number less than 1. 
	Then, $h_t(\bar{b}_1^{\circ})$ represents the same braid to the strands in $b'_4$, which correspond to $\tau_0(L \cap \pi^{-1}(\pi(N_p)))$. 
	We recall that we are abusing notation as mentioned in Remark \ref{rmk abuse of notation}.
	
	Similarly, if $b_i$ contributes the construction of $b'_j$ and if the center of a singular disk corresponding to $b_i$ is neither the same point nor the antipodal point of the center of the singular disk corresponding to $b'_j$, then the contribution of $b_i$ on $b_j'$ can be described by a map like $h_t$. 
	If a map is of the same type with $h_t$, let the map be of {\em trivial type}, because a map of trivial type adds strands which are not braided with each other.
	
	Then, we can describe the effect of $\tau_0$ on $\mathcal{B}_{\psi}$ as a matrix
	\begin{align*}
	\mathrm{T}_{0,\mathcal{B}_{\psi}} = \begin{pmatrix}
	0 & i & 0 & h_t & 0 & 0 \\ 
	h_1 + h_2 + h_3 & 0  & 0  & i_t & 0 & 0 \\ 
	0 & 0 & id & 0 & 0 & 0\\ 
	h_t & 0 & 0 & 0 & i & 0 \\
	i_t & 0 & 0 & h_1 + h_2 + h_3 & 0 & 0 \\
	0 & 0 & 0 & 0 & 0 & id 
	\end{pmatrix}.
	\end{align*} 
	Among the entries, $h_1, i$, and $id$ are of scaling type, $h_2$ and $h_3$ are of the first and second singular types, and $h_t$ and $i_t$ are of trivial type. 
	\vskip0.2in
	
	\noindent{\em Step 4 (General case).}
	A $\psi$ of generalized Penner type is a product of Dehn twists. 
	In the general case, when we apply $\psi$, each Dehn twist is followed by a Hamiltonian isotopy as $\sigma_2^{-1}$ is followed by $\Phi_t$ in step 2.
	Let $\psi_H = (\Phi_{1,1} \circ \delta_1 ) \circ \cdots \circ (\Phi_{l,1} \circ \delta_l)$, where $\psi = \delta_1 \circ \cdots \circ \delta_l$, $\delta_i$ is a Dehn twist, and $\Phi_{i,t}$ is a Hamiltonian isotopy which follows $\delta_i$.
	
	After applying the Hamiltonian isotopy, the effect of a Dehn twist $\tau_i$ (resp.\ $\sigma_j^{-1}$) on $\mathcal{B} \in \mathbb{B}$ is described by a matrix $\mathrm{T}_{i, \mathcal{B}}$ (resp.\ $\Sigma_{j,\mathcal{B}}$), whose entries are sums of maps of four types.
	As we mentioned in Step 3, the maps of scaling type explain how braids $b(L, S_p^\pm)$ or $b(L,\bar{S}_p^\pm)$ contribute on the construction of braids $b(\delta(L), S_{\delta(p)}^\pm)$ or $b(\delta(L),\bar{S}_{\delta(p)}^\pm)$, where $\delta$ is a Dehn twist.
	Similarly, the maps of the first and second singular types explain how braids $b(L,\delta(S_p))$ contribute on the construction of braid to $b(\delta(L),\bar{S}_{\delta(p)}^\pm)$.
	Finally, the maps of trivial type explain the other cases.
	
	This completes the proof of Lemma \ref{lem3}.
\end{proof}

\vskip0.2in
\noindent {\em Taking the limit of a braid sequence.}
We have obtained braid sequences $\{b(\psi^m(L),S_i)\}_{m \in \mathbb{N}}$, where $L$ is carried by $\mathcal{B}_{\psi}$, and $S_i$ is a singular disk of $\mathcal{B}_{\psi}^*$.
In the rest of this subsection, we construct a limit of $\{b(\psi^m(L),S_i)\}_{m \in \mathbb{N}}$ as $m \to \infty$.

We argue with the above example, i.e., 
$$M = P(\alpha, \beta_1, \beta_2), \psi = \tau_0 \circ \sigma_1^{-1} \circ \sigma_2^{-1}.$$ 
For convenience, let 
$$\mathcal{B} := \mathcal{B}_{\psi},\hspace{0.2em} \mathcal{B}' := F_{\sigma_2^{-1}}(\mathcal{B}),\hspace{0.2em} \mathcal{B}'' := F_{\sigma_1^{-1}}(\mathcal{B}'),$$ 
and let singular disks $S_p^+, \bar{S}_p^+, \bar{S}_p^-, S_q^+, \bar{S}_q^+$, and $\bar{S}_q^-$ of $\mathcal{B}$ be $S_1, \cdots, S_6$. 
Using notation from the proof of Lemma \ref{lem3}, we have matrices $\mathrm{T}_{0,\mathcal{B}''}, \Sigma_{1,\mathcal{B}'}$, and $\Sigma_{2,\mathcal{B}}$.
Then, we obtain $\Psi = \mathrm{T}_{0,\mathcal{B}''} \cdot \Sigma_{1,\mathcal{B}'} \cdot  \Sigma_{2,\mathcal{B}}$ by defining a multiplication of maps as a composition of them. 
Note that a product of two arbitrary matrices is not defined. 
For example, an input of $\Sigma_{2,\mathcal{B}}$ and an output of $\mathrm{T}_{0,\mathcal{B}''}$ are tuples of braids on singular disks of $\mathcal{B}^*$.
Thus, $\Sigma_{2,\mathcal{B}} \cdot \mathrm{T}_{0,\mathcal{B}''}$ is defined.
However, $\mathrm{T}_{\mathcal{B}''} \cdot \Sigma_{2,\mathcal{B}}$ is not defined since an input of $\mathrm{T}_{0,\mathcal{B}''}$ is a tuple of braids on singular disks of $\mathcal{B}^*$, but an output of $\Sigma_{2,\mathcal{B}}$ is a tuple of braids on singular disks of $\mathcal{B}'^*$.

Let $\mathring{b}_i$ be a representative of $b_i = b(L,S_i)$.
If
\begin{gather*}
\label{eqn matrix}
\begin{pmatrix}
\mathring{b}_{1,m} \\
\mathring{b}_{2,m} \\
\mathring{b}_{3,m} \\
\mathring{b}_{4,m} \\
\mathring{b}_{5,m} \\
\mathring{b}_{6,m} 
\end{pmatrix}
:= \Psi^m 
\begin{pmatrix}
\mathring{b}_1 \\
\mathring{b}_2 \\
\mathring{b}_3 \\
\mathring{b}_4 \\
\mathring{b}_5 \\
\mathring{b}_6 
\end{pmatrix},
\end{gather*} 
then $\mathring{b}_{i,m}$ is a representative of $b_{i,m}$.
Thus, in order to keep track of braid sequences $\{b_{i,m}\}_{m \in \mathbb{N}}$, it is enough to keep track of $\Psi^m$. 

Every entry of $\Psi^m$ is a sum of compositions of $3m$-maps.
The image of a composition of $3m$-maps is a solid torus.
By Remarks \ref{rmk radius of solid torus} and \ref{rmk radius of solid torus 2}, the radius of each solid torus appearing in $\Psi^m$ decreases exponentially and converges to zero as $m \to \infty$. 
%Also every strand of $\mathring{b}_{i,m}$ is contained in a disjoint union of solid tori. 

From another view points, we consider $\psi_H$.
Note that $\psi_H$ is defined in step 4 of the proof of Lemma \ref{lem3}. 
The proof of Lemma \ref{lem3} implies that 
$$\mathring{b}_{i,m} \subset \psi_H^m(N(\mathcal{B}_{\psi})) \cap \pi^{-1}(\partial S_i) \text{  for all  } m \in \mathbb{N} \text{  and for all  } i = 1, \cdots, 6.$$ 
Let 
$$B_{i,m} := \psi_H^{m}(N(\mathcal{B}_{\psi})) \cap \pi^{-1}(\partial S_i).$$
Then, $B_{i,m}$ is the disjoint union of solid tori.
More precisely, each solid torus in $B_{i,m}$ is the image of a composition of $3m$-maps, appearing in $\Psi^m$.
Conversely, for each composition of $3m$-maps appearing in $\Psi^m$, the image is a solid torus contained in $B_{i,m}$.
The radii of solid tori in $B_{i,m}$ are decreasing exponentially and are converging to zero as $m \to \infty$.

Since $\mathring{b}_{i,m} \subset B_{i,m}$ and $B_{i+1,m} \subset B_{i,m}$ for all $m \in \mathbb{N}$, there is a limit 
$$B_{i,\infty} := \lim_{m\to \infty}B_{i,m} = \cap_{m \in \mathbb{N}} B_{i,m}.$$
Thus, $B_{i,\infty}$ is the union of infinite strands as a subset of $\pi^{-1}(\partial S_i)$ and 
$$\lim_{m \to \infty} \mathring{b}_{i,m} = B_{i,\infty},$$
as a sequence of closed sets in $\pi^{-1}(\partial S_i)$.

\begin{remark}
	\label{rmk limit of braid seq}
	\mbox{}
	\begin{enumerate}
		\item We have constructed a sequence of specific representatives 
		$\{\mathring{b}_{i,m}\}_{m\in\mathbb{N}}$
		such that $$\lim_{m \to \infty} \mathring{b}_{i,m} = B_{i,\infty}.$$ 
		For the purposes of extending the lamination to the singular and regular disks in Sections 4.3 and 4.4, we assume that the limit $B_{i,\infty}$ is a specific closed subset in $\pi^{-1}(\partial S_i)$.
		\item Each strand of $B_{i,\infty}$ corresponds to an infinite sequence $\{f_m\}_{m\in \mathbb{N}}$ such that $f_1 \circ \cdots \circ f_{3m}$ appears in $\Phi^m$ for all $m \in \mathbb{N}$. 
	\end{enumerate}
	\end{remark}

\subsection{Lagrangian lamination on a singular disk}

Let $\psi$ be of generalized Penner type and let $L$ be a Lagrangian submanifold which is carried by $\mathcal{B}_{\psi}$.
In Section 4.2, on each singular disk $S_i$, we gave an inductive description of a sequence $\{b(\psi^m(L),S_i)\}_{m \in \mathbb{N}}$.
There is a limit $B_{i,\infty}$ of the sequence. 
Moreover, the limit $B_{i,\infty}$ depends only on $\psi$ and $B_{i,\infty}$ is independent to $L$.
In this present subsection, we will construct a Lagrangian lamination $\mathcal{L}_i \subset \pi^{-1}(S_i)$ from $B_{i,\infty}$. 
\begin{remark}
	If $\partial S_i$ is contained in the branch locus of $\mathcal{B}_{\psi}^*$, $B_{i,\infty}$ can be divided into two groups, as a braid $b$ was divided into $\bar{b}$ and $\tilde{b}$ in the Step 1 of the proof of Lemma \ref{lem3}.
	We will construct $\mathcal{L}_i$ from $B_{i, \infty} \cap \overline{\pi^{-1}(\mathring{S}_i)}$, which is one of two groups of $B_{i,\infty}$.

	If $\partial S_i$ is not contained in the branch locus of $\mathcal{B}_{\psi}^*$, then $\mathcal{B}_{i,\infty} \subset \overline{\pi^{-1}(\mathring{S}_i)}$.
	In this case, we will construct a Lagrangian lamination from $B_{i,\infty} = B_{i, \infty} \cap \overline{\pi^{-1}(\mathring{S}_i)}$. 
	Thus, we will simply say that the Lagrangian lamination is constructed from $B_{i, \infty} \cap \overline{\pi^{-1}(\mathring{S}_i)}$.
\end{remark}

\begin{lemma}
	\label{lem4}
	Let $\psi$ be of generalized Penner type. 
	For each singular disk $S_i$ of $\mathcal{B}_{\psi}$, there is a Lagrangian lamination $\mathcal{L}_i \subset \overline{\pi^{-1}(\mathring{S}_i)}$, such that
	\begin{enumerate}
		\item $\mathcal{L}_i \cap \pi^{-1}(\partial S_i)$ is the same braid with $B_{i,\infty} \cap \overline{\pi^{-1}(\mathring{S}_i)}$, where $B_{i,\infty}$ is the limit of a braid sequence, which depends only on $\psi$. 
		\item If $L$ is a Lagrangian submanifold of $M$ which is carried by $\mathcal{B}_{\psi}$, then for every $m \in \mathbb{N}$, there is a Lagrangian submanifold $L_m$ which is Hamiltonian isotopic to $\psi^m(L)$ and $L_m \cap \overline{\pi^{-1}(\mathring{S}_i)}$ converges to $\mathcal{L}_i$ as a sequence of closed subsets.
	\end{enumerate} 
\end{lemma}

\begin{proof}
	Let $\psi$ be of generalized Penner type, i.e., $\psi = \delta_1 \circ \cdots \circ \delta_l$, where $\delta_k$ is a Dehn twist $\tau_i$ or $\sigma_j^{-1}$.
	We will use similar notation with the previous subsection, for example, $S_i$ denotes a singular disk of $\mathcal{B}_{\psi}$, $\Psi$ denotes a matrix corresponding to $\psi$,  $\varphi_i : \pi^{-1}(\partial S_i) \stackrel{\sim}{\to} S^{n-1} \times \mathbb{D}^n$ denotes the identification induced from the fixed coordinate chart on $S_i$, and so on. 
	
	In this proof, first, we will construct $\mathcal{L}_i \subset \overline{\pi^{-1}(\mathring{S}_i)}$ satisfying the first condition, i.e., $\mathcal{L}_i \cap \pi^{-1}(\partial S_i) = B_{i, \infty} \cap \overline{\pi(\mathring{S}_i)}$. 
	Then, we will show that the constructed $\mathcal{L}_i$ satisfies the second condition.
		
	\vskip.2in
	\noindent{\em Construction of $\mathcal{L}_i$.}	
	As we mentioned in Remark \ref{rmk limit of braid seq}, a strand of $B_{i,\infty} \cap \overline{\pi(\mathring{S}_i)}$ is identified with an infinite sequence $\{f_m\}_{m \in \mathbb{N}}$ such that $f_1 \circ \cdots \circ f_{lk}$ appears in $\Psi^k$ for all $k \in \mathbb{N}$.
	Note that we are assuming that $\psi= \delta_1 \circ \cdots \circ \delta_l$ for some positive number $l$.
	For each strand $\{f_m\}_{m \in \mathbb{N}}$ of $B_{i,\infty} \cap \overline{\pi^{-1}(\mathring{S}_i)}$, we will construct a Lagrangian submanifold of $\overline{\pi^{-1}(\mathring{S}_i)}$ whose boundary agrees with the strand $\{f_m\}_{m \in \mathbb{N}}$. 
	
	First, for a given strand $\{f_m \}_{m \in \mathbb{N}}$, let us assume that $f_1$ is of trivial type. 
	Then, the strand is identified with a straight curve
	$$\{ (\theta, x_1, \cdots, x_n) \hspace{0.2em} | \hspace{0.2em} \theta \in S^{n-1} \} \subset S^{n-1} \times \mathbb{D}^n \stackrel{\varphi_i}{\simeq} \pi^{-1}(\partial S_i),$$
	where $x_i$ is a constant.
	A subsequence $\{f_m\}_{m \geq 2}$ determines constants $x_i$. 
	Let 
	$$D:=\{(p, x_1, \cdots, x_n) \hspace{0.2em} | \hspace{0.2em} p \in S_i \} \subset \mathbb{D}^n \times \mathbb{D}^n \stackrel{\varphi_i}{\simeq} \overline{\pi^{-1}(\mathring{S}_i)}.$$
	Then, $\varphi_i(D)$ is a Lagrangian disk in $\overline{\pi^{-1}(\mathring{S}_i)}$, whose boundary agrees with the strands $\{ f_m \}_{m \in \mathbb{N}}$.
	
	Second, let us assume that $f_1$ is not of trivial type, but there exists $m \in \mathbb{N}$ such that $f_m$ is of trivial type. 
	Let $k>1$ be the smallest number such that $f_k$ is of trivial type appearing in $\{ f_m \}_{m \in \mathbb{N}}$. 
	Then, $\tilde{\psi} = \delta_{k_0} \circ \cdots \circ \delta_l \circ \delta_1 \circ \cdots \circ \delta_{k_0-1}$, where $k_0 \cong k (\text{mod } l)$, is of generalized Penner type such that $\mathcal{B}_{\tilde{\psi}}$ has a singular disk $\tilde{S}_j$, so that $\tilde{B}_{j,\infty}$, the limit of the braid sequence corresponding to $\tilde{\psi}$ and $\tilde{S}_j$, has a strand identified with $\{f_m\}_{m \geq k}$.
	Thus, there is a Lagrangian disk in $\pi^{-1}(\tilde{S}_j)$ whose boundary agrees with $\{f_m\}_{m \geq k}$.
	Let $D$ denote the Lagrangian disk in $\pi^{-1}(\tilde{S}_j)$.
	Then, there is a connected component of 
	$$\Big((\Phi_{1,1} \circ \delta_1) \circ \cdots \circ (\Phi_{k_0,1} \circ \delta_k)\Big)(D) \cap \overline{\pi^{-1}(\mathring{S}_i)},$$
	whose boundary is $\{ f_m \}_{m \in \mathbb{N}}$, where $\Phi_{i,t}$ is a Hamiltonian isotopy mentioned mentioned in Section 4.1.
	
	To summarize, if there is at least one map of trivial type in $\{f_m\}_{m \in \mathbb{N}}$, then we have a Lagrangian submanifold in $\overline{\pi^{-1}(\mathring{S}_i)}$, whose boundary agrees with $\{ f_m\}_{m \in \mathbb{N}}$.
	Let $\mathcal{L}_{i,\infty}$ be the union of those Lagrangian submanifolds.
	
	Finally, let us assume that for all $m \in \mathbb{N}, f_m$ is not of trivial type. 
	Then, for all $k \in \mathbb{N}$, we will construct a sequence $\{f_m^k\}_{m \in \mathbb{N}}$ for each $k \in \mathbb{N}$, satisfying 
	\begin{enumerate}
		\item $\{f_m^k\}_{m\in\mathbb{N}}$ is a strand of $B_{i,\infty}$,
		\item if $ m \leq kl$, then $f_m^k = f_m$,
		\item there exists a constant $N_k \in \mathbb{N}$ such that $f_{kl + N_k}^k$ is of trivial type.
	\end{enumerate}  

	If there is a sphere having 2 or more plumbing points, there exists a sequence $\{f_m^k\}_{m \in \mathbb{N}}$ for all $k \in \mathbb{N}$.
	This is because of the following:
	
	We note that the finite sequence $\{f_t\}_{1\leq t \leq kl}$ explains a contribution of the braid on a singular disk $S_{i_0}$ on the construction of the braid on a singular disk $S_{j_0}$ when one applies $\psi^k$. 
	In other words, from the view point of Remark \ref{rmk radius of solid torus}, there is a connected component of $\psi^k(\overline{\pi^{-1}(\mathring{S}_i)}) \cap \pi^{-1}(S_{j_0})$ or $\psi^k(\pi^{-1}(\pi(N_p))) \cap \pi^{-1}(S_{j_0})$, where $p$ is the center of $S_{i_0}$ and $N_p$ is the neck at $p$, such that the boundary of the connected component is the image of $f_1 \circ \cdots \circ f_{kl}$. 
	
	If there exists a sphere having 2 or more plumbing points, the Dehn twist along the sphere appears in $\psi$, because of our assumption that every Dehn twist appears in $\psi$. 
	Let $\delta_i$ be the Dehn twist.
	For any plumbing points $p$ and $q$ of the sphere, $\delta_i(\pi^{-1}(\pi(N_p)))$ intersects $\pi^{-1}(S_q^+)$, if the sphere is positive, or $\pi^{-1}(S_q^-)$, otherwise. 
	Thus, there is a map of trivial type in $\Delta_i$, the matrix corresponding to $\delta_i$.
	
	For a sufficiently large $N$, $(\psi^N \circ \delta_1 \circ \cdots \circ \delta_i) (\pi^{-1}(\pi(N_p)))$ intersects $\pi^{-1}(S_{j_0})$.
	We can prove this by observing that $(\psi^N \circ \delta_1 \circ \cdots \circ \delta_{i-1} )(\pi^{-1}(S_q^\pm)) \cap \pi^{-1}(S_{j_0}) \neq \varnothing$ for some sufficiently large $N$. 
	Thus, there is a finite sequence of functions $\{g_j\}_{1 \leq j \leq Nl + i}$ such that $g_j$ is an entry of $\Delta_{j'}$, the matrix corresponding to $\delta_{j'}$, where $j' \cong j (\text{mod } l)$, and the image of $g_1 \circ \cdots \circ g_{Nl+i}$ is identified to the boundary of a connected component of $(\psi^N \circ \delta_1 \circ \cdots \delta_i) (\pi^{-1}(\pi(N_p))) \cap \pi^{-1}(S_{j_0})$. 
	Moreover, we can extend the finite sequence $\{g_j\}_{1 \leq j \leq Nl + i}$ to an infinite sequence $\{g_j\}_{j\in \mathbb{N}}$ such that $\{g_j\}_{j \in \mathbb{N}}$ appears in $B_{i,\infty}$.
	Then, by setting $f_{kl + j}^k = g_j$, we prove the existence of $\{f^k_m\}_{m \in \mathbb{N}}$. 
	
	We obtain a strand $\{f_m^k\}_{k \in \mathbb{N}}$ for each $k \in \mathbb{N}$. 
	These strands converge to $\{f_m\}_{m \in \mathbb{N}}$ as $k \to \infty$. 
	Moreover, by definition of $\mathcal{L}_{i, \infty}$, the boundary of $\mathcal{L}_{i,\infty}$ contains strands $\{f_m^k\}_{m \in \mathbb{N}}$ for all $k \in \mathbb{N}$. 
	Thus, the strand $\{f_m\}_{m \in \mathbb{N}}$ is contained in the boundary of $\mathcal{L}_{i}$, where $\mathcal{L}_{i} = \overline{\mathcal{L}_{i,\infty}}$, the closure of $\mathcal{L}_{i,\infty}$, i.e., the closure of $\mathcal{L}_{i,\infty}$.
	
	If there is no sphere with 2 or more plumbing points, then there is only one positive and one negative sphere intersecting at only one point because we are working on a connected plumbing space.
	For the case, we can construct a Lagrangian lamination $\mathcal{L}$ on $M$ by spinning. 
	Then, $\mathcal{L}_{i} :=\mathcal{L} \cap \pi^{-1}(S_i)$ is a Lagrangian lamination which we want to construct.
\begin{remark}
	We note that, if there is no sphere with 2 or more plumbing points, then  
	we can construct $\mathcal{L}$ without using singular and regular disks. 
\end{remark}
	
	\vskip.2in
	{\noindent {\em Convergence to $\mathcal{L}_i$}.} 
	Let $L_m := \psi_H^m(L)$. 
	We defined $\psi_H$ in Step 4 of the proof of Lemma \ref{lem3}.
	We will prove that $L_m \cap \overline{\pi^{-1}(\mathring{S}_i)}$ converges to $\mathcal{L}_i$.
	
	First, we will show that 
	\begin{gather}
	\label{eqn the limit lamination in a singular disk}
	\lim_{m \to \infty} L_m \cap \overline{\pi^{-1}(\mathring{S}_i)} = \lim_{m \to \infty} (\psi_H^m(N(\mathcal{B}_{\psi})) \cap \overline{\pi^{-1}(\mathring{S}_i)}).
	\end{gather}
	Since $\psi_H(N(\mathcal{B}_{\psi})) \subset N(\mathcal{B}_{\psi})$, $$\psi_H^{m+1}(N(\mathcal{B}_{\psi})) \cap \overline{\pi^{-1}(\mathring{S}_i)} \subset \psi_H^m(N(\mathcal{B}_{\psi})) \cap \overline{\pi^{-1}(\mathring{S}_i)}.$$
	Thus, there exists the limit 
	$$\lim_{m \to \infty} (\psi_H^m(N(\mathcal{B}_{\psi})) \cap \overline{\pi^{-1}(\mathring{S}_i)}) = \cap_m (\psi_H^m(N(\mathcal{B}_{\psi})) \cap \overline{\pi^{-1}(\mathring{S}_i)}).$$
	
	If we equip a Riemannian metric $g$ on $M$, then $d_H(\psi_H^m(\mathcal{B}_{\psi}), \psi_H^m(N(\mathcal{B}_{\psi})))$, where $d_H$ is the Hausdorff metric induced from $g$, converges to zero as $m \to \infty$ because of the same reason that $B_{i,m} := \psi_H^{m}(N(\mathcal{B}_{\psi})) \cap \pi^{-1}(\partial S_i) $ converges to an infinite braid $B_{i,\infty}$ in the last part of Section 4.2.
	
	Since for a large $N_0$, $L_{N_0}$ intersects $\pi^{-1}(S_j)$ for any singular disk $S_j$, $L_{m+N_0} \cap \overline{\pi^{-1}(\mathring{S}_i)}$ intersects every connected component of $\psi_H^m(N(\mathcal{B}_{\psi})) \cap \overline{\pi^{-1}(\mathring{S}_i)}$.
	Thus,
	$$0 \leq \lim_{m \to \infty} d_H(L_{m+N_0} \cap \overline{\pi^{-1}(\mathring{S}_i)}, \psi_H^m(N(\mathcal{B}_{\psi})) \leq \lim_{m \to \infty} 2 d_H(\psi_H^m(\mathcal{B}_{\psi}), \psi_H^m(\mathcal{B}_{\psi})) = 0.$$
	This proves Equation \eqref{eqn the limit lamination in a singular disk}.
	Let $\mathbb{L}_i$ be the limit in Equation \eqref{eqn the limit lamination in a singular disk}.
	
	Second, we show that $\mathbb{L}_i$ is $\mathcal{L}_i$.
	By the construction of $\mathcal{L}_i$, we know that 
	$$\mathcal{L}_i \subset \psi_H^m(N(\mathcal{B}_{\psi})) \cap \overline{\pi^{-1}(\mathring{S}_i)} \text{  for every  } m \in \mathbb{N}.$$
	It implies that $\mathcal{L}_i \subset \mathbb{L}_i$.
	Moreover, 
	$$\mathcal{L}_i \cap \pi^{-1}(\partial S_i) = \mathbb{L}_i = B_{i,\infty} \cap \overline{\pi^{-1}(\mathring{S}_i)}.$$ 
	Because every connected component of $\mathbb{L}_i$ has a boundary on $\partial S_i$, this shows $\mathcal{L}_i = \mathbb{L}_i$. 
\end{proof}

\subsection{Lagrangian lamination on a regular disk}
 
In the previous subsection, we constructed Lagrangian laminations on singular disks, when boundary data for singular disks are given.
In the present subsection, first, we will define boundary data for a regular disk. 
Then, second, we will construct Lagrangian laminations on regular disks from the given data.
Finally, we will prove Theorem \ref{lamination thm} as a corollary of Lemmas \ref{lem4} and \ref{lem6}. 

Before defining the boundary data, we remark that, by Remark \ref{rmk identify with cotangent bundle of disk}, $\overline{\pi^{-1}(\mathring{R}_i)}$ is symplectomorphic to $DT^*\mathcal{D}$, where $\mathcal{D}$ is a disk. 

We define a data $c_{j,m}$ on the boundary of a regular disk $R_j$ for $\psi^m(L)$, by setting  
$$c_{j,m} := L_m \cap \pi^{-1}(\partial R_j).$$
We defined $L_m$ in the proof of Lemma \ref{lem4}.
Note that $c_{j,m}$ is a closed subset, not a class of a closed subset.

To obtain a limit of $c_{j,m}$, we consider 
$$C_{j,m}:= \psi_H^m(N(\mathcal{B}_{\psi})) \cap \pi^{-1}(\partial R_j),$$ 
as we did in Section 4.2. 
Since $\psi_H^(N(\mathcal{B}_{\psi})) \subset N(\mathcal{B}_{\psi})$, $C_{j,m+1} \subset C_{j,m}$. 
Moreover, $C_{j,m}$ is the union of solid tori (resp.\ $S^{n-1} \times \mathbb{D}^n$) in $\pi^{-1}(\partial R_j)$ for the case $n=2$ (resp.\ of general $n$). 
If a symplectic manifold $M$ is equipped with a Riemannian metric $g$, we can measure the radii of solid tori in $C_{j,m}$.
The radii decrease exponentially and converge to zero as $m \to \infty$, because of the same reason that radii of solid tori comprising $B_{i,m}$ decrease exponentially and converge to zero as $m \to \infty$ in Section 4.2.
Then, the limit of $c_{j,m}$ is given by 
$$C_{j,\infty} = \lim_{m \to \infty}C_{j,m} = \cap_m C_{j,m}.$$

The next step is to smooth $R_j$. 
A regular disk $R_j$ has corners.
We will replace $R_j$ with a smooth disk $R_j'$.
This is because, at the end, a Lagrangian lamination will be given as graphs of closed sections. 
By smoothing $R_j$, it will be easier to handle closed sections. 

To smooth $R_j$, we subtract a tubular neighborhood $N(\partial R_j) \subset R_j$ from $R_j$. 
Let $R_j' := R_j \setminus N(\partial R_j)$.
Then, $R_j'$ is a smooth disk.
We replace $R_j$ with $R_j'$. 
To finish smoothing, we need to determine boundary data for $R_j'$ from $c_{j,m}$.
 
Each connected component of $c_{j,m}$ can be identified wit a section of a bundle $\pi^{-1}(\partial R_j)$ over $\partial R_j$. 
We can extend this section to a closed section of a bundle $\pi^{-1}(N(\partial R_j))$ over $N(\partial R_j)$ by computations.
Then, the graph of the extended section is a Lagrangian submanifold of $\pi^{-1}(N(\partial R_j))$.
The boundary of the Lagrangian submanifold on $\partial R_j'$ makes up the boundary data for $R_j'$.

From now, we assume that a regular disk $R_j$ is a smoothed disk.
Lemma \ref{lem5} claims that for a given data $c_{j,m}$ on a smoothed regular disk $R_j$, we can construct a Lagrangian submanifold $N_{j,m} \subset \overline{\pi^{-1}(\mathring{R}_i)}$ such that $\partial N_{j,m} = c_{j,m} \cap \overline{\pi^{-1}(\mathring{R}_i)}$. 

\begin{lemma}
	\label{lem5}
	Let $Q$ be a closed subset of $\partial T^*\mathbb{D}^n$ such that there exists a Lagrangian submanifold $L \subset T^*\mathbb{D}^n$ so that $L \cap \partial T^*\mathbb{D}^n = Q$ and $L$ is a union of Lagrangian disks transverse to fibers.
	Then, we can construct a Lagrangian submanifold $L$ uniquely up to Hamiltonian isotopy through Lagrangians transverse to the fibers.
\end{lemma}

To prove Lemma \ref{lem5}, we will use the following:
in Lemma \ref{lem5}, if an identification $\varphi: \partial T^*\mathbb{D}^n \stackrel{\sim}{\to} S^{n-1} \times \mathring{\mathbb{D}}^n$ is induced from a coordinate chart on $\mathbb{D}^n$, $\varphi(Q)$ represent the trivial braid because $L$ is a union of Lagrangian disks.  

\begin{proof}[Proof  of Lemma \ref{lem5}]
	The proof of Lemma \ref{lem5} consists of two parts, the construction of $L$ and the uniqueness of $L$.
	
	\vskip.2in
	\noindent{\em Construction.}
	We start the proof with the simplest case, i.e., when $Q$ is connected.
	In other words, $Q$ represents the braid with only one strand.

	By fixing coordinate charts on $\mathbb{D}^n$, we can write down $Q$ as a section of a disk bundle $\partial T^*\mathbb{D}^n$ over $\partial \mathbb{D}^n$, i.e.,
	$$ Q:= \{ f_1(x_1,\cdots,x_n)dx_1 + \cdots + f_n(x_1,\cdots,x_n)dx_n \hspace{0.2em} | \hspace{0.2em} x_1^2 +  \cdots + x_n^2 =1 \}.$$
	Then, the simplest case is proved by determining a function $\phi: \mathbb{D}^n \to \mathbb{R}$ such that $d \phi = f_1dx_1 + \cdots + f_ndx_n$ on $\partial \mathbb{D}^n$. 
	The graph of $d \phi$ is a Lagrangian submanifold which we would like to find. 
	Note that there are infinitely many $\phi$ satisfying the conditions, but the Hamiltonian isotopy class of the graph of $d\phi$ is unique through Lagrangians transverse to the fibers.
	
	If $Q$ has 2 or more connected components $l_i$, then we can write $l_i$ as a section over $\partial \mathbb{D}^n$. 
	For each $i$, we need to determine functions $\phi_i : \mathbb{D}^n \to \mathbb{R}$ such that $d \phi_i$ agrees with $l_i$ on $\partial \mathbb{D}^n$. 
	Moreover, to avoid self-intersection, we need $d \phi_i \neq d\phi_j$ for all $i \neq j$ everywhere. 
	Then, the union of graphs of $d \phi_i$ on $T^*\mathbb{D}^n$ is a Lagrangian submanifold $L$ which we want to construct.
	
	We discuss with the simplest non-trivial case, i.e., $Q$ has two connected components $l_0$ and $l_1$, and the dimension $2n =4$. 
	Without loss of generality, we assume that $l_0$ is the zero section.
	Then, we can assume that $\phi_0 \equiv 0$. 	
	We only need to determine $\phi_1$ such that $d \phi_1$ does not vanish everywhere. 
	
	We assume that there exists $\phi_1$ satisfying the conditions.
	Then, we will collect combinatorial data from $\phi_1$, and we will construct a function $\tilde{\phi}_1$ satisfying conditions, from the combinatorial data.
	Through this, we will see what combinatorial data we need.
	We will end the construction part by obtaining the combinatorial data from the given $Q$.
	 
	For convenience, we will use the polar coordinates instead of the $(x,y)$-coordinate on $\mathbb{D}^2$.
	Let $r_0$ be a small positive number. 
	We restrict the function $\phi_1$ on $[r_0,1] \times S^1$.
	On $\{1\} \times S^1 = \partial \mathbb{D}^2$ agrees with $l_1$.
	On $\{r_0\} \times S^1, d\phi_1$ is approximately a constant section $a dx + b dy = a( \cos \theta dr - r_0 \sin \theta d \theta) + b(\sin \theta dr +r_0 \cos \theta d \theta)$, where $d \phi_1(0,0) = a dx + b dy$ and $(x,y)$ are the standard coordinate charts of $\mathbb{D}^2$.
	We remark that on $\{r_0\} \times S^1$, the pair of graphs of $d\phi_i|_{\{r_0\}\times S^1}$ represents the trivial braid under the identification induced from the $(x,y)$-coordinates.  
	Then, the pair $(d\phi_0 \equiv 0, d\phi_1)$ implies an isotopy between two representatives of the trivial braid on $[r_0,1] \times S^1$. 
	
	For every $r_* \in [r_0,1]$, we can find all local maxima and minima of a function 
	$$ \theta \mapsto \phi_1(r_*,\theta).$$
	We mark $(r_*, \theta_*)$ as a red (resp.\ blue) point if the above function has a local maxima (resp.\ minima) at $\theta_*$. 
	If $r_*=1$, there are same number of red/blue marked points on $\{1\} \times S^1$, and there are only one red/blue marked point on $\{r_0\} \times S^1$.
	On $[r_0,1] \times S^1$, we have a collection $\mathcal{C}$ of curves shaded red and blue. 
	If a curve in $\mathcal{C}$ is not a circle, then the curve has two end points on the boundary of $[r_0,1] \times S^1$. 
	There are exactly two curves connecting both boundary components of $[r_0,1] \times S^1$, and those two curves have end points of the same color.
	 
	If we write $d \phi_1 = f d\theta + g dr$, then $f$ is zero on curves in $\mathcal{C}$. 
	Since $d \phi_1$ does not vanish, $g$ cannot be zero on the curves. 
	Thus, we can assign the sign of $g$ for each curve.
	Figure \ref{local max and min} is an example of a collection $\mathcal{C}$. 
	\begin{figure}[h]
		\centering
		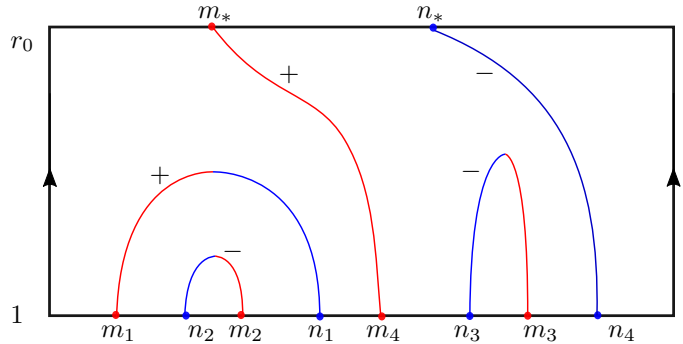
		\caption{Example of a collection $\mathcal{C}$ on $[r_0,1] \times S^1$.}
		\label{local max and min}
	\end{figure}
	
	Conversely, if we have a collection $\mathcal{C}$ of curves such that each curve is shaded red and blue and is equipped with a sign, then we can draw a graph of $\tilde{\phi}_1$ roughly. 
	This is because, the collection $\mathcal{C}$ determines the sign of horizontal directional derivative of $\tilde{\phi}_1$, i.e., $d\tilde{\phi}_1(\partial_{\theta})$ on every point of $[r_0,1] \times S^1$, and vertical directional derivative of $\tilde{\phi}_1$, i.e., $d\tilde{\phi}_1(\partial_r)$ on the curves. 
	From these, one obtains a (rough) graph of $\tilde{\phi}_1$.
	Thus, in order to determine a function $\phi_1$, it is enough to determine a collection $\mathcal{C}$ of curves in $[r_0,1] \times S^1$ from the given $Q$. 
	
	For the given $Q$, we assume that a connected component $l_0$ of $Q$ is the zero section without loss of generality. 
	For the other connected component $l_1$, one has $f_1, g_1:S^1 \to \mathbb{R}$ such that $l_1$ is the graph of $f_1 d\theta + g_1 dr$ on $\{1\} \times S^1 = \partial \mathbb{D}^2$. 
	We know that $Q$ represents the trivial braid with respect to the standard $(x,y)$-coordinate of $\mathbb{D}^2$.
	Thus, there is an isotopy $\Gamma : [r_0,1] \times S^1 \to \mathbb{D}^2$ such that 
	\begin{gather*}
	\Gamma(1, \theta) = (f(\theta), g(\theta)), \hspace{0.5em} \Gamma(r_0,\theta) = (A r_0 \cos \theta, A \sin \theta)\\
	\Gamma(t,\theta) \neq (0,0) \text{  for all  } (t,\theta) \in [r_0,1] \times S^1,
	\end{gather*}
	where $A$ is a constant.
	
	For every $r \in [r_0,1]$, let $\gamma_r (\theta) = \Gamma(r, \theta)$.
	Then, $\gamma_r$ is a closed curve in $\mathbb{D}^2$, for all $r$.
	Moreover, $\Gamma$ is a path connecting $\gamma_1$ and $\gamma_{r_0}$ in the loop space of $\mathring{\mathbb{D}}^2$ without touching the origin. 
	
	We mark $(r,\theta)$ on $[r_0,1] \times S^1$ as a red (resp.\ blue) point if $\gamma_r(\theta)$ intersects $dr$-axis from right to left (resp.\ from left to right).
	These marked points comprise curves in $[r_0,1] \times S^1$, and we have a collection $\mathcal{C}$ of curves, shaded red and blue, in $[r_0,1] \times S^1$.
	We know that $\gamma_1$ has an even number of intersection points.
	When $r$ decreases, there is a series of creations/removes of intersection points, which are given by finger moves along $dr$-axis.
	Each finger move does not touch the origin.
	Thus, for a curve in $\mathcal{C}$, every intersection point composing the curve lies on either the positive $dr$-axis or the negative $dr$-axis.
	Then, we can assign a sign for each curve in $\mathcal{C}$.
	
	Figure \ref{homotopy between loops} is an example of $\Gamma$, corresponding to the case described by Figure \ref{local max and min}. 
	\begin{figure}[h]
		\centering
		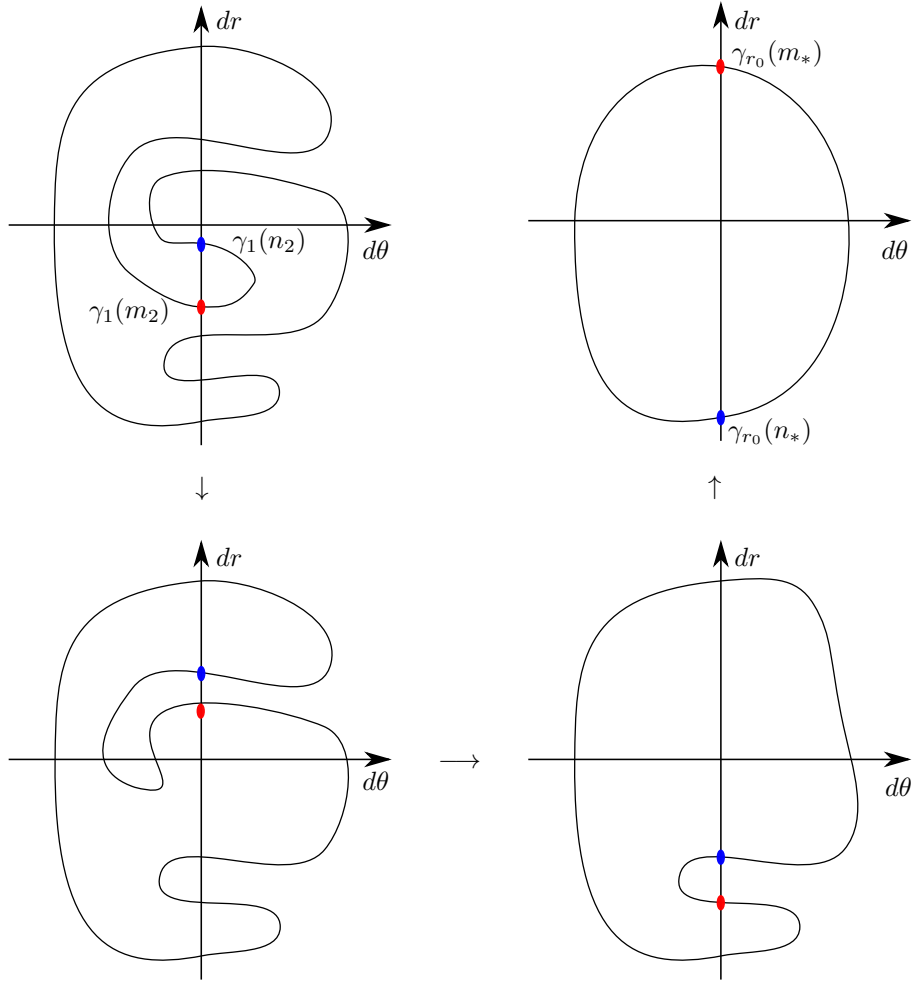
		\caption{Creation of a collection $\mathcal{C}$.}
		\label{homotopy between loops}
	\end{figure}	
	The upper left of Figure \ref{homotopy between loops} is $\gamma_1$ and the upper right is $\gamma_{r_0}$.
	Through the first arrow, we observe a finger move removing two intersection points. 
	Those two intersection points correspond to $m_2$, a local maxima shaded red, and $n_2$, a local minima shaded blue.
	Thus, we obtain a curve connecting $m_2$ and $n_2$ in Figure \ref{local max and min}.
	Moreover, the intersection points lie in the negative part of the $dr$-axis.
	Thus, we assign a negative sign to the curve.
	Similarly, we observe there are finger moves removing intersection points.
	We obtain curves connecting $m_i$ and $n_i$ for $i=1, 2$, and $3$ in Figure \ref{local max and min}. 
	After the finger moves, there are only two intersection points corresponding to $m_*$ and $n_*$, and we obtain curves connecting $m_4$(resp.\ $n_4$) and $m_*$(resp.\ $n_*$). 
	
	We have constructed a collection $\mathcal{C}$ of curves on $[r_0,1] \times S^1$ from an isotopy $\Gamma$.
	Thus, we can obtain a function $\phi_1 : [r_0,1] \times S^1 \to \mathbb{R}$. 
	In order to complete the proof, we need to extend $\phi_1$ into a small disk with radius $r_0$. 
	We have 
	$$\phi_1(x,y) = A r \sin \theta = Ay$$
	on the small disk.
	
	The situation for the general case is analogous.
	If $Q$ has more connected components $l_i$ for $i = 0, \cdots, k$, then we have to determine $\phi_i: \mathbb{D}^2 \to \mathbb{R}$ such that $d\phi_i = l_i$ on $\partial \mathbb{D}^2$, and $d\phi_i \neq d\phi_j$ for all $i \neq j$. 
	We fix an isotopy $\Gamma$, and obtain a collection $\mathcal{C}$ of curves on $[r_0,1] \times S^1$ from $\Gamma$.
	Each curve in $\mathcal{C}$ encodes restrictions on $d\phi_i - d\phi_j$ for some $i$ and $j$.
	More precisely,	$(\phi_i-\phi_j)$ has a local maxima (resp.\ minima) in the horizontal direction, only at a point of a curve shaded red (resp.\ blue), and $(d\phi_i-d\phi_j)(\partial_{r})$ has the sign assigned on the curve. 
	For the case of general dimension $2n$, we obtain combinatorial data from $Q$, i.e., a collection of curves on $[r_0, 1] \times S^{n-1}$ assigned a sign, and construct functions on $\mathbb{D}^n$ from the combinatorial data.
	% 2 dimensional case에서는 각 curve가 d \theta 방향으로의 증가/감소 정보만 encode 하면 되어서 blue/red로 구분, for a general n, each curve has to encode more data, increas/decrease in the direction of $d\theta_1, d\theta_2, \cdots, d\theta_{n-1}$. 
	%curve는 $dr \times d\theta_1 \times d\theta_2 \times \cdots \times d\theat_{n-1}$-cube에서 $dr-$axis 와의 intersection으로 한다. transversality같은거 가정하면 잘 될듯....
	
	\vskip.2in
	\noindent{\em Uniqueness.}
	Recall that the construction consists of three steps. 
	First, we choose an isotopy $\Gamma$ connecting $Q$ and the trivial representative of the trivial braid. 
	Then, we obtained a collection $\mathcal{C}$ of curves from $\Gamma$, such that each curve encodes restrictions on $d \phi_i - d \phi_j$.
	The last step is to construct a set of functions $\{ \phi_i : \mathbb{D}^n \to \mathbb{R}\}$.
	
	The construction depends on choices in the first and last steps.
	More precisely, for the first step, the choice of isotopy $\Gamma$ is not unique. 
	If we choose an isotopy $\Gamma$, then there is a unique collection $\mathcal{C}$. 
	However, a set $\{\phi_i\}$ of functions, which is constructed from the collection $\mathcal{C}$, is not unique. 
	We will show that the Hamiltonian isotopy class of $L$, through Lagrangians transverse to the fibers, is independent to those choices. 
	
	First, we discuss the choice in the third step. 
	Let us assume that we have a collection $\mathcal{C}$ of curves in $[r_0,1] \times S^{n-1}$  and two sets of functions $\{\phi_i \}_i$ and $\{\zeta_i \}_i$ satisfying the restrictions encoded by $\mathcal{C}$.
	Then, by setting $\eta_{i,t} := (1-t)\phi_i + t \zeta_i$, we obtain a family of sets of functions such that every member of the family satisfies the restrictions encoded by $\mathcal{C}$. 
	
	Let $L_t$ be the Lagrangian submanifold corresponding to $\{\eta_{i,t}\}$ for a fixed $t$.
	Then, $L_t$ is a Lagrangian isotopy connecting $L_0$, corresponding to $\{\phi_i \}$, and $L_1$, corresponding to $\{\zeta_i \}$. 
	Since $L_t$ is a disjoint union of Lagrangian disks in $T^*\mathbb{D}^n$,
	$L_0$ and $L_1$ are Hamiltonian isotopic.
	Thus, the Hamiltonian class of $L$ through Lagrangians transverse to the fibers is independent of the choice of functions for the third step of the construction.
	
	Before discussing the choice of the first step, note that a continuous change on a collection $\mathcal{C}$ does not make a change on the Hamiltonian isotopy class.
	More precisely, let $\mathcal{C}_0= \{ \gamma_1, \cdots, \gamma_N \}$ be a collection of curves and let $\{\phi_i\}$ be a set of functions corresponding to $\mathcal{C}_0$. 
	If $\{\gamma_{k,t}\}$ is a continuous family of curves with respect to $t$ such that $\gamma_{k,0} = \gamma_k$ for all $k$, then we can obtain a continuous family $\{\phi_{1,t}, \cdots, \phi_{N,t}\}$ such that $\phi_{i,0} = \phi_i$ and $\{\phi_{1,t}, \cdots, \phi_{N,t}\}$ corresponds to $\mathcal{C}_t := \{\gamma_{1,t}, \cdots, \gamma_{N,t} \}$.
	Then, it is easy to check that the Hamiltonian isotopy class of the union of graphs of $d\phi_{i,t}$ in $T^*\mathbb{D}^n$, through Lagrangians transverse to the fibers, is independent to $t$. 
	
	Finally, we will discuss the choice of $\Gamma$.
	Let $\Gamma_0$ and $\Gamma_1$ be two isotopies obtained from the given $Q$ in the first step. 
	Then, we can understand $\Gamma_0$ and $\Gamma_1$ as paths on the loop space of the configuration space of $\mathring{\mathbb{D}}^n$.
	Since the loop space is simply connected, there is a continuous family $\{\Gamma_t\}_{t \in [0,1]}$ connecting $\gamma_0$ and $\gamma_1$. 
	
	Let $\mathcal{C}_t$ be the collection of curves obtained from $\Gamma_t$ and let $\{ \phi_i \}$ be a set of functions constructed from $\mathcal{C}_0$.
	There is $\{\phi_{i,t}\}$ corresponding to $\mathcal{C}_t$ such that $\phi_{i,0} = \phi_i$. 
	Then, if $L_t$ is the union of graphs of $d\phi_{i,t}$, then the Hamiltonian class of $L_t$ is independent to $t$. 
	This shows the uniqueness of $L$, up to Hamiltonian isotopy, through Lagrangians transverse to the fibers.   
\end{proof}

For a smoothed regular disk $R_j$, there is a sequence of data $c_{j,m}$ for each $m \in \mathbb{N}$.
Then, we can construct a sequence of Lagrangian submanifolds $N_{j,m} \subset \overline{\pi^{-1}(\mathring{R}_j)}$ such that $N_{j,m} \cap \partial \overline{\pi^{-1}(\mathring{R}_j)} = c_{j,m}$. 
The following lemma, Lemma \ref{lem6}, claims that we can construct $N_{j,m}$ wisely, so that $N_{j,m}$ converges to a Lagrangian lamination $\mathcal{N}_j$ as $m$ goes to $\infty$.

\begin{lemma}
	\label{lem6}
	It is possible to construct $N_{j,m} \subset \overline{\pi^{-1}(\mathring{R}_j)}$ so that the sequence $N_{j,m}$ converges to a Lagrangian lamination $\mathcal{N}_j \subset \overline{\pi^{-1}(\mathring{R}_j)}$ as $m \to \infty$. 
\end{lemma}

\begin{proof}
	Let the boundary condition $c_{j,m}$ be the set $\{ l_{1,m}, \cdots, l_{{N_m},m} \}$, where $l_{i,m}$ is a connected component of $c_{j,m}$, or equivalently, $l_{i,m}$ is a strand of the braid represented by $c_{j,m}$.  
	We defined $C_{j,m}$ as a disjoint union of solid tori in $\pi^{-1}(\partial R_j)$ at the beginning of the present subsection. 
	Then, we can divide $c_{j,m}$ into a partition, so that $l_{i,m}$ and $l_{j,m}$ are in the same subset if and only if $l_{i,m}$ and $l_{j,m}$ are in the same solid torus (resp.\ $S^{n-1} \times \mathbb{D}^n$ for a higher dimensional case) in $C_{j,m}$. 
	After that, we randomly choose a connected component $l_{s,m}$ from each subset of the partition. 
	
	By Lemma \ref{lem5}, there is $\phi_{s,m}: R_j \to \mathbb{R}$ such that $d\phi_{s,m} =l_{s,m}$ on $\partial R_j$.
	Then, the graph of $d\phi_{s,m}$ is a Lagrangian disk in $\overline{\pi^{-1}(\mathring{R}_i)}$. 
	We can choose a neighborhood $N(\phi_{s,m})$ of the graph of $d\phi_{s,m}$ in $\overline{\pi^{-1}(\mathring{R}_i)}$, such that $N(\phi_{s,m}) \simeq T^*\mathbb{D}^n$ and $N(\phi_{s,m}) \cap \pi^{-1}(\partial R_j)$ is the torus in $C_{j,m}$ containing $l_{s,m}$.
	Moreover, we can assume that 
	$$d_H(N(\phi_{s,m}), \text{the graph of }d\phi_{s,m}) < 2 r^m,$$ 
	where $d_H$ is the Hausdorff metric induced by a fixed Riemannian metric.
	
	We apply Lemma \ref{lem5} to $\{ l_{t,m+1} \in c_{j,m+1} \hspace{0.2em} | \hspace{0.2em} l_{t,m+1} \subset N(\phi_{s,m}) \}$ in $N(\phi_{s,m}) \simeq T^*\mathbb{D}^n$.
	Then, we can construct $\phi_{t,m+1} : R_j \to \mathbb{R}$ such that $d\phi_{i,m+1} = l_{t,m+1}$ on $\partial R_j$ and the graph of $d\phi_{t,m+1}$ is contained in $N(\phi_{s,m+1})$.
	We repeat this procedure inductively on $m \in \mathbb{N}$.
	
	Let $l$ be a strand of $C_{j,\infty}$.
	Then, there is a sequence $l_{i_m,m} \in c_{j,m}$ such that $l_{i_m,m}$ converges to $l$. 
	If we construct $\phi_{i,m}$ by repeating the above procedure, we know that $$d_H(d\phi_{i_m,m}, d\phi_{i_n,n}) < 4 r^{\max(m,n)}.$$
	Thus, $d\phi_{i_m,m}$ converges.
	Moreover, by assuming that $\phi_{i,m}(p) = 0$ for every $i$ and $m$, where $p$ is a center of $R_j$, $\phi_{i_m,m}$ converges to a function $\phi$.
	The graph of $d\phi$ is a Lagrangian disk in $\overline{\pi^{-1}(\mathring{R}_j)}$ such that whose boundary is $l$, the stand of $C_{j,\infty}$. 
	The union of graphs of $d\phi$ is the Lagrangian lamination $\mathcal{N}_j$ which $N_{j,m}$ converges to.
\end{proof}

\begin{proof}[Proof of Theorem \ref{lamination thm}]
	By Lemma \ref{lem4}, there is a Lagrangian lamination $\mathcal{L}_i$ in $\overline{\pi^{-1}(\mathring{S}_i)}$ and by Lemma \ref{lem6}, there is a Lagrangian lamination $\mathcal{N}_j$ in $\overline{\pi^{-1}(\mathring{R}_j)}$.
	Moreover, every Lagrangian lamination agrees with each other along boundaries, thus we can glue them.
	Then we obtain a Lagrangian lamination $\mathcal{L}$ in $M$. 
\end{proof}

\subsection{A generalization}
In the previous sections, we assumed that $\psi$ is of generalized Penner type. 
In the present subsection, we discuss a symplectic automorphism $\psi : (M,\omega) \to (M,\omega)$, not necessarily to be of generalized Penner type, with some assumptions.

First, we assume that there is a Lagrangian branched submanifold $\mathcal{B}_{\psi}$ such that $\psi(\mathcal{B}_{\psi})$ is (weakly) carried by $\mathcal{B}_{\psi}$. 
The proof of Lemma \ref{lem1} carries over with no change.
Thus, if a Lagrangian submanifold $L$ is (weakly) carried by $\mathcal{B}_{\psi}$, then $\psi(L)$ is carried by $\mathcal{B}_{\psi}$. 

As mentioned in Section 4.1, we assume that $\mathcal{B}_{\psi}^*$ admits a decomposition into a union of finite number of singular disks $S_i \simeq \mathbb{D}^n$ and regular disks $R_j \simeq \mathbb{D}^n$. 

\begin{proof}[Proof of Theorem \ref{generalized theorem}]
	First, we define data on the boundary of each singular and regular disk, in the same way we did for the case of $\psi$ of generalized Penner type.
	Then, on a regular disk $R_j$, the proofs of Lemma \ref{lem5} and Lemma \ref{lem6} carry over with no change. 
	Thus, we can construct a Lagrangian lamination on $\pi^{-1}(R_j)$. 
	
	On a singular disk $S_i$, we define the boundary data in the same way. 
	In other words, the boundary data is defined by the isotopy class of $ \psi^m(L) \cap \pi^{-1}(\partial S_i)$. 
	We also can obtain a matrix $\Psi$, which explains how the sequences of braids are constructed inductively. 
	However, the rest of the proof of Lemma \ref{lem4} does not carry over.
	This is because in the proof of Lemma \ref{lem4}, functions of trivial type have a key role. 
	To use the same proof, we need to show that there are enough functions of trivial type. 
	However, the assumptions cannot imply the existence of enough functions of trivial type. 
	
	For a singular disk $S_i$, let $\{f_m\}_{m \in \mathbb{N}}$ be a strand of the limit braid on $S_i$. 
	We note that each strand can be identified to an infinite sequence of functions.
	We forget specific functions $f_m$, but remember their types.
	Then, we obtain a sequence of types. 
	The sequence of types determines the ``shape'' of strand, for example, how many times the strand is rotated.
	
	We can construct a symplectomorphism $\phi$ which is of generalized Penner type such that $\mathcal{B}_{\phi}$ has a singular disk $S$ so that the limit braid assigned on $S$ has a strand of the same shape. 
	In Section 4.3, we constructed a Lagrangian submanifold $L_0 \subset \overline{\pi^{-1}(\mathring{S})}$ such that $\partial L_0$ is the strand. 
	Since $\overline{\pi^{-1}(\mathring{S})} \simeq \overline{\pi^{-1}(\mathring{S}_i)}$, we assume that $L_0$ is a Lagrangian submanifold in $\overline{\pi^{-1}(\mathring{S}_i)}$ and $\partial L_0$ has the same shape to the strand which we choose. 
	By scaling and translating $L_0$ inside $\overline{\pi^{-1}(\mathring{S}_i)}$, we obtain a Lagrangian submanifold whose boundary agrees with the strand. 
		
	The rest of the proof is the same as the proof of Theorem \ref{lamination thm}.
\end{proof}

\section{Application on the Lagrangian Floer homology}
\label{section pseudo-Anosov functors}
In this section, we will give an application of the previous sections on Lagrangian Floer homology.
More precisely, we will prove Theorem \ref{thm Lagrangian floer homology} and give an example in Section \ref{subsection pA functors - lemma}.

\subsection{Setting}
\label{subsection pA functors - setting}
In the present subsection, we will explain terminology in Theorem \ref{thm Lagrangian floer homology}.

In Section \ref{section pseudo-Anosov functors}, we assume that our symplectic manifold $M$ is a plumbing space $M = P(\alpha_1, \cdots, \alpha_m, \beta_1, \cdots, \beta_l)$ of Penner type defined as follows:
\begin{definition}
	\label{def plumbing space of Penner type}
	A plumbing space $M =P(\alpha_1, \cdots, \alpha_m, \beta_1, \cdots, \beta_l)$ is of {\em Penner type} if $\alpha_i$ and $\beta_j$ satisfy
	\begin{enumerate}
		\item $\alpha_1, \cdots, \alpha_m$ and $\beta_1, \cdots, \beta_l$ are $n$-dimensional spheres,
		\item $\alpha_i \cap \alpha_j = \varnothing$, and $\beta_i \cap \beta_j =\varnothing$, for all $i \neq j$.
	\end{enumerate} 
\end{definition}
Note that $P(\alpha_1, \cdots, \alpha_m, \beta_1, \cdots, \beta_l)$ is defined in Section 2.1.

From now on, we will define an involution $\eta : M \stackrel{\sim}{\to} M$, which is associated to $M$. 
\vskip0.2in

\noindent{\em Involution $\eta_0$ on $T^*S^n$ :}
First, we will define an involution $\eta_0$ on $T^*S^n$. 
Let 
\begin{gather*}
S^n = \{ x \in \mathbb{R}^{n+1} \hspace{0.2em} | \hspace{0.2em} |x|=1 \}, \\
T^*S^n = \{ (x,y) \in S^n \times \mathbb{R}^{n+1} \hspace{0.2em} | \hspace{0.2em} x \in S^n, <x,y>= 0 \}.
\end{gather*}
Then, we define $\eta_0 : T^*S^n \stackrel{\sim}{\to} T^*S^n$ as follow:
\begin{gather*}
\eta_0(x_1, \cdots, x_{n+1}, y_1, \cdots, y_{n+1}) = (x_1, x_2, -x_3, \cdots, -x_{n+1}, y_1, y_2, -y_3, \cdots, -y_{n+1}).
\end{gather*}

Let 
\begin{gather*}
S =\{ (\cos \theta, \sin \theta, 0, \cdots, 0) \hspace{0.2em} | \hspace{0.2em} \theta \in [0, 2\pi] \} \subset S^n, \\
T^*S = \{ (\cos \theta, \sin \theta, 0, \cdots, 0, -r \sin \theta, r \cos \theta, 0, \cdots, 0) \hspace{0.2em} | \hspace{0.2em} \theta \in [0, 2\pi], r \in \mathbb{R} \} \subset T^*S^n. 
\end{gather*}
Then, it is easy to check that $T^*S$ is the set of fixed points of $\eta_0$, equivalently, $\eta^{fixed}_0 = T^*S$.
\vskip0.2in

\noindent{\em Involution $\eta$ associated to $M$ :}
First, we will construct an involution $\eta_{\alpha_i}$ and $\eta_{\beta_j}$ on $T^*\alpha_i$ and $T^*\beta_j$ for every $i$ and $j$.
Note that $T^*\alpha_i, T^*\beta_j \subset M$.  

For each $\alpha_i$, we will choose an embedded circle $S_{\alpha_i} \subset \alpha_i$ such that $S_{\alpha_i}$ contains every plumbing point of $\alpha_i$. 
Then, there is a symplectic isomorphism $\phi_{\alpha_i} : T^*S^n \stackrel{\sim}{\to} T^*\alpha_i$ such that $\phi_{\alpha_i}(S^n)= \alpha_i$ and $\phi_{\alpha_i}(S) = S_{\alpha_i}$. 
One obtains an involution $\eta_{\alpha_i} : T^*\alpha_i \stackrel{\sim}{\to} T^*\alpha_i$ by setting 
\begin{gather*}
\eta_{\alpha_i} := \phi_{\alpha_i} \circ \eta_0 \circ (\phi_{\alpha_i})^{-1}.
\end{gather*}
Similarly, one obtains an involution $\eta_{\beta_j} : T^*\beta_j \stackrel{\sim}{\to} T^*\beta_j$ in the same way.

Without loss of generality, one can assume that $\eta_{\alpha_i}(x) = \eta_{\beta_j}(x)$ for every $x \in T^*\alpha_i \cap T^*\beta_j$.
Finally, the involution $\eta : M \stackrel{\sim}{\to} M$ is defined as follows:
\begin{gather*}
\eta(x) := 
\left\{\begin{matrix}
\eta_{\alpha_i}(x) \text{  if  } x \in T^*\alpha_i,\\
\eta_{\beta_j}(x) \text{  if  } x \in T^*\beta_j. 
\end{matrix}\right.
\end{gather*}
We will call $\eta$ {\em the involution associated to $M$}.

\begin{remark}
	\label{rmk properties of involution}
	Let $\tilde{M}$ be the set of fixed points of $\eta$, i.e., $\tilde{M} = \{ x \in M \hspace{0.2em} | \hspace{0.2em} \eta(x) = x \}$.
	It is easy to check that $\tilde{M}$ is a $2$--dimensional symplectic submanifold of $M$. 
	Moreover, $\tilde{M}$ is symplectomorphic to a plumbing space $P(S_{\alpha_1}, \cdots, S_{\alpha_m}, S_{\beta_1}, \cdots, S_{\beta_l})$ of Penner type. 
	Note that $S_{\alpha_i}$ and $S_{\beta_j}$ are embedded circles in $\alpha_i$ and $\beta_j$. 
	
	We call $\tilde{M}$ {\em the fixed surface of $M$}. 
\end{remark}

\subsection{Proof of Theorem \ref{thm Lagrangian floer homology}}
\label{subsection pA functors - proof}

Let $M$ be a plumbing space of Penner type.
Let $\eta$ be the associated involution of $M$.
Let $L_0$ and $L_1$ be a transversal pair of Lagrangian submanifolds such that 
\begin{enumerate}
	\item $\eta(L_i) = L_i$.
	\item Let $\tilde{L}_i = L_i \cap M_i$. Then, $\tilde{L}_i$ is a Lagrangian submanifold of $\tilde{M}$. 
	\item $L_0 \cap L_1 = \tilde{L}_0 \cap \tilde{L}_1$. 
	\item $L_0$ and $L_1$ are not isotopic to each other.
\end{enumerate}
We will compute $\mathbb{Z}/2$--graded Lagrangian Floer homology $HF^*(L_0,L_1)$ over the Novikov field $\Lambda$ of characteristic 2. 
To do this, we will prove that chain complexes $CF^*(L_0,L_1)$ and $CF^*(\tilde{L}_0,\tilde{L}_1)$ are the same chain complexes. 
More precisely, we will show that those two chain complexes have the same generators and the same differential maps. 

First, it is easy to show that $CF^*(L_0,L_1)$ and $CF^*(\tilde{L}_0,\tilde{L}_1)$ have the same generators since $L_0$ and $L_1$ satisfy that $L_0 \cap L_1 = \tilde{L}_0 \cap \tilde{L}_1$.
Thus, $CF^*(L_0,L_1) = CF^*(\tilde{L}_0,\tilde{L}_1)$ as vector spaces.

Second, let $\partial$ (resp.\ $\tilde{\partial}$) denote the differential map on $CF^*(L_0,L_1)$ (resp.\ $CF^*(\tilde{L}_0,\tilde{L}_1)$). 
Then, 
\begin{gather*}
\partial (p) = \sum_{\substack{q \in L_0 \cap L_1 \\ [u]:ind([u])=1}}(\#\mathcal{M}(p,q;[u],J))T^{\omega([u])}q,
\end{gather*}
where $J$ is an almost complex structure on $M$, $u$ is a holomorphic strip connecting $p$ and $q$, and $\mathcal{M}(p,q;[u],J)$ is the moduli space of holomorphic strips.
We skip the foundational details of the definition of $\partial$.

One can easily check that $\eta \circ u$ is also another holomorphic strip connecting $p$ and $q$. 
Let assume that for a holomorphic strip $u$, the image of $u$ is not contained in $\tilde{M}$. 
Then, $u$ and $\eta \circ u$ will be canceled together in $\partial(p)$, since the Novikov field $\Lambda$ is of characteristic 2.
Thus, in order to define the differential map $\partial$, it is enough to count holomorphic strips $u$ such that the image of $u$ is contained in $\tilde{M}$. 

On the other hands, in order to define $\tilde{\partial} : CF^*(\tilde{L}_0,\tilde{L}_1) \to CF^*(\tilde{L}_0,\tilde{L}_1)$, one needs to count the holomorphic strips on $\tilde{M}$. 
Thus, $\partial(p) = \tilde{\partial}(p)$ for all $p \in L_0 \cap L_1 = \tilde{L}_0 \cap \tilde{L}_1$. 

Under the assumptions, $HF^*(L_0,L_1)= HF^*(\tilde{L}_0, \tilde{L}_1)$. 
Note that the former is defined on $M^{2n}$, but the latter is defined on a surface $\tilde{M}$.
Then, Lemma 2.18 of \cite{MR3289326} completes the proof. 
\qed

\subsection{Example \ref{exmp of homology theorem}}
\label{subsection pA functors - lemma}
In the present subsection, we will prove Lemmas \ref{lemma transversal intersection of Lagrangian branched submanifolds} and \ref{lemma carried by} in order to slightly weaken the difficulty of applying Theorem \ref{thm Lagrangian floer homology}.
Then, we will give the Example \ref{exmp of homology theorem}.

Before giving the statement of Lemmas \ref{lemma transversal intersection of Lagrangian branched submanifolds} and \ref{lemma carried by}, we will establish notation.
In Section \ref{section pseudo-Anosov functors}, $M = P(\alpha_1, \cdots, \alpha_m, \beta_1, \cdots, \beta_l)$ is a plumbing space of Penner type. 
Then, as we did in Section 3.4, we can constructed a set $\mathbb{B}$ of Lagrangian branched submanifolds of $M$.

Every Lagrangian branched submanifold $\mathcal{B} \in \mathbb{B}$ is a union of (parts of) $\alpha_i$ and $\beta_j$ and Lagrangian connected sums $\alpha_i$ and $\beta_j$. 
However, there are two possible Lagrangian connect sums of $\alpha_i$ and $\beta_j$ at each plumbing point $p \in \alpha_i \cap \beta_j$.
They are $\alpha_i \#_p \beta_j$ and $\beta_j \#_p \alpha_i$.
By assuming that $\alpha_i$ is a positive sphere and $\beta_j$ is a negative sphere, one considers the Lagrangian connected sum $\beta_j \#_p \alpha_i$, not $\alpha_i \#_p \beta_j$. 
Similarly, by assuming that $\alpha_i$ is negative and $\beta_j$ is positive, one can construct another set $\mathbb{B}^{op}$ of Lagrangian branched submanifolds.

\begin{lemma}
	\label{lemma transversal intersection of Lagrangian branched submanifolds}
	Let $\mathcal{B}_1, \mathcal{B}_2 \in \mathbb{B} \cup \mathbb{B}^{op}$. 
	Then, there is a Hamiltonian isotopy $\Phi_t : M \to M$ such that 
	\begin{enumerate}
		\item $\Phi_t \circ \eta = \eta \circ \Phi_t$, 
		\item $\mathcal{B}_0 \pitchfork \Phi_1(\mathcal{B}_1)$,
		\item for every $q \in \mathcal{B}_0 \cap \Phi_1(\mathcal{B}_1)$, $q$ is not a plumbing point or the antipodal point of a plumbing point.
	\end{enumerate}
\end{lemma}

\begin{proof}
	Since $\mathcal{B}_1$ is a union of (parts of) compact cores and their Lagrangian connected sums, we will construct Hailtonian isotopies perturbing each compact cores $\alpha_i$ and $\beta_j$. 
	Then, one obtains a perturbation of $\mathcal{B}_1$ as a union of (parts of) perturbations of  $\alpha_i$, $\beta_j$ and their Lagrangian connected sums.  
	
	First, we choose a smooth function $f_i: \alpha_i \to \mathbb{R}$ with isolated critical points such that 
	\begin{enumerate}
		\item every plumbing point $p \in \alpha_i$, $f_i(p)=f_i(-p)=0$, where $-p$ is the antipodal point of $p$ on $\alpha_i$, 
		\item every critical point $q$ of $f_i$ lies on $S_{\alpha_i}$ and $q \neq p, -p$ for any plumbing point $p \in \alpha_i$,
		\item $|df_i(x)| < \epsilon$ for all $x \in \alpha_i$ and for a sufficiently small fixed positive number $\epsilon$,
		\item $f_i \circ \eta_{\alpha_i} = f_i$, where $\eta_{\alpha_i}$ is the involution on $T^*\alpha_i$ defined in Section \ref{subsection pA functors - setting}.  
	\end{enumerate}
	We remark that
	\begin{gather*}
	T^*\alpha_i \stackrel{\phi_{\alpha_i}}{\simeq} T^*S^n = \{(x,y) \in \mathbb{R}^{n+1} \times \mathbb{R}^{n+1} \hspace{0.2em} | \hspace{0.2em} |x| =1, <x,y>=0 \},
	\end{gather*}	
	where $\phi_{\alpha_i}: T^*S^n \stackrel{\sim}{\to} T^*\alpha_i$ is the identification which we used in Section \ref{subsection pA functors - setting}.
	Also, we remark that in (3), $|df_i(x)|$ is given by the standard metric on $\mathbb{R}^{2n+2}$.	
	
	Then, we can extend $f_i$ to $\tilde{f}_i : T^*\alpha_i \to \mathbb{R}$ as follows.
	Let $\delta: [0,\infty) \to \mathbb{R}$ be a smooth decreasing function such that 
	$$\delta([0,\epsilon]) = 1, \delta([2\epsilon, \infty)) = 0.$$
	We set 
	$$\tilde{f}_i : T^*\alpha_i \to \mathbb{R}, \tilde{f}_i(x,y) = \delta(|y|) f_i(x).$$
	Similarly, we can get $\tilde{g}_j : T^*\beta_j \to \mathbb{R}$ in the same way.
	
	These Hamiltonian functions $\tilde{f}_i$ and $\tilde{g}_j$ induce Hamiltonian isotopies on $T^*\alpha_i$ and $T^*\beta_j$. 
	Moreover, these Hamiltonian isotopies could be extended on the plumbing space $M$ since the Hamiltonian isotopies have compact supports on $T^*\alpha_i$ and $T^*\beta_j$.
	
	Let $\Phi_{\alpha_i,t}:M \stackrel{\sim}{\to} M$ be the (extended) Hamiltonian isotopy associated to $\tilde{f}_i$. 
	Then, it is easy to check that 
	\begin{gather*}
	\Phi_{\alpha_i,t} \circ \eta = \eta \circ \Phi_{\alpha_i,t}, \\
	\Phi_{\alpha_i,t}(\alpha_k) = \alpha_k, \text{  if  } k \neq i, \\
	\Phi_{\alpha_i,t}(\beta_j) = \beta_j \text{  for all  } j, \\
	\Phi_{\alpha_i,1}(\alpha_i) = \Gamma(d f_i),
	\end{gather*}
	where $\Gamma(d f_i)$ is the graph of $d f_i$ in $T^*\alpha_i \subset M$. 
	Similarly, one can obtain a Hamiltonian isotopy $\Phi_{\beta_j,t}:M \stackrel{\sim}{\to} M$ for each $\beta_j$ in the same way. 
	
	Let 
	$$\Phi_t = \prod_{\beta_j} \Phi_{\beta_j,t} \circ \prod_{\alpha_i} \Phi_{\alpha_i, t}.$$
	Then, it is easy to check that $\Phi_t$ satisfies the first condition of Lemma \ref{lemma transversal intersection of Lagrangian branched submanifolds}.
	Moreover, one can assume that $\Phi_1(\mathcal{B}_1)$ is constructed from $\Phi_1(\alpha_i)$ and $\Phi_1(\beta_j)$.
	Thus, it is easy to prove that $\mathcal{B}_0$ and $\Phi_1(\mathcal{B}_1)$ satisfy the second and the last conditions of Lemma \ref{lemma transversal intersection of Lagrangian branched submanifolds}.
\end{proof}

From now on, we will explain how Lemma \ref{lemma transversal intersection of Lagrangian branched submanifolds} weakens a difficulty of applying Theorem \ref{thm Lagrangian floer homology}. 
The difficulty we will consider is the last condition of Theorem \ref{thm Lagrangian floer homology}, i.e., $L_0 \cap L_1 = \tilde{L}_0 \cap \tilde{L}_1$

Let assume that $L_0$ (resp.\ $L_1$) is a Lagrangian submanifold which is carried by $\mathcal{B}_0$ (resp,\ $\mathcal{B}_1$) $\in \mathbb{B} \cup \mathbb{B}^{op}$. 
Note that $\Phi_1(L_1)$ is carried by $\Phi_1(\mathcal{B}_1)$, where $\Phi_1$ is the Hamiltonian isotopy constructed in Lemma \ref{lemma transversal intersection of Lagrangian branched submanifolds}.
We will count the numbers of intersections $L_0 \cap \Phi_1(L_1)$ and $\tilde{L}_0 \cap \Phi_1(\tilde{L}_1)$. 
If these numbers are the same, then $L_0 \cap \Phi_1(L_1) = \tilde{L}_0 \cap \Phi_1(\tilde{L}_1)$.

First, we remark that $\tilde{L}_0$ (resp.\ $\Phi_1(\tilde{L}_1)$) is a curve which is carried by a train track $\mathcal{B}_0 \cap \tilde{M}$ (resp.\ $\Phi_1(\mathcal{B}_1) \cap \tilde{M}$). 
Then, $\tilde{L}_0$ (resp.\ $\Phi_1(\tilde{L}_1)$) has weights on the train track $\mathcal{B}_0 \cap \tilde{M}$ (resp.\ $\Phi_1(\mathcal{B}_1) \cap \tilde{M}$).
Moreover, the number of $\tilde{L}_0 \cap \Phi_1(\tilde{L}_1)$ is the following:
$$\sum_{x \in \mathcal{B}_0 \cap \Phi_1(\mathcal{B}_1)} (\text{the weight of $\tilde{L}_0$ at } x) \cdot (\text{the weight of $\Phi_1(\tilde{L}_1)$ at  } x).$$ 

To count the number of $L_0 \cap \Phi_1(L_1)$, we can assume that $L_0 \cap \Phi_1(L_1)$ is contained in a small neighborhood of  $\mathcal{B}_0 \cap \Phi_1(\mathcal{B}_1)$.
Since $L_0$ is carried by $\mathcal{B}_0$, not strongly carried by, $L_0$ can have singular points. 
However, the singular points are lying near plumbing points or the antipodal of plumbing points.
Since the intersection points of $\mathcal{B}_0$ and $\Phi_1(\mathcal{B}_1)$ are not plumbing points or their antipodals, every $x \in L_0 \cap \Phi_1(L_1)$ is a regular point of $L_0$ (resp.\ $\Phi_1(L_1)$). 
It means that the number $|L_0 \cap \Phi_1(L_1)|$ is also give by
$$\sum_{x \in \mathcal{B}_0 \cap \Phi_1(\mathcal{B}_1)} (\text{the weight of $\tilde{L}_0$ at } x) \cdot (\text{the weight of $\Phi_1(\tilde{L}_1)$ at  } x).$$
Thus, $|L_0 \cap \Phi_1(L_1)| = |\tilde{L}_0 \cap \Phi_1(\tilde{L}_1)|$.

\begin{lemma}
	\label{lemma carried by}
	Let $L_0$ and $L_1$ be carried by $\mathcal{B}_0, \mathcal{B}_1 \in \mathbb{B} \cup \mathbb{B}^{op}$.
	Then, there is a Hamiltonian isotopy $\Phi_t$ such that 
	$$L_0 \cap \Phi_1(L_1) = \tilde{L}_0 \cap \Phi_1(\tilde{L}_1).$$
\end{lemma}
Thus, if $L_0$ and $L_1$ are carried by $\mathcal{B}_0, \mathcal{B}_1 \in \mathbb{B} \cup \mathbb{B}^{op}$, and if $L_0$ and $L_1$ satisfy conditions (1), (2), and (4) of Theorem \ref{thm Lagrangian floer homology}, then one can apply Theorem \ref{thm Lagrangian floer homology} for $L_0$ and $\Phi_1(L_1)$.

\begin{exmp}
	\label{exmp of homology theorem}
	Let $\psi_0$ and $\psi_1$ be symplectomorphisms of Penner type, i.e., $\psi_0$ and $\psi_1$ are products of positive (resp.\ negative) powers of $\tau_i$ and negative (resp.\ positive) powers of $\sigma_j$, where $\tau_i$ and $\sigma_j$ are Dehn twists along $\alpha_i$ and $\beta_j$ respectively.
	Let assume that $L_0$ (resp.\ $L_1$) is a Lagrangian submanifold of $M$, which is generated from one of compact cores by applying $\psi_0$ (resp.\ $\psi_1$), i.e.,
	$$L_0 = \psi_0(\alpha_k) \text{  or  } \psi_0(\beta_j), \hspace{0.2em} L_1 = \psi_1(\alpha_k) \text{  or  } \psi_1(\beta_j).$$ 

	Then, $\eta(L_i) = L_i$ since 
	\begin{gather*}
	\eta(\alpha_i) = \alpha_i \text{ for all  } i, \eta(\beta_j)=\beta_j \text{  for all  } j, \\
	\eta \circ \tau_i = \tau_i \circ \eta \text{  for all  } i, \eta \circ \sigma_j = \sigma_j \circ \eta \text{  for all } j. 
	\end{gather*}
	Moreover, $\tilde{L}_i = \psi_i(\tilde{\alpha}_k)$ or $\psi_i(\tilde{\beta_j})$.
	Thus, $\tilde{L}_i$ is a Lagrangian submanifold of $\tilde{M}$.
	Finally, $L_i$ is carried by $\mathcal{B}_{\psi_i}$.
	
	Thus, if $L_0$ and $L_1$ are not isotopic to each other, then one can apply Theorem \ref{thm Lagrangian floer homology}. 
\end{exmp}
\bibliographystyle{abbrv}
\bibliography{reference}
\end{document}